\documentclass[10pt,a4paper,reqno]{amsart}
\usepackage[foot]{amsaddr}

\usepackage{amsmath,amsfonts,amsthm,epsfig,graphicx,mathtools,tikz-cd, amssymb,setspace,enumerate,enumitem}
\usepackage[foot]{amsaddr}
\usepackage{caption}
\usepackage{mathrsfs,color}
\usepackage[margin=1in]{geometry}
\usepackage{mathabx}

\usepackage{mathtools}
\mathtoolsset{showonlyrefs}
\usepackage[page,header]{appendix}
\usepackage{titletoc}

\newcommand{\rd}{\,\mathrm{d}}
\numberwithin{equation}{section}
\newtheorem{theorem}{Theorem}[section]
\newtheorem{lemma}[theorem]{Lemma}
\newtheorem{corollary}[theorem]{Corollary}

\newtheorem{proposition}[theorem]{Proposition}

\newtheorem{remark}[theorem]{Remark}

\usepackage{apptools}
\AtAppendix{\counterwithin{theorem}{section}}

\usepackage{mathtools}

\def\bu{{\bf u}}

\def\bx{{\bf x}}
\def\by{{\bf y}}

\def\cE{\mathcal{E}}
\def\cB{\mathcal{B}}
\def\cM{\mathcal{M}}

\def\cR{\mathcal{R}}
\def\cF{\mathcal{F}}

\def\pv{\textnormal{p.v.\,}}
\def\sgn{\textnormal{sgn}}
\def\supp{\textnormal{supp\,}}

\def\essinf{\textnormal{essinf\,}}

\begin{document}
\title[Global Minimizers for Anisotropic Interaction Energies]{Global Minimizers of a Large class of Anisotropic Attractive-Repulsive Interaction Energies in 2D}

\author{Jos\'e A. Carrillo$^{\dagger}$, Ruiwen Shu$^{\dagger}$}
\date{\today}
\subjclass[2020]{}
\address[$\dagger$]{Mathematical Institute, University of Oxford, Oxford OX2 6GG, UK. Emails: {\tt carrillo@maths.ox.ac.uk, shu@maths.ox.ac.uk}}

\maketitle

\begin{abstract}
We study a large family of Riesz-type singular interaction potentials with anisotropy in two dimensions. Their associated global energy minimizers are given by explicit formulas whose supports are determined by ellipses under certain assumptions. More precisely, by parameterizing the strength of the anisotropic part we characterize the sharp range in which these explicit ellipse-supported configurations are the global minimizers based on linear convexity arguments. Moreover, for certain anisotropic parts, we prove that for large values of the parameter the global minimizer is only given by vertically concentrated measures corresponding to one dimensional minimizers. We also show that these ellipse-supported configurations generically do not collapse to a vertically concentrated measure at the critical value for convexity, leading to an interesting gap of the parameters in between. In this intermediate range, we conclude by infinitesimal concavity that any superlevel set of any local minimizer in a suitable sense does not have interior points. Furthermore, for certain anisotropic parts, their support cannot contain any vertical segment for a restricted range of parameters, and moreover the global minimizers are expected to exhibit a zigzag behavior. All these results hold for the limiting case of the logarithmic repulsive potential, extending and generalizing previous results in the literature. Various examples of anisotropic parts leading to even more complex behavior are numerically explored.
\end{abstract}

\section{Introduction}

In this paper we study the interaction energy functional 
\begin{equation}\label{E}
E[\rho] = \frac{1}{2}\int_{\mathbb{R}^2}\int_{\mathbb{R}^2} W(\bx-\by)\rho(\by)\rd{\by}\rho(\bx)\rd{\bx}
\end{equation}
where $\rho$ is a probability measure on $\mathbb{R}^2$. For notational convenience, we use $\rho(\bx)\rd{\bx}$ to denote the integral with respect to a general Borel measure $\rho$. Here, $W$ is a two-dimensional \emph{anisotropic} attractive-repulsive interaction potential, given by
\begin{equation}\label{W}
W(\bx) = |\bx|^{-s}\Omega(\theta) + |\bx|^2,\quad 0<s<2,
\end{equation}
and
\begin{equation}\label{Wlogintro}
    W_{\log}(\bx) = -\ln |\bx| + \Omega(\theta) + |\bx|^2\,,
\end{equation}
that can be seen as a limit $s\to 0^+$ of potentials similar to \eqref{W}.
Here $\theta\in S^1 = [-\pi,\pi)$ denotes the angle of $\bx$: $\frac{\bx}{|\bx|}=(\cos\theta,\sin\theta)$, and $\Omega$ is a function defined on $S^1$, modeling the anisotropic effect of the interaction. Throughout this paper, we will assume the angle function $\Omega$ satisfies
\begin{equation*}
    \text{{\bf (H)}: $\Omega$ is smooth (i.e., $\Omega\in C^\infty(S^1)$), strictly positive, and $\Omega(\theta+\pi)=\Omega(\theta)$.}
\end{equation*}

Anisotropic interactions of the form \eqref{Wlogintro} appear in the modelling of edge dislocations of the same sign \cite{MPS, GPPS1} in material science. Their associated energy minimizers were conjectured to be vertical walls of dislocations, conjecture settled in \cite{MRS19}, where the authors proved that the global minimizer, modulo translations, is given by the semi-circle law on the vertical axis.

Isotropic attractive-repulsive interaction energies have been extensively studied in the past decade, in terms of the existence \cite{CCP15,SST15}, uniqueness modulo translations \cite{CV1,CV2,lopes2017uniqueness,shu2021newtonian}, support \cite{BCLR1,BCLR2,BKHUB15,CFP,LM21,davies2021classifying,davies2021classifying2}, regularity \cite{CV1,CV2,SV14,CDM16} and symmetry \cite{carrilloshu21} of the global/local minimizers. Notice that throughout this paper, the uniqueness of minimizer is always discussed up to translation. 

Finding the unique global minimizers of the interaction energy for particular potentials is a classical problem in potential theory \cite{F35,ST97}.
More precisely, for the repulsive logarithmic potential with quadratic confinement $W(\bx)=-\ln |\bx|+|\bx|^2$ in 2D, that is \eqref{Wlogintro} with $\Omega=1$, it is known~\cite{F35} that the unique global minimizer is the characteristic function of a suitable Euclidean ball, while for $W(\bx) = |\bx|^{-s} + |\bx|^2$, that is \eqref{W} with $\Omega=1$ and $0<s<2$, it is known \cite{CV1,CV2} that the unique global minimizer of the associated interaction energy $E$ is given by
\begin{equation}\label{rho2}
    \rho_2(\bx)=C_2(R_2^2-|\bx|^2)_+^{s/2}
\end{equation}
for some positive constants $R_2,C_2$ depending on $s$ (see \eqref{R2C2} for the explicit formula for $R_2,C_2$). Here, the uniqueness of minimizer can be easily obtained from the \emph{linear interpolation convexity} (LIC) property of $W$, as explained below. For an interaction potential $W$, we say it has the LIC property, as defined in \cite{carrilloshu21}, if for any two compactly supported probability measures $\rho_0\ne \rho_1$ with the same center of mass, the energy along their linear interpolation curve $E[(1-t)\rho_0+t\rho_1],\,t \in [0,1]$ is always strictly convex. This is equivalent to say that $\frac{\rd^2}{\rd{t}^2}E[(1-t)\rho_0+t\rho_1] = 2E[\rho_0-\rho_1] > 0$. It is straightforward to see that LIC implies the uniqueness of energy minimizer.

For anisotropic attractive-repulsive interaction potentials, the behavior of energy minimizers becomes much more interesting because radial symmetry of minimizers is no longer expected. Intuitively, the shape of the minimizer should elongate along the direction where $\Omega$ has smaller values. To observe the phase transition phenomena of anisotropic energy minimizers, we consider a family of potentials
\begin{equation}\label{Walpha}
W_\alpha(\bx) = |\bx|^{-s}(1 + \alpha \omega(\theta)) + |\bx|^2,\quad 0<s<2,\quad \alpha \ge 0
\end{equation}
where $\omega$ is an angle function and $\alpha$ is a scaling parameter. We  assume that $\omega$ satisfies
\begin{equation*}
    \text{{\bf (h)}: $\omega$ is smooth, nonnegative, not identically zero, $\omega(\theta+\pi)=\omega(\theta)$, and $\omega(\frac{\pi}{2})=0$,}
\end{equation*}
where in the last condition we fix the angle $\frac{\pi}{2}$ as the direction with smallest $\omega$ value, without loss of generality. In this case the associated energy will be denoted as $E_\alpha$ (when $\omega$ is clear from context). Using this family of potentials, one can study how the energy minimizers behave as $\alpha$ changes.

A series of recent works \cite{MRS19,CMMRSV,CMMRSV2,MMRSV20,MMSRV21-1} study a particular family of anisotropic potentials, corresponding to a limiting case of the potential $W_\alpha$ in \eqref{Walpha} as $s\to 0^+$, see Section \ref{sec_s0}, given by\footnote{In the references, the quadratic part of the potential was taken as $|\bx|^2/2$, but the results on the energy minimizers are the same via a rescaling of the spatial variable.} 
\begin{equation}\label{particular}
    W_{\log,\alpha}(\bx) = -\ln|\bx| + \alpha\omega(\theta) + |\bx|^2,\quad \omega(\theta)= \cos^2\theta,\quad\alpha\ge 0 .
\end{equation} 
As already mentioned above, for the isotropic case $\alpha=0$, the unique energy minimizer is a constant multiple of the characteristic function of a ball. For $\alpha>0$, \cite{MRS19,CMMRSV} show the following behavior of energy minimizers:
\begin{itemize}
    \item If $0< \alpha \le \alpha_L := 1$, then $W_{\log,\alpha}$ has the LIC property, and the energy minimizer is unique, being a constant multiple of the characteristic function of an ellipse. This ellipse is elongated along the $x_2$-axis.  When $\alpha=\alpha_L$, the ellipse degenerates into a vertical segment, meaning that the minimizer becomes 
    $
    \rho_{\textnormal{1D}}(\bx) = \rho_1(x_2)\delta(x_1)$ with $ \rho_1(x)=C(R^2-|x|^2)_+^{1/2}
    $
    for some positive constants $C,R$. Here $\rho_1(x)$ is the unique energy minimizer for the \emph{one-dimensional} interaction potential $-\ln|x|+|x|^2$.

    \item If $\alpha>\alpha_L$, then the  energy minimizer is unique and given by $\rho_{\textnormal{1D}}$.
\end{itemize}
In higher dimensions ($d\ge 3$), a similar potential $W_\alpha(\bx) = -\frac{1}{|\bx|^{d-2}} + \alpha\frac{x_1^2}{|\bx|^d} + |\bx|^2$ was studied in \cite{CMMRSV2}. Similar conclusion is obtained when $W_\alpha$ has the LIC property (i.e., $-1< \alpha \le d-2$), but in the case with larger $\alpha$, the behavior of the energy minimizers remains unknown. Finally, the most recent results in the literature \cite{MMSRV21-2} show that the ellipses are still the global minimizers for a class of interaction potentials of the form \eqref{particular} via a perturbative argument. Notice that in \cite{MMSRV21-2} they do not assume that the perturbation $\omega$ is positive. However, since the perturbation is bounded below, one can add a constant to \eqref{particular} making $\omega$ positive without changing the minimization problem.

\begin{figure}[ht!]
    \centering
    \includegraphics[width=0.92\textwidth]{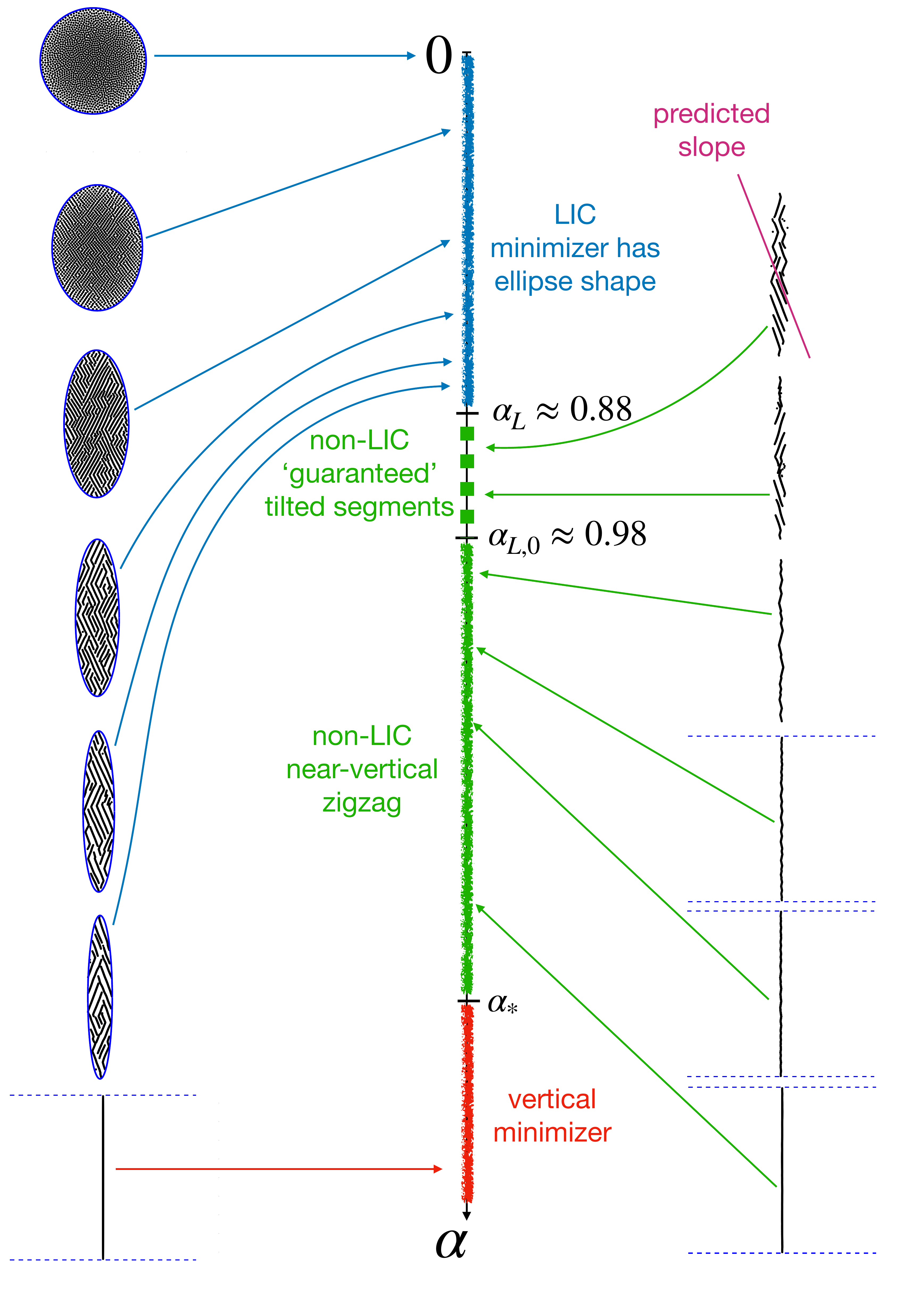}
    \caption{An illustrative example using a particle gradient flow simulation with 1600 particles. The interaction potential is $W_\alpha$ as in \eqref{Walpha} with $s=0.4$ and $\omega(\theta) = \cos^4 \theta + 0.1 \cos^2\theta$ with the $\alpha$ values $0,0.2,0.5,0.7,0.8,0.85,0.9,0.95,1,1.2,1.5,2,4$. The blue ellipses are the predicted shape of the unique minimizer in the LIC cases. The blue dashed lines are the height of $\rho_{\textnormal{1D}}$. The pink line is the predicted slope of the fragmented segments. See Section \ref{sec_example0} for more details.}
    \label{fig:example0}
\end{figure}

In this work we aim to understand the energy minimizers for $W_\alpha$ for the full range of singular repulsive potentials $0< s < 2$ and its limiting case $s=0$. Our main results find large families of anisotropic parts $\omega$ for which we can show the typical behavior illustrated in Figure \ref{fig:example0}, without any smallness assumption as in \cite{MMSRV21-2}. Moreover, we give sharp conditions for some of the critical values in this illustration.

Let us elaborate more our previous statement. In fact, one can see that the typical behavior of the energy minimizers can be much more complicated than the previously considered particular case \eqref{particular}. The existence of compactly supported minimizers for very general anisotropic parts is shown for completeness in Appendix \ref{app_exist}, and thus we may focus on the study of uniqueness and other properties of the minimizers. The main content of the present work will be devoted to the study of the case $0<s<1$. In fact, the limiting case $s=0$ (i.e., logarithmic potentials) exhibits similar phenomena as the $0<s<1$ case, while the case $1\le s<2$ has relatively simple phenomena. They will be studied in Sections \ref{sec_s0} and  \ref{sec_s12} respectively. 

For the $0<s<1$ case, we start by analyzing the LIC property of a general potential $W$ given by \eqref{W}, in Section \ref{sec_LIC}. By calculating its Fourier transform explicitly, we show that the LIC property for $W$ is equivalent to the nonnegativity of $\hat{W}$ away from the origin (Theorem \ref{thm_LICequiv}). For $W_\alpha$ in \eqref{Walpha}, this gives a critical value\footnote{Here $\alpha_L$ depends on $s$ and $\omega$ in \eqref{Walpha}, but we suppress such dependence for notational simplicity. Similar for $\alpha_{L,0}$ and $\alpha_*$ in later context.} $\alpha_L > 0$ (see \eqref{alphaL}), around which the behavior of energy minimizers has a drastic change:
\begin{itemize}
    \item For $0\le \alpha \le \alpha_L$, $W_\alpha$ has the LIC property, and thus there exists a unique energy minimizer.
    \item For $\alpha>\alpha_L$, $W_\alpha$ is not LIC, and we show that $W_\alpha$ is \emph{infinitesimal concave} (Theorem \ref{thm_concave}), a concept introduced in \cite[Section 7]{carrilloshu21}. One expects the energy minimizers to have a lower dimensional support and/or exhibit complicated behavior.
\end{itemize}

We first study the case when $W$ has the LIC property (i.e., $W_\alpha$ with $0\le \alpha \le \alpha_L$), in Section \ref{sec_ell}. To our surprise, although the unique energy minimizer is no longer a characteristic function as in \cite{CMMRSV}, the shape of its support is still an \emph{ellipse}. In fact, the minimizer is given by the push-forward of $\rho_2$ by a linear transformation (Theorem \ref{thm_ell}). The proof is based on a decomposition of the LIC potential $|\bx|^{-s}\Omega(\theta)$ into a positive linear combination of one-dimensional potentials (Corollary \ref{cor_decomp}), which is a consequence of the explicit formula for its Fourier transform. We also notice in Section \ref{sec:gradflow} that the set of such `elliptic' distributions is closed under the flow map of the associated Wasserstein-2 gradient flow (Theorem \ref{thm_ode}). For this class of special solutions of the gradient flow, we show its long time convergence to the minimizer (Theorem \ref{thm_odeconv}). These solutions of the gradient flow are reminiscent of rotating vortex patch solutions in fluid mechanic equations \cite{MO04,HMV13}. 

Then we turn to the study of non-LIC potentials (i.e., $W_\alpha$ with $\alpha>\alpha_L$) in Section \ref{sec:largealpha}. For large $\alpha$, the behavior of the energy minimizers depends on the degeneracy of $\omega(\theta)$ near its minimum point $\theta=\frac{\pi}{2}$:
\begin{itemize}
    \item If $\omega$ is non-degenerate at $\frac{\pi}{2}$ (having a positive second derivative), then 
    \begin{equation}\label{rho1D}
    \rho_{\textnormal{1D}}(\bx) = \rho_1(x_2)\delta(x_1),\quad \rho_1(x)=C_1(R_1^2-|x|^2)_+^{(1+s)/2}
    \end{equation}
    is the unique energy minimizer for sufficiently large $\alpha$ (Theorem \ref{thm_coer}). Here $\rho_1(x)$ is the unique energy minimizer for the one-dimensional interaction potential $|x|^{-s}+|x|^2$, and the positive constants $C_1,R_1$ are given in \eqref{R1C1}. This is similar to the large-$\alpha$ case in \cite{CMMRSV}. Our result is proved by a similar comparison argument as in \cite{CMMRSV}, but requires the design of a special potential (Lemma \ref{lem_coer}) to compare with. Such design is enabled by the tools developed in Section \ref{sec_ell}.
    
    \item If $\omega$ is degenerate at $\frac{\pi}{2}$, then $\rho_{\textnormal{1D}}$ is never a global energy minimizer for any $\alpha$ (Theorem \ref{thm_degen}), not even a Wasserstein-infinity ($d_\infty$)-local minimizer.  However, the width of the support of any global energy minimizer shrinks to zero as $\alpha\rightarrow\infty$ (Theorem \ref{thm_width}). Numerical simulation shows that a typical minimizer in this case is a vertical segment with slight zigzags (see Section \ref{sec_example2}).
\end{itemize}

Finally we study the case when $\alpha$ is slightly greater than $\alpha_L$ in Section \ref{sec_complex}. In fact, if the ellipse-shaped minimizer at $\alpha=\alpha_L$ does not collapse to $\rho_{\textnormal{1D}}$, then neither can happen if $\alpha$ is slightly greater than $\alpha_L$. Due to infinitesimal concavity, it is reasonable to expect the formation of interesting lower-dimensional patterns for the energy minimizers. To understand these structures, we conduct an asymptotic expansion of the generated potential $W*\rho$ around a segment-like piece within $\rho$ (Proposition \ref{prop_expan}), showing that its local stability is determined by the sign of the Fourier transform of $W$ in its perpendicular direction. For some particular examples of $\omega$, this gives a critical value $\alpha_{L,0}>\alpha_L$, such that for any $\alpha\in (\alpha_L,\alpha_{L,0})$, the support of any $d_\infty$-local minimizer cannot contain any `vertical segments'. This guarantees that $\rho_{\textnormal{1D}}$ is not a $d_\infty$-local minimizer, and minimizers are expected to exhibit a zigzag behavior with a collection of tilted slopes.

Finally, Section 7 is devoted to the $s=0$ case, in which we derive the logarithmic potential \eqref{Wlogintro} as a limit of potentials similar to \eqref{W}. This allows us to extend most results for the case $0<s<1$ to the case $s=0$. In Section 8 we study the case $1\le s<2$, in which we show that LIC always holds, and we conclude that the unique energy minimizer is always supported on some ellipse. In Section 9 we illustrate the main results of the paper by numerical simulations of several examples, and explore cases not covered by our theory. Let us finally remark that some preliminary ideas seem to indicate that our techniques are applicable to higher dimensions.

%%%%%%%%%%%%%%%%%%%%%%%%%%%%%%%%%%%%%%%%%%%%%%%%%%

\section{LIC and Fourier transform}\label{sec_LIC}

If $W$ is given by \eqref{W} with $\Omega$ satisfying {\bf (H)}, then it is clear that
\begin{equation*}
\begin{split}
     \text{{\bf (W)}: } &\text{ $W$ is even, locally integrable, lower-semicontinuous,}\\
    & \text{and bounded above and below by positive multiples of $|\bx|^{-s}$ near 0.}
\end{split}\end{equation*}
Then, using the approximation argument in \cite[Lemma 2.5]{carrilloshu21}, one can justify that
\begin{equation}\label{EFT}
    2E[\mu] = \int_{\mathbb{R}^2} \hat{W}_{\textnormal{rep}}(\xi)|\hat{\mu}(\xi)|^2 \rd{\xi}
\end{equation}
for any compactly supported signed measure $\mu$ with $\int_{\mathbb{R}^2}\mu(\bx)\rd{\bx}=\int_{\mathbb{R}^2}\bx\mu(\bx)\rd{\bx}=0$ and $E[|\mu|]<\infty$, where $W_{\textnormal{rep}}(\bx):=|\bx|^{-s}\Omega(\theta)$ is the repulsive part of $W$. Here the quadratic term $|\bx|^2$ makes no contribution to $E[\mu]$ due to the mean-zero properties of $\mu$, as observed in \cite{lopes2017uniqueness}. This motivates our detailed study of the Fourier transform of functions of the form $|\bx|^{-s}\Omega(\theta)$.

\subsection{Fourier transform of the potential}

We first compute the Fourier transform of the potential $W$, ignoring the quadratic part. We will denote the Fourier variable $\xi = |\xi|\vec{e}_\varphi$ (see \eqref{vece} for the definition of $\vec{e}_\varphi$) throughout the paper.

\begin{lemma}\label{lem_FT}
For any smooth $\Omega$ with $\Omega(\theta+\pi)=\Omega(\theta)$, we have
\begin{equation}\label{lem_FT_1}
\cF[|\bx|^{-s}\Omega(\theta)] = |\xi|^{-2+s}\tilde{\Omega}(\varphi),\quad 0<s<2
\end{equation}
with
\begin{equation}\label{lem_FT_3}
\tilde{\Omega}(\varphi) = \tau_{2-s}\int_{-\pi}^{\pi} |\cos(\varphi-\theta)|^{-2+s}\Big(\Omega(\theta)-\Omega(\varphi+\frac{\pi}{2})\Big)\rd{\theta} + c_s\Omega(\varphi+\frac{\pi}{2})\,.
\end{equation}
Here, $\tau_{2-s}$ and $c_s$ are explicitly computable, see Appendix \ref{app:constants}. $\tilde{\Omega}$ is also smooth and even. If $1<s<2$, then the formula can be simplified as
\begin{equation}\label{lem_FT_2}
\tilde{\Omega}(\varphi) = \tau_{2-s}\int_{-\pi}^{\pi} |\cos(\varphi-\theta)|^{-2+s}\Omega(\theta)\rd{\theta},\quad 1<s<2.
\end{equation}

\end{lemma}

\begin{remark}
Since $|\bx|^{-s}\Omega(\theta)$ is even, similar formulas work for the inverse Fourier transform, i.e., 
\begin{equation}
\Omega(\theta) = \tau_{s}\int_{-\pi}^{\pi} |\cos(\theta-\varphi)|^{-s}\Big(\tilde{\Omega}(\varphi)-\tilde{\Omega}(\theta+\frac{\pi}{2})\Big)\rd{\varphi} + c_{2-s}\tilde{\Omega}(\theta+\frac{\pi}{2}),\quad 0<s<2
\end{equation}
and
\begin{equation}\label{lem_FT_2r}
\Omega(\theta) = \tau_s\int_{-\pi}^{\pi} |\cos(\theta-\varphi)|^{-s}\tilde{\Omega}(\varphi)\rd{\varphi},\quad 0<s<1.
\end{equation}
\end{remark}

\begin{proof}

By scaling arguments, it is easy to show that $\cF[|\bx|^{-s}\Omega(\theta)]$ has to take the form \eqref{lem_FT_1} as long as it is a locally integrable function.

We first assume $1<s<2$. Then $\cF[|\bx|^{-s}\Omega(\theta)] \in L^2+L^\infty$ is a locally integrable function since $|\bx|^{-s}\Omega(\theta) \in L^2+L^1$. Also, its Fourier transform is given by  $$\cF[|\bx|^{-s}\Omega(\theta)](\xi) = \int_{\mathbb{R}^2} |\bx|^{-s}\Omega(\theta) e^{-2\pi i \bx\cdot\xi}\rd{\bx}=\lim_{R\rightarrow\infty}\int_{\cB(0;R)} |\bx|^{-s}\Omega(\theta) e^{-2\pi i \bx\cdot\xi}\rd{\bx} ,$$
where the integral is interpreted as an improper integral in the radial direction and $\cB(0;R)$ denotes the ball of radius $R$ centered at 0. Without loss of generality, we consider $\xi=\vec{e}_0$ to obtain
\begin{equation}\begin{split}
\cF[|\bx|^{-s}\Omega(\theta)](\xi) = & \int_{\mathbb{R}^2} |\bx|^{-s}\Omega(\theta) e^{-2\pi i \bx\cdot\xi}\rd{\bx} = \lim_{R\rightarrow\infty}\int_0^R \int_{-\pi}^{\pi}\Omega(\theta) e^{-2\pi i \bx\cdot\xi} \rd{\theta} \, r^{1-s}\rd{r} \\
= & \lim_{R\rightarrow\infty}\int_0^R \int_{-\pi}^{\pi}\Omega(\theta) e^{-2\pi i r \cos\theta} \rd{\theta} \, r^{1-s}\rd{r} \\
= & \lim_{R\rightarrow\infty}\int_0^R \int_{-\pi}^{\pi} \Omega(\theta)\cos(2\pi \cos\theta\cdot r)\rd{\theta}\,  r^{1-s}\rd{r} \\
= & \lim_{R\rightarrow\infty}\int_{-\pi}^{\pi}\int_0^R\cos(2\pi \cos\theta\cdot r)   r^{1-s}\rd{r}\,\Omega(\theta) \rd{\theta} \\
= & \lim_{R\rightarrow\infty}\int_{-\pi}^{\pi}\int_0^{R|2\pi \cos\theta|}  r^{1-s}\cos r   \rd{r} \,|2\pi \cos\theta|^{-2+s}\Omega(\theta) \rd{\theta}   \\
= & \int_{-\pi}^{\pi}\int_0^\infty  r^{1-s}\cos r   \rd{r} |2\pi \cos\theta|^{-2+s}\Omega(\theta) \rd{\theta}   \\
\end{split}\end{equation}
where we used a change of variable $r\mapsto |2\pi \cos\theta|r$ in the second last equality, and the last equality is justified by the dominated convergence theorem since $$
\left|\int_0^{R|2\pi \cos\theta|}r^{1-s}\cos r   \rd{r} |2\pi \cos\theta|^{-2+s}\Omega(\theta)\right| \le C | \cos\theta|^{-2+s}|\Omega(\theta)| \in L^1.
$$ 
The  integral $\int_0^\infty  r^{1-s}\cos r   \rd{r}$ (interpreted as an improper integral) is equal to $\Gamma(2-s)\sin\frac{\pi(s-1)}{2}$ by \eqref{calc3}, which is positive for $1<s<2$.  Therefore \eqref{lem_FT_2} follows in view of \eqref{calc4}. It is clear that \eqref{lem_FT_2} is equivalent to \eqref{lem_FT_3} for $1<s<2$ due to the definitions of $\tau_{2-s}$ and $c_s$ in \eqref{calc1} and \eqref{calc4}.

For the case $0<s\le 1$, we first show that $\cF[|\bx|^{-s}\Omega(\theta)]$ is a locally integrable function. In fact, we fix a smooth radial cutoff function $\psi$ supported on $\{1\le |\bx|\le 4\}$ such that $\psi(\bx)+\psi(\bx/2)=1$ on $\{2\le |\bx|\le 4\}$ analogous to Littlewood-Paley decomposition. Then we decompose $|\bx|^{-s}\Omega(\theta)$ as
\begin{equation}
    |\bx|^{-s}\Omega(\theta) = |\bx|^{-s}\Omega(\theta)\left(1-\sum_{k=0}^\infty \psi(2^{-k}\bx)\right) + \sum_{k=0}^\infty |\bx|^{-s}\Omega(\theta)\psi(2^{-k}\bx)
\end{equation}
where the summation converges in the sense of tempered distributions.

We first notice $\cF[|\bx|^{-s}\Omega(\theta)(1-\sum_{k=0}^\infty \psi(2^{-k}\bx))]\in L^\infty$ since its argument is compactly supported on $\cB(0;2)$ and thus in $L^1$. The summation can be written as
\begin{equation}
    \sum_{k=0}^\infty |\bx|^{-s}\Omega(\theta)\psi(2^{-k}\bx) = \sum_{k=0}^\infty 2^{-ks}g(2^{-k}\bx),\quad g(\bx):=|\bx|^{-s}\Omega(\theta)\psi(\bx)
\end{equation}
Therefore, we have
\begin{equation}
    \cF\Big[\sum_{k=0}^\infty |\bx|^{-s}\Omega(\theta)\psi(2^{-k}\bx)\Big] = \sum_{k=0}^\infty 2^{k(2-s)}\hat{g}(2^k\xi)
\end{equation}
in the sense of tempered distributions. Notice that $\hat{g}\in L^1$ since $g$ is smooth and compactly supported, and we have $\|\hat{g}(2^k\cdot)\|_{L^1} = 2^{-2k}\|\hat{g}\|_{L^1}$. Therefore the above summation also converges in $L^1$. This shows $\cF\Big[\sum_{k=0}^\infty |\bx|^{-s}\Omega(\theta)\psi(2^{-k}\bx)\Big]\in L^1$, and thus $\cF[|\bx|^{-s}\Omega(\theta)]\in L^1+L^\infty$.

In the case $0<s<1$, to prove the formula \eqref{lem_FT_3} at $\xi=\vec{e}_0$, we first assume $\Omega(0)=\Omega(\pi/2)=0$. Then, interpreting as an improper integral in the radial direction,
\begin{equation}\begin{split}
\cF[|\bx|^{-s}\Omega(\theta)](\xi) =  & \lim_{R\rightarrow\infty}  \int_0^R\int_{-\pi}^{\pi}\cos(2\pi r \cos\theta)\,\Omega(\theta)\rd{\theta}\,r^{1-s}\rd{r}   \\
=  & \lim_{R\rightarrow\infty}  \frac{1}{2\pi}\int_0^R\int_{-\pi}^{\pi}\sin(2\pi r \cos\theta)\left(\frac{\Omega(\theta)}{\sin\theta}\right)'\rd{\theta}\,r^{-s}\rd{r}   \\
\end{split}\end{equation}
by integration by parts in $\theta$ and using the fact that $\frac{\Omega(\theta)}{\sin\theta}$ is smooth. Then we interchange the order of integrals and use a change of variable to get
\begin{equation}\begin{split}
\cF[|\bx|^{-s}\Omega(\theta)](\xi) =  & \lim_{R\rightarrow\infty}  \frac{1}{2\pi}\int_{-\pi}^{\pi}\int_0^R\sin(2\pi r \cos\theta)r^{-s}\rd{r}\,\left(\frac{\Omega(\theta)}{\sin\theta}\right)'\rd{\theta}  \\
=  & \lim_{R\rightarrow\infty}  \frac{1}{2\pi}\int_{-\pi}^{\pi}\int_0^{2\pi R|\cos\theta|}\!\!\!\!\!\!\!r^{-s}\sin r \rd{r}\,|2\pi \cos\theta|^{s-1}\sgn(\cos\theta)\left(\frac{\Omega(\theta)}{\sin\theta}\right)'\rd{\theta} . \\
\end{split}\end{equation}
Since $0<s\le 1$, $\int_0^{2\pi R |\cos\theta|}r^{-s}\sin r \rd{r}$ is uniformly bounded in $R$, and we may use the dominated convergence theorem to take $R\rightarrow\infty$ to obtain
\begin{equation}\label{Fsin}\begin{split}
\cF[|\bx|^{-s}\Omega(\theta)](\xi) = &  \frac{1}{2\pi}\int_0^\infty r^{-s}\sin r \rd{r} \int_{-\pi}^{\pi}|2\pi \cos\theta|^{s-1}\sgn(\cos\theta)\left(\frac{\Omega(\theta)}{\sin\theta}\right)'\rd{\theta} . \\
\end{split}\end{equation}
Then we get
\begin{equation}\begin{split}
\cF[|\bx|^{-s}\Omega(\theta)](\xi) = & (2\pi)^{s-2}(s-1)\int_0^\infty r^{-s}\sin r \rd{r} \int_{-\pi}^{\pi} |\cos\theta|^{s-2} \Omega(\theta)\rd{\theta}
\end{split}\end{equation}
where the last integral in $\theta$ is well defined because $\Omega(\theta) = O(|\theta-\frac{\pi}{2}|)$ near $\pi/2$, and the last equality can be justified by cutting off a small interval $[\pi/2-\epsilon,\pi/2+\epsilon]$ and integrating by parts. Noticing that $\int_0^\infty r^{-s}\sin r \rd{r} = \Gamma(1-s)\cos\frac{\pi s}{2}$ by \eqref{calc3s}, we get the \eqref{lem_FT_3} in the case $\Omega(0)=\Omega(\pi/2)=0$ using the definition \eqref{calc4} for $\tau_s$.

The general case of \eqref{lem_FT_3} follows by noticing that it is true for $\Omega=1$ and $\Omega=\cos^2\theta$ with $\xi=\vec{e}_0$. Here, for the case $\Omega=\cos^2\theta$, one can first show that $\tilde{\Omega}(\varphi) = c_s (1-(2-s)\cos^2\varphi)/s$ by using \eqref{lem_FT_2} reversely, and then verify that the resulting $\tilde{\Omega}(0)$ coincides with \eqref{lem_FT_3} using \eqref{calc1} and \eqref{calc4}.

In the case $s=1$, the same calculation up to \eqref{Fsin} works for any $\Omega$ with $\Omega(0)=0$. Then we get
\begin{equation}
    \cF[|\bx|^{-s}\Omega(\theta)](\xi) = (2\pi)^{-1}\int_0^\infty r^{-1}\sin r \rd{r} \cdot 4 \Omega(\frac{\pi}{2}) = \Omega(\frac{\pi}{2})
\end{equation}
which is \eqref{lem_FT_3} (noticing that $\tau_1=0,\,c_1=1$. The general case of \eqref{lem_FT_3} follows by noticing that it is true for $\Omega=1$.

\end{proof}

In most parts of this paper, we focus on the range
$0<s<1$. We notice that \eqref{lem_FT_2r} gives a decomposition of $\Omega(\theta)$ into a linear combination of functions of the form $|\cos(\theta-\varphi)|^{-s}$ for various values of $\varphi$. Notice that $|\bx|^{-s}|\cos(\theta-\varphi)|^{-s}=|\bx\cdot \vec{e}_\varphi|^{-s}$. This gives a decomposition of  $|\bx|^{-s}\Omega(\theta)$ into a linear combination of 1D Riesz potentials along the directions $\vec{e}_\varphi$.
\begin{corollary}\label{cor_decomp}
Let $0<s<1$ and $\Omega$ satisfy {\bf (H)}. Let $\tilde{\Omega}$ be defined by \eqref{lem_FT_1}. Then
\begin{equation}\label{decomp}\begin{split}
|\bx|^{-s}\Omega(\theta)
=  \tau_s\int_{-\pi}^\pi |\bx\cdot \vec{e}_\varphi|^{-s}\tilde{\Omega}(\varphi)\rd{\varphi}.
\end{split}\end{equation} 
\end{corollary}

As a consequence, interaction potentials $W$ of the form \eqref{W} with LIC property are exactly given by positive linear combinations of such rotated 1D potentials.

\subsection{Behavior of LIC/non-LIC potentials}

It was proved in \cite[Theorem 2.4]{carrilloshu21} that the Euler-Lagrange condition for energy minimizers is sufficient for LIC potentials.
\begin{lemma}\label{lem_EL}
Assume $W$ satisfies {\bf (W)}, and $W$ has the LIC property. Assume there exists a compactly supported global energy minimizer (which has to be unique up to translation). Then it is the only probability measure satisfying
\begin{equation}\label{EL}
    (W*\rho)(\bx) \le \essinf(W*\rho),\quad \rho \text{ a.e.}
\end{equation}
up to translation.
\end{lemma}

On the contrary, if $\hat{W}_{\textnormal{rep}}$ is negative somewhere, then the structure of the Fourier transform \eqref{lem_FT_1} enforces the `infinitesimal concave' (in the sense of \cite[Section 7]{carrilloshu21}) property of $W$.

\begin{proposition}\label{thm_concave}
Let $W$ is given by \eqref{W} with $0<s<2$ and $\Omega$ satisfying {\bf (H)}. Assume
$\tilde{\Omega}$ is negative somewhere. Then $W$ is infinitesimal concave, i.e., for any $\epsilon>0$, there exists a function $\mu\in L^\infty(\mathbb{R}^2)$ such that $\int \mu = 0$, $\supp\mu \subset \cB(0;\epsilon)$ and $E[\mu]<0$.
\end{proposition}

Using the conclusion of this theorem, the argument of \cite[Theorem 7.1]{carrilloshu21} applies, showing that any superlevet set of any $d_\infty$-local minimizer does not have interior points.
\begin{proof}
Without loss of generality, assume $\tilde{\Omega}(0) < 0$. Then there exists $\varphi_1>0$ such that 
\begin{equation}\label{omeganeg}
\tilde{\Omega}(\varphi) \le -c_1 < 0,\quad \forall |\varphi| \le \varphi_1.
\end{equation}

Let $0<\epsilon<1$ be fixed. Take a nonnegative smooth radial function $\phi$ supported on $\cB(0;1)$ with $\int \phi = 1$. We first take $R=R(\epsilon) \ge \max\{1,1/\epsilon\}$ to be determined, and define $\mu_1$ by
\begin{equation}
\hat{\mu}_1(\xi) = \frac{1}{2}\Big(\phi\Big(\frac{1}{R}(\xi_1-A,\xi_2)\Big)+\phi\Big(\frac{1}{R}(\xi_1+A,\xi_2)\Big)\Big),\quad A = \frac{R}{\varphi_1}+R.
\end{equation}
Here $\mu_1$ is real and mean-zero because $\hat{\mu}_1$ is even and $\hat{\mu}_1(0)=0$. Also,  
\begin{equation}
\supp\hat{\mu}_1 \subset \cB\big((A,0);R\big)\cup \cB\big((-A,0);R\big) \subset \{(r\cos\varphi,r\sin\varphi): r\ge 0,\,|\varphi| \le\varphi_1\}
\end{equation}
since any point in $\cB((A,0);R)$ has $\xi_1\ge A-R = \frac{R}{\varphi_1}$ and $|\xi_2| \le R$ and thus $|\varphi| \le |\tan\varphi| \le \varphi_1$. Also notice that $\sup_{\xi\in\supp\hat{\mu}_1}|\xi| = \frac{R}{\varphi_1}+2R \le C R$. Therefore, \eqref{lem_FT_1} and \eqref{omeganeg} give
\begin{equation}\label{Emu1}
E[\mu_1] = \frac{1}{2}\int_{\mathbb{R}^2} \hat{W}(\xi) |\hat{\mu}_1(\xi)|^2\rd{\xi} \le -c R^{s},
\end{equation}
where $c$ depends on the constants $c_1$ and $\varphi_1$ fixed for the rest of the proof.
Then we compute
\begin{equation}
\mu_1(\bx) = \frac{R^2}{2}\check{\phi}(R\bx)(e^{2\pi i A x_1}+e^{-2\pi i A x_1}) = R^2\check{\phi}(R\bx)\cos(2\pi  A x_1).
\end{equation}
We define a compactly supported mean-zero $L^\infty$ function
\begin{equation}
    \mu(\bx) = \chi_{\cB(0;\epsilon)}(\bx) \Big(\mu_1(\bx) - \frac{1}{|\cB(0;\epsilon)|}\int_{\cB(0;\epsilon)}\mu_1(\by)\rd{\by}\Big).
\end{equation}

To estimate the difference between $E[\mu_1]$ and $E[\mu]$, we start by taking $m> 0$ to be chosen, and noticing that
\begin{equation}
    |\check{\phi}(\bx)| \le C (1+|\bx|)^{-m},
\end{equation}
where $C$ may depend on $m$, since $\phi$ is smooth and compactly supported. Therefore, we deduce
\begin{equation}\label{mu1m}
    |\mu_1(\bx)| \le C R^2(1+R|\bx|)^{-m}.
\end{equation}
Therefore, using the estimate $|W(\bx)| \le C(|\bx|^{-s}+|\bx|^2)$, we obtain
\begin{equation}\label{Wmu1}\begin{split}
    |(W*\mu_1)(\bx)| \le & \int_{|\bx-\by|<1+|\bx|}|\mu_1(\bx-\by)W(\by)|\rd{\by} + \int_{|\by|\ge 1+|\bx|}|\mu_1(\by)W(\bx-\by)|\rd{\by} \\
    \le & C R^2 \int_{|\bx-\by|<1+|\bx|}\!\!\!\!\!\!\!\!\!\!\!\!\!\!\!\!|W(\by)|\rd{\by}  + C R^2\int_{|\by|\ge 1+|\bx|}\!\!\!\!\!\!\!\!\!\!\!\!(1+R|\by|)^{-m}(|\bx-\by|^{-s}+|\bx-\by|^2)\rd{\by}\\
    \le & C R^2(1+|\bx|)^4  + C R^2\int_{|\by|\ge 1+|\bx|}(1+R|\by|)^{-m}|\bx-\by|^2\rd{\by}\\
    \le & C R^2(1+|\bx|)^4  + C R^2\int_{\mathbb{R}^2}(1+|\bx-\by|)^{-m}|\bx-\by|^2\rd{\by}\\
    \le & C R^2(1+|\bx|)^4 \\
\end{split}\end{equation}
where the third inequality uses $|\bx-\by|\ge 1$ for the second integral, the second last inequality uses $R\ge 1$ and $|\by|\ge |\bx-\by|/2$ for any $|\by|>|\bx|$, and the last inequality holds as long as we take $m>4$. Since $\mu_1$ is mean-zero on $\mathbb{R}^2$, we have
\begin{equation}\begin{split}
    \left|\int_{\cB(0;\epsilon)}\mu_1\rd{\bx}\right|= & \left|\int_{\cB(0;\epsilon)^c}\mu_1\rd{\bx}\right| \le CR^2 \int_{\cB(0;\epsilon)^c} (1+R|\bx|)^{-m} \rd{\bx} \\
    = & C\int_{\cB(0;\epsilon R)^c} (1+|\bx|)^{-m} \rd{\bx} \le C(\epsilon R)^{-m+2}.
\end{split}
\end{equation}
Therefore, by the definition of $\mu$, we obtain
\begin{equation}\label{mumu1}\begin{split}
    |\mu(\bx)-\mu_1(\bx)| 
    \le & |\mu_1(\bx)|\chi_{\cB(0;\epsilon)^c}(\bx) + C\epsilon^{-2}(\epsilon R)^{-m+2}\chi_{\cB(0;\epsilon)}(\bx)
\end{split}\end{equation}
Notice that for any $0<\epsilon<1$,
\begin{equation}
    (W*\chi_{\cB(0;\epsilon)})(\bx) \le C(1+|\bx|^2).
\end{equation}
As a consequence, we deduce
\begin{equation}\begin{split}
    |\big(W*( C\epsilon^{-2}(\epsilon R)^{-m+2}\chi_{\cB(0;\epsilon)})\big)(\bx)| \le C(1+|\bx|^2)\epsilon^{-2}(\epsilon R)^{-m+2} \le C\epsilon^{-2}(1+|\bx|^2).
\end{split}\end{equation}
whenever $m\ge 2$, since we assumed $R\ge 1/\epsilon$. Therefore, combined with \eqref{Wmu1} (which is also true if $\mu_1$ is replaced by $|\mu_1|\chi_{\cB(0;\epsilon)^c}$), we see that
\begin{equation}
    |(W*\mu)(\bx)| \le C R^2\epsilon^{-2} (1+|\bx|)^4
\end{equation}
using $R\ge 1$. Finally, combined with \eqref{Wmu1} and \eqref{mumu1} and using $R\ge 1,\,\epsilon<1$,
\begin{align*}
    |E[\mu]-E[\mu_1]| \le & \frac{1}{2}\int_{\mathbb{R}^2} |(W*\mu)(\bx)(\mu(\bx)-\mu_1(\bx))|\rd{\bx} + \frac{1}{2}\int_{\mathbb{R}^2} |(W*\mu_1)(\bx)(\mu(\bx)-\mu_1(\bx))|\rd{\bx} \\
    \le & C R^2\epsilon^{-2}\int_{\mathbb{R}^2}(1+|\bx|)^4|\mu(\bx)-\mu_1(\bx)|\rd{\bx} \\
    \le & C R^4\epsilon^{-2}\int_{\cB(0;\epsilon)^c}\!\!\!\!(1+|\bx|)^4(1+R|\bx|)^{-m}\rd{\bx} + C R^2\epsilon^{-4}(\epsilon R)^{-m+2}\int_{\cB(0;\epsilon)}\!\!\!\!(1+|\bx|)^4\rd{\bx} \\
    \le & C R^4\epsilon^{-2}\int_{\cB(0;\epsilon)^c}(1+R|\bx|)^{-m+4}\rd{\bx} + C R^2\epsilon^{-4}(\epsilon R)^{-m+2}\epsilon^2 \\
    \le & C R^2\epsilon^{-2}(\epsilon R)^{-m+6} + C R^2\epsilon^{-4}(\epsilon R)^{-m+2}\epsilon^2 \\
    = & C (\epsilon^{-m+4}R^{-m+8} + \epsilon^{-m}R^{-m+4})
\end{align*}
whenever $m>6$. This implies
\begin{equation}
    \frac{|E[\mu]-E[\mu_1]|}{R^s} \le C (\epsilon^{-m+4}R^{-m+8-s} + \epsilon^{-m}R^{-m+4-s})
\end{equation}
We take $m=9$ to guarantee that the above powers of $R$ are negative. Then taking $R$ sufficiently large so that the above RHS is smaller than $c/2$ where $c$ is as in \eqref{Emu1}, we see that $E[\mu]<0$.
\end{proof}

Finally, we establish the following equivalence between the LIC property and the nonnegativity of $\hat{W}_{\textnormal{rep}}$.
\begin{theorem}
\label{thm_LICequiv}
Let $W$ is given by \eqref{W} with $\Omega$ satisfying {\bf (H)}. Then $W$ has the LIC property if and only if $\tilde{\Omega}$ given as in \eqref{lem_FT_1} is nonnegative.
\end{theorem}

Notice that the LIC property is a notion of strict convexity, while the nonnegativity of $\tilde{\Omega}$ is a non-strict inequality. The equivalence of them comes from an argument involving analytic functions, as in the proof below.

\begin{proof}
Suppose $\tilde{\Omega}$ is nonnegative, i.e., $\hat{W}_{\textnormal{rep}}$ is nonnegative. Clearly $\hat{W}_{\textnormal{rep}}$ is not identically zero. Since $\hat{W}_{\textnormal{rep}}$ is continuous away from 0, there exists some $\xi_0\ne 0$ and $0<\epsilon<|\xi_0|$ such that $\hat{W}_{\textnormal{rep}}\ge \epsilon$ on $\cB(\xi_0;\epsilon)$.

We claim that for any compactly supported signed measure $\mu\ne 0$, $\hat{\mu}$ cannot be identically zero in $\cB(\xi_0;\epsilon)$. In fact, $\hat{\mu}(\xi) = \int_{\mathbb{R}^2} e^{-2\pi i \bx\cdot\xi} \mu(\bx)\rd{\bx}$ is an analytic function in $\xi\in \mathbb{C}^2$, and having $\hat{\mu}$ identically zero in $\cB(\xi_0;\epsilon)$ would imply that $\hat{\mu}(\xi)=0$ for any $\xi\in \mathbb{C}^2$, contradicting $\mu\ne 0$.

Therefore $2E[\mu] = \int_{\mathbb{R}^2} \hat{W}_{\textnormal{rep}}(\xi)|\hat{\mu}(\xi)|^2 \rd{\xi} \ge \epsilon\int_{\cB(\xi_0;\epsilon)} |\hat{\mu}(\xi)|^2 \rd{\xi} > 0$ for any compactly supported signed measure $\mu\ne 0$ with $\int_{\mathbb{R}^2}\mu(\bx)\rd{\bx}=\int_{\mathbb{R}^2}\bx\mu(\bx)\rd{\bx}=0$, since $|\hat{\mu}|^2$ is continuous and not identically zero on $\cB(\xi_0;\epsilon)$. This gives the LIC property of $W$.

Conversely, if $\tilde{\Omega}$ is negative somewhere, then Proposition \ref{thm_concave} shows that $W$ is infinitesimal concave, and in particular, $W$ does not have LIC. 
\end{proof}

For $\Omega(\theta) = 1 + \alpha \omega(\theta)$ as given in \eqref{Walpha} with $\omega$ satisfying {\bf (h)}, we have $\tilde{\Omega}(\varphi) = c_s(1 + \alpha \tilde{\omega}(\varphi))$ for some smooth function $\tilde{\omega}$, given by \eqref{lem_FT_1} applied to $\omega$ up to a constant multiple by $c_s$. If $0<s<1$, then $\tilde{\omega}$ is necessarily sign-changing because otherwise \eqref{lem_FT_2r} would imply that $\omega$ is strictly positive, contradicting the assumption $\omega(\frac{\pi}{2})=0$. Therefore, there exists a critical value 
\begin{equation}\label{alphaL}
    \alpha_L := -\frac{1}{\min_\varphi \tilde{\omega}(\varphi)} > 0
\end{equation}
depending on $s$ and $\omega$, around which $W_\alpha$ changes from LIC to non-LIC. This naturally leads to a drastic change in the behavior of energy minimizers as stated in the introduction.

%%%%%%%%%%%%%%%%%%%%%%%%%%%%

\section{Minimizers of LIC potentials}\label{sec_ell}

In this section we study the unique global energy minimizer for potentials $W$ given in \eqref{W} with the LIC property. Denote $\cR_\theta = \begin{pmatrix}
\cos\theta & -\sin\theta \\ \sin\theta & \cos\theta
\end{pmatrix}$ as the rotation matrix, and 
\begin{equation}\label{rhoab}
    \rho_{a,b}(\bx) = \frac{1}{ab}\rho_2\Big(\frac{x_1}{a},\frac{x_2}{b}\Big),\quad \rho_{a,b,\eta}(\bx) = \rho_{a,b}(\cR_{-\eta} \bx)
\end{equation}
for $a,b>0,\,\eta\in\mathbb{R}$, where $\rho_2$ is defined in \eqref{rho2}. We also denote $\rho_{0,b}$ as the weak limit of $\rho_{a,b}$ as $a\rightarrow 0^+$ (similar for $\rho_{0,b,\eta}$ and $\rho_{a,0,\eta}$). Notice that $\rho_{a,b}$ is the push-forward of $\rho_2$ by the linear transformation $\begin{pmatrix}
a & 0 \\ 0 & b
\end{pmatrix}$, and $\rho_{a,b,\eta}$ is the push-forward of $\rho_{a,b}$ by $\cR_\eta$. $\supp\rho_{a,b}$ is an ellipse (possibly degenerate) with axes parallel to the coordinate axes and axis lengths $a R_2$ and $b R_2$, where $R_2$ is as in \eqref{rho2}. $\supp\rho_{a,b,\eta}$ is the previously described ellipse rotated by the angle $\eta$ counterclockwise.
\begin{theorem}\label{thm_ell}
Let $W$ be given by \eqref{W} with $0<s<1$ and $\Omega$ satisfying {\bf (H)}. Let $\tilde{\Omega}$ be given by \eqref{lem_FT_1} with $\tilde{\Omega}\ge 0$. Then exactly one of the following holds (up to translation):
\begin{itemize}
\item There exists a unique tuple $(a,b,\eta)\in (0,\infty)^2\times[0,\pi/2)$ such that $\rho_{a,b,\eta}$
is the unique minimizer of $E$.
\item There exists a unique pair $(b,\eta)\in(0,\infty)\times[0,\pi)$ such that $\rho_{0,b,\eta}$ is the unique minimizer of $E$.
\end{itemize}
If $\tilde{\Omega}\ge c >0$, then item 1 must happen.

\end{theorem}

\begin{remark}
From the proof, it is clear that the same conclusion holds if $|\bx|^2$ is replaced by any positive definite quadratic potential $\alpha_1x_1^2+\alpha_2x_2^2+2\alpha_3x_1x_2$.
\end{remark}
\begin{remark}
If one further requires the symmetry condition $\Omega(\theta)=\Omega(-\theta)$ (i.e., symmetry of the potential about the $x_1$-axis), then either item 1 happens with $\eta=0$, or item 2 happens with $\eta=0$ or $\eta=\pi/2$. This follows from the fact that $D(a,b)$ in \eqref{lem_ab_2} below is always zero under this symmetry condition.
\end{remark}

We first prove a lemma on the linear projection of $\rho_{a,b,\eta}$ onto 1D subspaces. 
\begin{lemma}\label{lem_lin}
Let $T$ be a linear transformation on $\mathbb{R}^2$ whose image is 1D, spanned by $\vec{e}_\varphi$. Then
\begin{equation}\label{lem_lin_1}
    (T_\#\rho_{a,b,\eta}) (y_1\vec{e}_\varphi+y_2\vec{e}_\varphi^\perp) = \lambda\rho_1(\lambda y_1) \delta(y_2)
\end{equation}
where
\begin{equation}
    \lambda = \frac{R_1}{\max\big\{|T(\bx)|:\bx\in\supp\rho_{a,b,\eta}\big\}}\,,
\end{equation}
for any $a,b\in [0,\infty)^2$, $\eta\in\mathbb{R}$ with $(a,b)\ne (0,0)$.
\end{lemma}

\begin{proof}
We start by proving it in the isotropic case $a=b=1$, i.e., for $\rho_{a,b,\eta}=\rho_2$.
By rotation and rescaling, we may assume $\varphi=\pi/2$ and 
\begin{equation}
    T = \begin{pmatrix}
    0 & 0 \\ \sin\beta & \cos\beta
    \end{pmatrix}
\end{equation}
for some $\beta\in\mathbb{R}$. Then notice that
\begin{equation}
    T = T_1T_2,\quad T_1=\begin{pmatrix}
    0 & 0 \\ 0 & 1
    \end{pmatrix},\quad T_2=\begin{pmatrix}
    \cos\beta & -\sin\beta\\ \sin\beta & \cos\beta
    \end{pmatrix}
\end{equation}
Since $T_2$ is an orthogonal matrix and $\rho_2$ is radially symmetric, we have $(T_2)_\#\rho_2 = \rho_2$. Then $T_\#\rho_2 = (T_1)_\#\rho_2$. By explicit calculation, we have $(T_1)_\#\rho_2 = \tilde{\rho}_1(x_2)\delta(x_1)$ where
\begin{equation}
\tilde{\rho}_1(x_2) = 2\int_{0}^{\sqrt{R_2^2-x_2^2}} \big((R_2^2-x_2^2)-x_1^2\big)^{s/2}\rd{x_1} = C(R_2^2-x_2^2)^{(1+s)/2}
\end{equation}
is a rescaling of $\rho_1$ (with total mass 1). Therefore, in the general case,  \eqref{lem_lin_1} holds for some $\lambda>0$. The value of $\lambda$ is determined by matching the support of the two sides.

For $\rho_{a,b,\eta}$, the conclusion follows from the fact that $\rho_{a,b}$ is the push-forward of $\rho_2$ by the composition of $\cR_\eta$ and $\begin{pmatrix}
    a & 0 \\ 0 & b
    \end{pmatrix}$.

\end{proof}

To prove Theorem \ref{thm_ell}, the key observation is that \eqref{decomp} implies that $\big(|\bx|^{-s}\Omega(\theta)\big)*\rho_{a,b}$ is necessarily a quadratic function in $\supp\rho_{a,b}$.

\begin{lemma}\label{lem_ab}
Assume $0<s<1$ and $\Omega$ satisfies {\bf (H)}. Assume $a,b\in(0,\infty)$. Then
\begin{equation}\label{lem_ab_1}
(|\bx|^{-s}\Omega(\theta)+A x_1^2 + B x_2^2 + 2D x_1x_2)*\rho_{a,b} =  C_{\Omega,a,b},\quad \bx\in \supp\rho_{a,b}
\end{equation}
for some constant $C_{\Omega,a,b}$, with
\begin{equation}\label{lem_ab_2}\begin{split}
& \begin{pmatrix}A(a,b) \\
B(a,b) \\
D(a,b)
\end{pmatrix} = \tau_s (R_1/R_2)^{2+s}\int_{-\pi}^\pi (a^2\cos^2\varphi + b^2\sin^2\varphi)^{-(2+s)/2}\begin{pmatrix}\cos^2\varphi \\
\sin^2\varphi \\
\cos\varphi\sin\varphi
\end{pmatrix}\, \tilde{\Omega}(\varphi)\rd{\varphi}\\ 
\end{split}\end{equation}
Furthermore, if $\tilde{\Omega}\ge  0$, then $\big(|\bx|^{-s}\Omega(\theta)+A x_1^2 + B x_2^2 + 2D x_1x_2\big)*\rho_{a,b}$ achieves its minimal value on $\supp\rho_{a,b}$.

If $\tilde{\Omega}\ge  0$ and $a=0,b>0$, then the same is true provided that the integral in the expression of $A$ is finite. If $\tilde{\Omega}\ge  0$ and $a>0,b=0$, then the same is true provided that the integral in the expression of $B$ is finite. 
\end{lemma}

The explicit formula of $C_{\Omega,a,b}$ is given by
\begin{equation}
    C_{\Omega,a,b} = V_1 \tau_s (R_1/R_2)^{s}\int_{-\pi}^\pi (a^2\cos^2\varphi + b^2\sin^2\varphi)^{-s/2}\, \tilde{\Omega}(\varphi)\rd{\varphi}
\end{equation}
where $V_1$ is given by \eqref{calcV1}.

\begin{remark}
Here, in case $a,b>0$, the steady state condition \eqref{lem_ab_1} works for possibly non-LIC potentials, showing that $\rho_{a,b}$ is always a steady state for the interaction potential $|\bx|^{-s}\Omega(\theta)+A x_1^2 + B x_2^2+ 2D x_1x_2$. It is guaranteed to be the unique (global) minimizer for LIC potentials, but it cannot even be a $d_\infty$-local minimizer for non-LIC potentials, because it violates the necessary condition given by Proposition \ref{thm_concave} and the sentence after it.
\end{remark}

\begin{remark}\label{rem_ab}
Although $R_1$ is not well-defined for $1\le s < 2$, we notice that 
\begin{align*}
    \tau_s R_1^{2+s}&=(2\pi)^{-s}\Gamma(s)\sin\frac{\pi(1-s)}{2}\Big(\frac{2\cos\frac{s\pi}{2}}{s(s+1)\pi} \beta\Big(\frac{1}{2},\frac{3+s}{2}\Big)\Big)^{-1}\\&=(2\pi)^{-s}\Gamma(s)\Big(\frac{2}{s(s+1)\pi} \beta\Big(\frac{1}{2},\frac{3+s}{2}\Big)\Big)^{-1}
\end{align*}
is well-defined for $1\le s < 2$. It turns out that \eqref{lem_ab_1} and \eqref{lem_ab_2} are also true for $1\le s < 2$ (see Lemma \ref{lem_ab2}), but requires a different proof.
\end{remark}

\begin{figure}[ht!]
    \centering
    \includegraphics[width=0.5\textwidth]{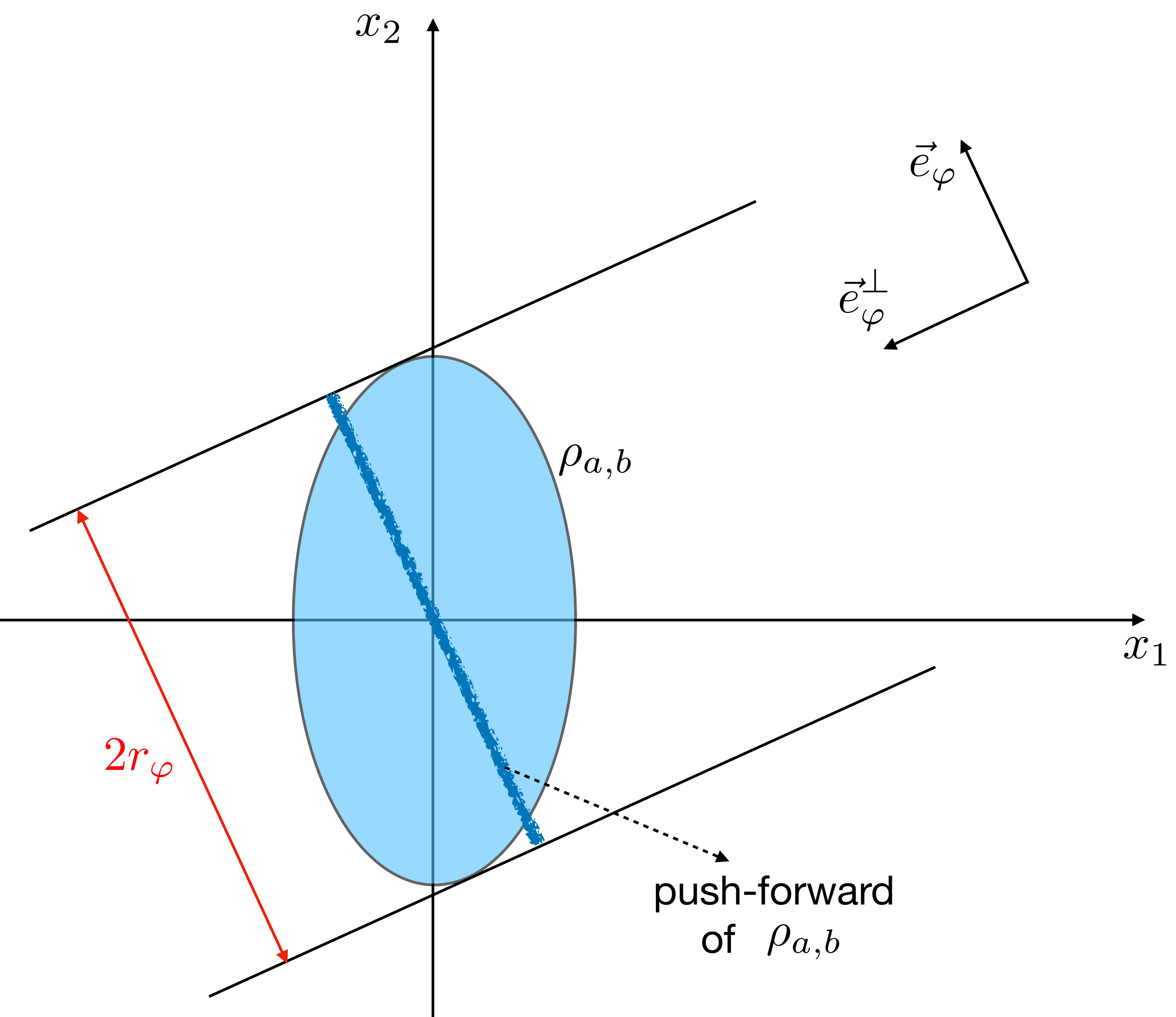}
    \caption{Calculation of the potential generated by $\rho_{a,b}$ via 1D projections.}
    \label{fig:ell}
\end{figure}

\begin{proof}
We first assume $a,b\in(0,\infty)$. Notice that for a fixed $\varphi$, the rotated 1D interaction potential $|\bx\cdot \vec{e}_\varphi|^{-s}$ generates a 1D potential
\begin{equation}\label{1Dgen}
\big(|\bx\cdot \vec{e}_\varphi|^{-s} * \rho_{a,b}\big) (y_1\vec{e}_\varphi+y_2\vec{e}_\varphi^\perp) = \int_{\mathbb{R}}|y_1-z_1|^{-s}\int_{\mathbb{R}} \rho_{a,b}(z_1\vec{e}_\varphi+z_2\vec{e}_\varphi^\perp) \rd{z_2}\rd{z_1}
\end{equation}
in the sense that the last quantity is independent of $y_2$. The inner integral $\int_{\mathbb{R}} \rho_{a,b}(z_1\vec{e}_\varphi+z_2\vec{e}_\varphi^\perp) \rd{z_2}$ is exactly the push-forward of $\rho_{a,b}$ by the projection onto the $\vec{e}_\varphi$ direction (as a function of $z_1$). This implies it is a rescaling of $\rho_1$, by Lemma \ref{lem_lin}. Since the maximum of $|z_1|$ with the constraint $z_1\vec{e}_\varphi+z_2\vec{e}_\varphi^\perp\in \supp\rho_{a,b}$ is \begin{equation}
    r_\varphi := R_2(a^2\cos^2\varphi + b^2\sin^2\varphi)^{1/2},
\end{equation}
we see that
\begin{equation}
\int_{\mathbb{R}} \rho_{a,b}(z_1\vec{e}_\varphi+z_2\vec{e}_\varphi^\perp) \rd{z_2} = \frac{R_1}{r_\varphi}\rho_1\Big(\frac{R_1}{r_\varphi}z_1\Big).
\end{equation}
See Figure \ref{fig:ell} as an illustration.

The fact that $\rho_1$ minimizes the energy associated to the potential $|x|^{-s}+|x|^2$ implies that
\begin{equation}\label{V1}
\int_{\mathbb{R}} \big(|y-z|^{-s}+|y-z|^{2}\big)\rho_1(z)\rd{z} \left\{\begin{split}
   & = V_1,\quad y\in [-R_1,R_1] \\
   & > V_1,\quad y\notin [-R_1,R_1] \\
\end{split}\right.
\end{equation}
for some constant $V_1$ given in \eqref{calcV1}. 
Rescaling by some $\lambda>0$, we get
\begin{equation}
\int_{\mathbb{R}} \big(|y-z|^{-s}+\lambda^{2+s}|y-z|^{2}\big)\lambda\rho_1(\lambda z)\rd{z} \left\{\begin{split}
   & = V_1\lambda^{s},\quad y\in [-R_1/\lambda,R_1/\lambda] \\
   & > V_1\lambda^{s},\quad y\notin [-R_1/\lambda,R_1/\lambda] \\
\end{split}\right.\,.
\end{equation}
Applying to \eqref{1Dgen}, we see that
\begin{equation}
\Big(\Big(|\bx\cdot \vec{e}_\varphi|^{-s}+\Big(\frac{R_1}{r_\varphi}\Big)^{2+s}|\bx\cdot \vec{e}_\varphi|^2\Big) * \rho_{a,b}\Big) (y_1\vec{e}_\varphi+y_2\vec{e}_\varphi^\perp) \left\{\begin{split}
   & = V_1\Big(\frac{R_1}{r_\varphi}\Big)^s,\quad y_1\in [-r_\varphi,r_\varphi] \\
   & > V_1\Big(\frac{R_1}{r_\varphi}\Big)^s,\quad y_1\notin [-r_\varphi,r_\varphi] \\
\end{split}\right.\,.
\end{equation}
In particular, $\big(|\bx\cdot \vec{e}_\varphi|^{-s}+(\frac{R_1}{r_\varphi})^{2+s}|\bx\cdot \vec{e}_\varphi|^2\big) * \rho_{a,b}$ achieves minimum on $\supp\rho_{a,b}$.

Combining with \eqref{decomp} and integrating in $\varphi$, we see that 
\begin{equation}
    \Big(|\bx|^{-s}\Omega(\theta) + \tau_s \int_{-\pi}^\pi \Big(\frac{R_1}{r_\varphi}\Big)^{2+s}|\bx\cdot \vec{e}_\varphi|^2 \tilde{\Omega}(\varphi)\rd{\varphi}\Big)*\rho_{a,b}
\end{equation}
is constant on $\supp\rho_{a,b}$, and achieves minimum on $\supp\rho_{a,b}$ if $\tilde{\Omega}\ge 0$. The last integral is 
\begin{equation*}
    \tau_s (R_1/R_2)^{2+s}\int_{-\pi}^\pi\!\!\! (a^2\cos^2\varphi + b^2\sin^2\varphi)^{-(2+s)/2}(x_1^2\cos^2\varphi+x_2^2\sin^2\varphi+2x_1x_2\cos\varphi\sin\varphi) \tilde{\Omega}(\varphi)\rd{\varphi}
\end{equation*}
which is a quadratic function in $\bx$. Therefore we get the conclusion \eqref{lem_ab_1} with the coefficients $A,B,D$ given by \eqref{lem_ab_2}.

Finally we treat the case $\tilde{\Omega}\ge  0$, $a=0,b>0$ (the case $a>0,b=0$ is similar). The previous argument clearly works if $0\notin \supp\tilde{\Omega}$ since the integrals in \eqref{lem_ab_2} are locally integrable. To treat the general case, we take a sequence of nonnegative  functions $\{\tilde{\Omega}_n\}$ satisfying {\bf (H)} such that $0\notin\supp\tilde{\Omega}_n$,  $\{\tilde{\Omega}_n(\varphi)\}$ is increasing in $n$ and converges to $\tilde{\Omega}(\varphi)$ for every $\varphi$ (noticing that the assumption $A<\infty $ implies $\tilde{\Omega}(0)=0$, such a sequence can be obtained by multiplying $\tilde{\Omega}$ with some mollifiers). Then for any $\bx\in\supp\rho_{0,b}$,  the monotone convergence theorem shows that $\big(|\bx|^{-s}\Omega_n(\theta)*\rho_{0,b}\big)(\bx)$ converges to $\big(|\bx|^{-s}\Omega(\theta)*\rho_{0,b}\big)(\bx)$. Also, $A_n,B_n$ converge to $A,B$ since  $B<\infty$ is clear and $A<\infty$ is assumed, and then $D_n$ converges to $D$ since the integrand of $D_n$ is dominated by that of $A_n+B_n$. Therefore $\big((|\bx|^{-s}\Omega_n(\theta) + A_n x_1^2+ B_n x_2^2 + 2D_n x_1x_2\big)*\rho_{0,b})(\bx)$ converges to $\big((|\bx|^{-s}\Omega(\theta) + A x_1^2 + B x_2^2 + 2D x_1x_2)*\rho_{0,b}\big)(\bx)$. This implies $C_{\Omega_n,0,b}$ converges, say, to $C_{\Omega,0,b}$. This justifies \eqref{lem_ab_1}. 

To see that $\big(|\bx|^{-s}\Omega(\theta) + A x_1^2 + B x_2^2 + 2D x_1x_2\big)*\rho_{0,b}$ achieves its minimal value on $\supp\rho_{a,b}$, we notice that the previous arguments applied to  $\Omega_n$ shows that $\big((|\bx|^{-s}\Omega_n(\theta)+ A_n x_1^2 + B_n x_2^2 + 2D_n x_1x_2)*\rho_{0,b}\big)(\bx)  \ge C_{\Omega_n,0,b}$ for any $\bx$. Sending $n\rightarrow\infty$, we get the conclusion by the monotone convergence theorem.

\end{proof}

Then we seek the values of $(a,b,\eta)$ which match with $(A,B,D)=(1,1,0)$ specified by the the quadratic potential in \eqref{W}.
\begin{lemma}\label{lem_abf}
For $0<s<1$ and a given $\Omega$ satisfying {\bf (H)}, let $A_\eta,B_\eta,D_\eta$ be given by \eqref{lem_ab_2} with $\Omega$ replaced by $\Omega(\eta+\cdot)$, as functions of $(a,b)$. If $\tilde{\Omega}\ge 0$, then at least one of the following two statements is true:
\begin{itemize}
    \item[i)] There exists $(a,b,\eta)\in (0,\infty)^2\times[0,\pi/2)$ such that $A_\eta(a,b)=1,B_\eta(a,b)=1,D_\eta(a,b)=0$. 
    \item[ii)] There exists $(b,\eta)\in(0,\infty)\times[0,\pi)$ such that $B_\eta(0,b)=1,D_\eta(0,b)=0$ and $0\le A_\eta(0,b)\le 1$.
\end{itemize}
Furthermore, if $\tilde{\Omega}\ge c > 0$, then item {i)} must happen.
\end{lemma}

\begin{proof}

{\bf STEP 1}: 
Treat the case $\tilde{\Omega}\ge c > 0$. 

We claim that there exists a unique pair $(a,b)\in (0,\infty)^2$ such that $A(a,b)=B(a,b)=1$. We first show the existence. Since $A$ and $B$ are homogeneous in $(a,b)$, it suffices to find $b\in (0,\infty)$ such that
\begin{equation}\label{fb}
    f(b) = \frac{A(1,b)}{B(1,b)} = \frac{\int_{-\pi}^\pi (\cos^2\varphi + b^2\sin^2\varphi)^{-(2+s)/2}\cos^2\varphi \tilde{\Omega}(\varphi)\rd{\varphi}}{\int_{-\pi}^\pi (\cos^2\varphi + b^2\sin^2\varphi)^{-(2+s)/2}\sin^2\varphi \tilde{\Omega}(\varphi)\rd{\varphi}} = 1.
\end{equation}
It is clear that $f(b)$ is continuous. As $b\rightarrow 0^+$, the numerator converges to $\int_{-\pi}^\pi|\cos\varphi|^{-s} \tilde{\Omega}(\varphi)\rd{\varphi} < \infty$ but the denominator goes to $\infty$ because $\tilde{\Omega}\ge c > 0$ and the pointwise limit of the integrand has a non-integrable singularity at $\pi/2$. Therefore $\lim_{b\rightarrow 0^+} f(b) = 0$. Similarly, $\lim_{b\rightarrow \infty} f(b) = \infty$. Therefore we conclude that there exists $b$ such that $f(b)=1$.

To see the uniqueness of such $(a,b)$, we notice that for $a,b>0$, the Jacobian
\begin{equation}
    \frac{\partial(A,B)}{\partial(a,b)} = -\begin{pmatrix} a\int \cos^4\varphi  &  a\int \cos^2\varphi\sin^2\varphi  \\
    b\int \cos^2\varphi\sin^2\varphi & b\int \sin^4\varphi \\
    \end{pmatrix}
\end{equation}
where the integrals are in $\varphi\in[-\pi,\pi]$ with respect to the weight $(s+2)\tau_s (R_1/R_2)^{2+s}(a^2\cos^2\varphi + b^2\sin^2\varphi)^{-(s+4)/2}\tilde{\Omega}(\varphi)\rd{\varphi}$. Since $\tilde{\Omega}\ge c>0$, the matrices 
$$
\begin{pmatrix} a \cos^4\varphi  &  a \cos^2\varphi\sin^2\varphi  \\
b \cos^2\varphi\sin^2\varphi & b \sin^4\varphi \\
\end{pmatrix}
$$ 
are semi-positive-definite rank-1 matrices for every $\varphi$, having different eigenvectors for different $\varphi\in [0,\pi/2]$, we see that $\frac{\partial(A,B)}{\partial(a,b)}$ is negative-definite for any $a,b>0$. This implies the injective property of the map $(a,b)\mapsto (A,B)$, which finishes the proof of the claim.

Notice that $\Omega(\cdot+\eta)$ corresponds to $\tilde{\Omega}(\cdot+\eta)$ on the Fourier side, and therefore the previous claim applies to $\Omega(\cdot+\eta)$. Then we see that for every $\eta\in [0,\pi/2]$, we have a unique pair $(a_\eta,b_\eta)$ such that $A_\eta(a_\eta,b_\eta)=B_\eta(a_\eta,b_\eta)=1$. Since the map  $(a,b,\eta)\mapsto (A_\eta(a,b),B_\eta(a,b))$ is clearly smooth on $(0,\infty)^2\times\mathbb{R}$, we see that $a_\eta,b_\eta$ depend continuously on $\eta$, and so does $g(\eta) := D_\eta(a_\eta,b_\eta)$. 

Notice that $a_{\pi/2}=b_0,\,b_{\pi/2}=a_0$ since $\rho_{b,a}$ is exactly the rotation of $\rho_{a,b}$ by $\pi/2$. Therefore
\begin{equation}\begin{split}
    g(\frac{\pi}{2}) = & \tau_s (R_1/R_2)^{2+s}\int_{-\pi}^\pi (b_0^2\cos^2\varphi + a_0^2\sin^2\varphi)^{-(2+s)/2}
\cos\varphi\sin\varphi
\, \tilde{\Omega}(\varphi+\frac{\pi}{2})\rd{\varphi} \\
= & \tau_s (R_1/R_2)^{2+s}\int_{-\pi}^\pi (b_0^2\sin^2\varphi + a_0^2\cos^2\varphi)^{-(2+s)/2}
\sin\varphi(-\cos\varphi)
\, \tilde{\Omega}(\varphi)\rd{\varphi} = -g(0)\\
\end{split}\end{equation}
by a change of variable $\varphi\mapsto \varphi-\pi/2$. Therefore, the continuity of $g$ implies the existence of $\eta\in[0,\pi/2)$ such that $g(\eta)=0$, which gives the conclusion.

{\bf STEP 2}: Treat the general case $\tilde{\Omega}\ge 0$.

We define $\Omega_n = \frac{1}{n} + (1-\frac{1}{n})\Omega$, and then the corresponding $\tilde{\Omega}_n = \frac{1}{n} c_s + (1-\frac{1}{n})\tilde{\Omega}$ is bounded from below by $\frac{1}{n} c_s>0$. Applying STEP 1, we get $a_n,b_n,\eta_n$ with $A_{\eta_n}(a_n,b_n)=B_{\eta_n}(a_n,b_n)=1,\,D_{\eta_n}(a_n,b_n)=0$. By subtracting a subsequence, we may assume that $a_n\rightarrow a\in [0,\infty]$, $b_n\rightarrow b\in [0,\infty]$, $\eta_n\rightarrow \eta\in [0,\pi/2]$.

It is clear that $a<\infty$ because otherwise one would have
\begin{equation}
    A_{\eta_n}(a_n,b_n) \le a_n^{-(2+s)}\|\tilde{\Omega}\|_{L^\infty}\tau_s (R_1/R_2)^{2+s}\int_{-\pi}^\pi |\cos\varphi|^{-s}\rd{\varphi} \rightarrow 0,\quad \text{ as }n\rightarrow\infty
\end{equation}
contradicting to $A_{\eta_n}(a_n,b_n)=1$. Similarly $b<\infty$. It is also clear that $(a,b)\ne (0,0)$ because otherwise one would have $A_{\eta_n}(a_n,b_n)+B_{\eta_n}(a_n,b_n)\rightarrow\infty$.

Then we separate into cases:
\begin{itemize}
    \item If $a,b>0$, then it is clear that $A_{\eta_n}(a_n,b_n)\rightarrow A_\eta(a,b)$ and similarly for $B$ and $D$, and we get the conclusion with item i) in the statement.
    \item If $a=0,b>0$, then $B_{\eta_n}(a_n,b_n)\rightarrow B_\eta(0,b)$ by the dominated convergence theorem, and we get $B_\eta(0,b)=1$. We will then prove item ii) in the statement of the theorem by proving     \begin{equation}\label{ABD_bound}
        0\le A_\eta(0,b)x_1^2+B_\eta(0,b)x_2^2+2D_\eta(0,b)x_1x_2\le x_1^2+x_2^2
    \end{equation} for any $(x_1,x_2)$.
    In fact, the lower bound follows from
    \begin{equation*}\begin{split}
        & A_\eta(a,b)x_1^2+B_\eta(a,b)x_2^2+2D_\eta(a,b)x_1x_2 \\
        = & \tau_s (R_1/R_2)^{2+s}\int_{-\pi}^\pi (a^2\cos^2\varphi + b^2\sin^2\varphi)^{-(2+s)/2}(x_1\cos\varphi+x_2\sin\varphi)^2
\, \tilde{\Omega}(\varphi+\eta)\rd{\varphi}\ge 0
    \end{split}\end{equation*}
    for any $(x_1,x_2)$. To see the upper bound, we use Fatou's Lemma to get
    \begin{equation*}\begin{split}
        & \tau_s (R_1/R_2)^{2+s}\int_{-\pi}^\pi (a^2\cos^2\varphi + b^2\sin^2\varphi)^{-(2+s)/2}(x_1\cos\varphi+x_2\sin\varphi)^2
        \, \tilde{\Omega}(\varphi+\eta)\rd{\varphi}\\ 
        = & \tau_s (R_1/R_2)^{2+s}\int_{-\pi}^\pi \lim_{n\rightarrow\infty}(a_n^2\cos^2\varphi + b_n^2\sin^2\varphi)^{-(2+s)/2}(x_1\cos\varphi+x_2\sin\varphi)^2
        \, \tilde{\Omega}(\varphi+\eta_n)\rd{\varphi}\\ 
        \le & \lim_{n\rightarrow\infty}\tau_s (R_1/R_2)^{2+s}\int_{-\pi}^\pi (a_n^2\cos^2\varphi + b_n^2\sin^2\varphi)^{-(2+s)/2}(x_1\cos\varphi+x_2\sin\varphi)^2
        \, \tilde{\Omega}(\varphi+\eta_n)\rd{\varphi}\\ 
        = & \lim_{n\rightarrow\infty}A_{\eta_n}(a_n,b_n)x_1^2+B_{\eta_n}(a_n,b_n)x_2^2+2D_{\eta_n}(a_n,b_n)x_1x_2 = x_1^2+x_2^2.
    \end{split}\end{equation*}
    \item If $a>0,b=0$, then we get item ii) in the statement of the theorem with $\eta\in [\pi/2,\pi)$, similar to the previous case.
\end{itemize}

\end{proof}

\begin{proof}[Proof of Theorem \ref{thm_ell}]
If $\tilde{\Omega}\ge c > 0$, then Lemma \ref{lem_abf} shows the existence of $(a,b,\eta)\in(0,\infty)^2\times [0,\pi/2)$ such that $A_\eta=B_\eta=1,D_\eta=0$. Then Lemma \ref{lem_ab} shows that $\rho_{a,b}$ satisfies the condition \eqref{EL} for $\Omega(\cdot+\eta)$. Therefore $\rho_{a,b}$ is the unique global energy minimizer for $\Omega(\cdot+\eta)$, i.e., $\rho_{a,b,\eta}$ is the unique global energy minimizer for $\Omega$.

If we only assume $\tilde{\Omega}\ge 0$, then either we have the same situation as before, or we have item ii) as in Lemma \ref{lem_abf}. For the case of item ii), we have some $(b,\eta)\in (0,\infty)\times [0,\pi)$ such that  $B=1,\,D=0$ and $0\le A \le 1$ (where $B$ refers to $B_\eta(0,b)$, and similar for $A,D$ below). Then Lemma \ref{lem_ab} shows that $(|\bx|^{-s}\Omega(\theta+\eta)+A x_1^2 + x_2^2)*\rho_{0,b}$ achieves its minimum on $\supp\rho_{0,b}$. Therefore $(|\bx|^{-s}\Omega(\theta+\eta)+x_1^2 + x_2^2)*\rho_{0,b}$ also achieves its minimum on $\supp\rho_{0,b}$, i.e., satisfies the condition \eqref{EL}. Therefore $\rho_{0,b}$ is the unique global energy minimizer for $\Omega(\cdot+\eta)$, i.e., $\rho_{0,b,\eta}$ is the unique global energy minimizer for $\Omega$.
\end{proof}

\section{A special solution to the gradient flow}
\label{sec:gradflow}

The Wasserstein-2 gradient flow associated to the energy \eqref{E} formally reads
\begin{equation}\label{flow}
    \partial_t \rho + \nabla\cdot (\rho\bu) = 0,\quad \bu = -\nabla W*\rho
\end{equation}
where $\rho(t,\bx)$ is a time-dependent particle density function defined on $\mathbb{R}_+\times \mathbb{R}^2$. Lemma \ref{lem_ab} shows that $\bu$ is linear if $\rho$ has the form $\rho_{a,b,\eta}$ at some $t$. As a consequence, the class of functions $\rho_{a,b,\eta}$ is invariant under the gradient flow \eqref{flow}. For simplicity, we will further assume $\Omega(\theta)=\Omega(-\theta)$, so that $D=0$ always holds. In this case, we will focus on the invariance of the class of functions $\rho_{a,b}$. The general case will be briefly discussed in Remark \ref{rem_flow}.

\begin{theorem}\label{thm_ode}
Assume $0<s<1$ and $W$ is given by \eqref{W} with $\Omega$ satisfying {\bf (H)} and $\Omega(\theta)=\Omega(-\theta)$. Assume the positive functions $a(t),b(t)$ solve the ODE system
\begin{equation}\label{ab_ode}\left\{\begin{split}
    & a'(t) = 2\big(A(a(t),b(t))-1\big)a(t) \\
    & b'(t) = 2\big(B(a(t),b(t))-1\big)b(t) \\
\end{split}\right.
\end{equation}
where $A,B$ are defined in \eqref{lem_ab_2}. Then $\rho(t,\cdot)=\rho_{a(t),b(t)}$ solves \eqref{flow}.
\end{theorem}

The global wellposedness of \eqref{ab_ode} for initial condition $(a_0,b_0)\in (0,\infty)^2$ follows from the uniform upper and lower bounds on $(a(t),b(t))$, as will be shown later in the proof of Theorem \ref{thm_odeconv}.

\begin{proof}
Lemma \ref{lem_ab} shows that
\begin{equation}
    W*\rho_{a,b} = (1-A)x_1^2 + (1-B)x_2^2 + \text{constant},\quad \bx\in \supp\rho_{a,b}.
\end{equation}
Therefore, we get
\begin{equation}
    -\nabla W*\rho_{a,b} = 2\begin{pmatrix}(A-1)x_1 \\ (B-1)x_2\end{pmatrix},\quad \bx\in \supp\rho_{a,b}.
\end{equation}
If $\rho(t,\cdot)=\rho_{a(t),b(t)}$ with $a(t)$ and $b(t)$ satisfying \eqref{ab_ode}, then we compute each term in the continuity equation to obtain
\begin{equation}\begin{split}
    \partial_t\rho = & \partial_t \Big(\frac{1}{ab} \rho_2\Big(\frac{x_1}{a},\frac{x_2}{b}\Big)\Big)\\ 
    = & \Big(-\frac{a'}{a^2b}-\frac{b'}{ab^2}\Big) \rho_2\Big(\frac{x_1}{a},\frac{x_2}{b}\Big) + \frac{1}{ab} \partial_{x_1}\rho_2\Big(\frac{x_1}{a},\frac{x_2}{b}\Big)\Big(-\frac{x_1 a'}{a^2}\Big)\\ &\qquad \qquad\qquad \qquad \qquad \qquad\,+ \frac{1}{ab} \partial_{x_2}\rho_2\Big(\frac{x_1}{a},\frac{x_2}{b}\Big)\Big(-\frac{x_2 b'}{b^2}\Big)\,,
\end{split}\end{equation}
\begin{equation}
    \rho \nabla\cdot \bu = \frac{1}{ab} \rho_2\Big(\frac{x_1}{a},\frac{x_2}{b}\Big)  \big(2(A-1)+2(B-1)\big)\,,
\end{equation}
and
\begin{equation}\begin{split}
    \bu\cdot\nabla\rho  = & 2(A-1)x_1\partial_{x_1}\Big(\frac{1}{ab} \rho_2\Big(\frac{x_1}{a},\frac{x_2}{b}\Big) \Big) + 2(B-1)x_2\partial_{x_2}\Big(\frac{1}{ab} \rho_2\Big(\frac{x_1}{a},\frac{x_2}{b}\Big) \Big) \\
     = & 2(A-1)x_1\frac{1}{ab} \partial_{x_1}\rho_2\Big(\frac{x_1}{a},\frac{x_2}{b}\Big)  \frac{1}{a} + 2(B-1)x_2\frac{1}{ab} \partial_{x_2}\rho_2\Big(\frac{x_1}{a},\frac{x_2}{b}\Big)  \frac{1}{b}\,. \\
\end{split}\end{equation}
Summing these quantities and using \eqref{ab_ode}, we see that \eqref{flow} holds in $\supp\rho(t,\cdot)$.
\end{proof}

Then we analyze the long time behavior of \eqref{ab_ode}. We define the energy functional
\begin{equation}
    \cE(a,b) = \cM(a,b) + \frac{s}{2} (a^2+b^2)
\end{equation}
with
\begin{equation}\label{M}
    \cM(a,b) = a^2A + b^2B = \tau_s (R_1/R_2)^{2+s}\int_{-\pi}^\pi (a^2\cos^2\varphi + b^2\sin^2\varphi)^{-s/2} \tilde{\Omega}(\varphi)\rd{\varphi}.
\end{equation}
$\cE(a,b)$ is a constant multiple of the original total energy \eqref{E}, see Remark \ref{rem_cE} for details. 

\begin{theorem}\label{thm_odeconv}
Assume $0<s<1$ and $W$ is given by \eqref{W} with $\Omega$ satisfying {\bf (H)} and $\tilde{\Omega}\ge c > 0$. Let $(a(t),b(t))$ solves \eqref{ab_ode} with initial condition $(a_0,b_0)\in (0,\infty)^2$. Then we have the exponential convergence in energy
\begin{equation}
   \cE(t) - \cE(a_\infty,b_\infty) \le Ce^{-\lambda t}
\end{equation}
for some $C,\lambda>0$, where $(a_\infty,b_\infty)$ denotes the unique solution to $A=B=1$. Also, $(a(t),b(t))\rightarrow (a_\infty,b_\infty)$ as $t\rightarrow\infty$.
\end{theorem}

\begin{proof}
From the integral expression, it is clear that
\begin{equation}
    \partial_a \cM = -s a A,\quad \partial_b \cM = -s b B.
\end{equation}
Therefore, we deduce
\begin{align}\label{dE}
    \cE' &= -s a A a' -s b B b' + s(a a' + b b') \nonumber\\
    &= -s\big(a a'(A-1) + b b'(B-1)\big) = -2s \big(a^2(A-1)^2+b^2(B-1)^2\big),
\end{align}
i.e., $\cE$ is decreasing in time.

Then, from the expression of $\cE$, it is clear that $a(t),b(t)<C_{\textnormal{unif}}$ for some constant $C_{\textnormal{unif}}$ uniformly in time. Also, notice that $A(a,b)\ge c a^{-s-1}$ for $a$ close to 0 if $b$ is bounded from above. In fact, this can be seen from the expression of $A(a,b)$ in \eqref{lem_ab_2}, in which the integrand is well-approximated by $\tilde{\Omega}(0)(a^2+b^2\varphi^2)^{-(2+s)/2}$ near $\varphi=0$ for small $a$. Therefore $a'>0$ for sufficiently small $a$, i.e., $a(t)>c_{\textnormal{unif}}$ uniformly in time. Similarly $b(t)>c_{\textnormal{unif}}$ uniformly in time.

In the proof of Lemma \ref{lem_abf}, we have shown that the Jacobian $\frac{\partial(A,B)}{\partial(a,b)}$ is negative-definite. Since the eigenvalues of $\frac{\partial(A,B)}{\partial(a,b)}$ are continuous functions of $a,b$, there exists $\lambda>0$ such that
\begin{equation}
    \frac{\partial(A,B)}{\partial(a,b)}\le -\lambda I_2,\quad \forall(a,b)\in [c_{\textnormal{unif}},C_{\textnormal{unif}}]^2.
\end{equation}
This implies that the map $(a,b)\mapsto(A,B)$ is invertible (as a smooth map), and 
\begin{equation}
    (a-a_\infty)^2+(b-b_\infty)^2 \le C \big((A(a,b)-1)^2+(B(a,b)-1)^2 \big),\quad \forall(a,b)\in [c_{\textnormal{unif}},C_{\textnormal{unif}}]^2
\end{equation}
since $A(a_\infty,b_\infty)=B(a_\infty,b_\infty)=1$ by definition. As a consequence, \eqref{dE} gives
\begin{equation}
    \cE' \le -c \big((a-a_\infty)^2+(b-b_\infty)^2\big).
\end{equation}
As a smooth function with a unique minimum $(a_\infty,b_\infty)$, $\cE$ clearly satisfies the estimate
\begin{equation}
    \cE(a,b)-\cE(a_\infty,b_\infty) \le C\big((a-a_\infty)^2+(b-b_\infty)^2\big),\quad \forall(a,b)\in [c_{\textnormal{unif}},C_{\textnormal{unif}}]^2.
\end{equation}
Therefore, we conclude that
\begin{equation}
    \cE' \le -c\big(\cE(a,b)-\cE(a_\infty,b_\infty)\big),
\end{equation}
which gives the exponential convergence in energy as $t\rightarrow\infty$, as well as the convergence of $(a(t),b(t))$ to $(a_\infty,b_\infty)$.
\end{proof}

\begin{remark}\label{rem_flow}
For a general $\Omega$ satisfying {\bf (H)}, one can similarly show that $\rho_{a(t),b(t),\eta(t)}$ solves \eqref{flow} if $(a,b,\eta)$ solves the ODE system
\begin{equation}\label{ab_ode2}\left\{\begin{split}
    & a'(t) = 2\big(A_{\eta(t)}(a(t),b(t))-1\big)a(t) \\
    & b'(t) = 2\big(B_{\eta(t)}(a(t),b(t))-1\big)b(t) \\
    & \eta'(t) = 2D_{\eta(t)}\frac{a(t)^2+b(t)^2}{a(t)^2-b(t)^2}
\end{split}\right.
\end{equation}
where $A_\eta,B_\eta,D_\eta$ are as in Lemma \ref{lem_abf}. We leave to the interested reader to check that $\rho_{a(t),b(t),\eta(t)}$ solves \eqref{flow} by a direct computation as in Theorem \ref{thm_ode}. From Lemma \ref{lem_ab}, it is clear that $\rho_{a,b,\eta},\,a,b>0$ cannot be a steady state for the gradient flow \eqref{flow} unless $A_\eta(a,b)=B_\eta(a,b)=1,\,D_\eta(a,b)=0$. With the further assumption $\tilde{\Omega} \ge c >0$, this would imply that a steady state of the form $\rho_{a,b,\eta}$ has to be the unique energy minimizer. As a consequence, the unique steady state of \eqref{ab_ode2} is given by the parameters $(a,b,\eta)$ for the energy minimizer. The analysis of wellposedness and long time behavior of this ODE system is left as future work.
\end{remark}

\begin{remark}\label{rem_cE}
We show that $\cE$ is a constant multiple of the original total energy \eqref{E}. By Lemma \ref{lem_ab}, we have
\begin{equation}
    (|\bx|^{-s}\Omega(\theta)+Ax_1^2+Bx_2^2) * \rho_{a,b} = V_1(R_1/R_2)^{-2}\cM(a,b) ,\quad \bx\in\supp\rho_{a,b}.
\end{equation}
Therefore, we deduce that
\begin{equation}\begin{split}
    (|\bx|^{-s}\Omega(\theta)+x_1^2+x_2^2) * \rho_{a,b} = & V_1(R_1/R_2)^{-2}\cM(a,b) + (1-A)x_1^2+(1-B)x_2^2\\
    & + \int_{\mathbb{R}^2} \big((1-A)y_1^2+ (1-B)y_2^2\big)\rho_{a,b}(\by)\rd{\by},\quad \bx\in\supp\rho_{a,b}
\end{split}\end{equation}
and
\begin{equation}
    2E[\rho_{a,b}] = V_1(R_1/R_2)^{-2}\cM(a,b) + \int_{\mathbb{R}^2} \big((2-2A)x_1^2+(2-2B)x_2^2 \big)\rho_{a,b}(\bx)\rd{\bx}.
\end{equation}
Using Lemma \ref{lem_lin} and the explicit formula of $\rho_1$ in \eqref{rho1D}, we see that
\begin{equation}
    \int_{\mathbb{R}^2} x_2^2 \,\rho_{a,b}(\bx)\rd{\bx} = \int_{\mathbb{R}} x^2 \frac{R_1}{R_2b}\rho_1\Big(\frac{R_1}{R_2b}x\Big)\rd{x} = (R_2/R_1)^2 C_1 R_1^{4+s}\beta\Big(\frac{3}{2},\frac{s+3}{2}\Big)b^2
\end{equation}
and similarly
\begin{equation}
    \int_{\mathbb{R}^2} x_1^2\, \rho_{a,b}(\bx)\rd{\bx} = (R_2/R_1)^2 C_1 R_1^{4+s}\beta\Big(\frac{3}{2},\frac{s+3}{2}\Big)a^2.
\end{equation}
Therefore, combined with \eqref{M}, we get
\begin{equation}
    2E[\rho_{a,b}] = (R_1/R_2)^{-2}\Big((V_1-2\tilde{C}_1)\cM(a,b) + 2\tilde{C}_1 (a^2+b^2) \Big)
\end{equation}
where $\tilde{C}_1 = C_1 R_1^{4+s}\beta(\frac{3}{2},\frac{s+3}{2}) = \frac{1}{4+s}R_1^2$. Using the expression \eqref{calcV1} for $V_1$, one can see that $\frac{V_1}{\tilde{C_1}} = \frac{4+2s}{s}$. Therefore we see that $E[\rho_{a,b}] = (R_1/R_2)^{-2}V_1\frac{1}{2+s} \cE$. This is another way to check that $\cE$ is decreasing in time along the flow of \eqref{flow}.
\end{remark}

%%%%%%%%%%%%%%%%%%%%%%%%%%%%%%%%%%%%%%%%%%%%%%%%%%%%%%

\section{Behavior of minimizers for large $\alpha$}
\label{sec:largealpha}

In this section we discuss the behavior of minimizers for the potential $W_\alpha$ given by \eqref{Walpha}, for $\alpha$ large. We will focus on the case when $\omega(\theta)$ achieves its minimal value at the only point $\omega(\frac{\pi}{2})=0$, and thus intuitively minimizers tend to concentrate along the $x_2$-axis.

\subsection{$\rho_{\textnormal{1D}}$ is the minimizer of strongly coercive potentials}

\begin{theorem}\label{thm_coer}
For $0<s<1$, there exists a constant $C_*=C_*(s)$ such that the following holds. Let $W$ be given by \eqref{W} with $\Omega$ satisfying {\bf (H)}, $\Omega(\frac{\pi}{2})=1$ and 
\begin{equation}
    \Omega(\theta)\ge 1+C_*\big|\theta-\frac{\pi}{2}\big|^2,\quad \forall\theta\in [0,\pi]. 
\end{equation}
Then  $\rho_{\textnormal{1D}}$ is the unique minimizer of $E$ (up to translation).

Assume $W_\alpha$ is given by \eqref{Walpha} with $\omega$ satisfying {\bf (h)} and 
\begin{equation}\label{thm_coer_1}
    \omega(\theta)\ge c_\omega\big|\theta-\frac{\pi}{2}\big|^2,\quad \forall\theta\in [0,\pi].
\end{equation}
Then there exists a unique $0<\alpha_*\le C_*/c_\omega$ (depending on $s$ and $\omega$), such that for any $\alpha > \alpha_*$, $\rho_{\textnormal{1D}}$ is the unique minimizer of $E_\alpha$ (up to translation), and for any $\alpha<\alpha_*$, $\rho_{\textnormal{1D}}$ is not a minimizer of $E_\alpha$.
\end{theorem}

\begin{remark}
The non-degeneracy condition \eqref{thm_coer_1} is satisfied if $\omega(\theta)>0$ for any $\theta\in[0,\pi]\backslash\{\frac{\pi}{2}\}$ and $\omega''(\frac{\pi}{2})>0$. We point out that finding an exact formula for $\alpha_*$ or approximating it numerically seems to be very hard.
\end{remark}

To prove the theorem, we need the following construction which enables a comparison argument.
\begin{lemma}\label{lem_coer}
For any fixed $0<s<1$, there exists $\Omega_*$ satisfying {\bf (H)}, with $\Omega_*(\frac{\pi}{2})=1$, $\Omega_*\ge 1$ such that the associated potential $W_*(\bx)=|\bx|^{-s}\Omega_*(\theta)+|\bx|^2$ satisfies $\hat{W}_*(\xi)\ge 0$ for any $\xi\ne 0$ and $W_**\rho_{\textnormal{1D}}$ achieves minimum on $\supp\rho_{\textnormal{1D}}$.
\end{lemma}
\begin{proof}
Consider $\Omega_*$ satisfying {\bf (H)} with $0\notin\supp\tilde{\Omega}_*$, $\Omega_*(\frac{\pi}{2})=1$ and $\Omega_*(\theta)=\Omega_*(-\theta)$. We apply Lemma \ref{lem_ab} with $a=0,b=R_1/R_2$, which is allowed since $A$ and $B$ take finite values. Since $\rho_1$ is the minimizer of the energy with the 1D potential $|x|^{-s}+|x|^2$, we see that $W_**\rho_{\textnormal{1D}}$ is constant on $\supp\rho_{\textnormal{1D}}$ . This implies $B=1$ in Lemma \ref{lem_ab}. The constraint $\Omega_*(\frac{\pi}{2})=1$ is equivalent to
\begin{equation}
    2\tau_s\int_0^\pi |\sin\varphi|^{-s}\tilde{\Omega}_*(\varphi)\rd{\varphi} = 1
\end{equation}
by \eqref{lem_FT_2r}. With this constraint satisfied,
\begin{equation}
    A = \tau_s %(1/R_2)^{2+s}
    \int_{0}^\pi \cot^2\varphi|\sin\varphi|^{-s}\tilde{\Omega}_*(\varphi)\rd{\varphi}
\end{equation}
can be made arbitrarily small by taking $|\sin\varphi|^{-s}\tilde{\Omega}_*(\varphi)$ as a mollifier concentrated near $\varphi=\pi/2$. In particular, we can make $A<1$. Then Lemma \ref{lem_ab} gives that $(|\bx|^{-s}\Omega_*(\theta) + Ax_1^2+x_2^2)*\rho_{\textnormal{1D}}$ achieves minimum on $\supp\rho_{\textnormal{1D}}$, and the same is clearly also true if $Ax_1^2$ is replaced by $x_1^2$. Finally, since $\tilde{\Omega}_*\ge 0$, we apply Lemma \ref{lem_EL} and Theorem \ref{thm_LICequiv} to see that $\rho_{\textnormal{1D}}$ is the unique minimizer of the associated interaction energy. This implies $\Omega_*\ge \Omega_*(\frac{\pi}{2})=1$ because otherwise a rotated version of $\rho_{\textnormal{1D}}$ would have smaller energy.
\end{proof}

\begin{proof}[Proof of Theorem \ref{thm_coer}]
Let $\Omega_*$ given by Lemma \ref{lem_coer}, and the associated potential and energy $W_*$ and $E_*$. The assumptions on $\Omega_*$ implies the existence of $C_*>0$ such that $1+C_*|\theta-\frac{\pi}{2}|^2 \ge \Omega_*(\theta)$ for any $\theta\in [0,\pi]$, and equality only holds for $\theta=\frac{\pi}{2}$. Also, as seen in the previous proof, $\rho_{\textnormal{1D}}$ is the unique minimizer of $E_*$.

Since $\Omega(\frac{\pi}{2})=1$ and $\Omega(\theta)\ge 1+C_*|\theta-\frac{\pi}{2}|^2$ for $\theta\in [0,\pi]$, for any compactly supported probability measure $\rho$ we have
\begin{equation}\label{Ecomp}
E[\rho] \ge E_*[\rho] \ge E_*[\rho_{\textnormal{1D}}] = E[\rho_{\textnormal{1D}}]
\end{equation}
by $\Omega \ge \Omega_*$, the minimizing property of $\rho_{\textnormal{1D}}$ for $E_*$, and the fact that $E[\rho_{\textnormal{1D}}]$ only involves the values of $W$ with $\theta=\frac{\pi}{2}$, the latter being the same as those in $W_*$. This shows that $\rho_{\textnormal{1D}}$ is a minimizer of $E$. Furthermore, if $\rho$ is not supported on a vertical line, then the first inequality in \eqref{Ecomp} is strict since $\Omega(\theta)>\Omega_*(\theta)$ whenever $\theta\ne\pi/2$. If $\rho$ is supported on a vertical line, then the uniqueness of energy minimizer for the 1D potential $|x|^{-s}+|x|^2$ shows that the second inequality in \eqref{Ecomp} is strict unless $\rho=\rho_{\textnormal{1D}}$. Therefore we conclude that $\rho_{\textnormal{1D}}$ is the unique minimizer of $E$ (up to translation).

For the statement on $W_\alpha$, we notice that the assumptions for $\omega$ implies that $1+\alpha \omega(\theta)$ satisfies the assumptions on $\Omega$ for the previous part for $\alpha=C_*/c_\omega$, and it follows that $\rho_{\textnormal{1D}}$ is the unique minimizer of $E_\alpha$. The same comparison argument also shows that if $\alpha_1<\alpha_2$ and $\rho_{\textnormal{1D}}$ is a minimizer of $E_{\alpha_1}$, then it is the unique minimizer of $E_{\alpha_2}$. 

By Lemma \ref{lem_FT}, $\tilde{\omega}$ is smooth, and thus $\tilde{\Omega}_\alpha=c_s+\alpha \tilde{\omega}\ge c >0$ if $\alpha$ is sufficiently small. In this case Theorem \ref{thm_ell} shows that $\rho_{\textnormal{1D}}$ is not a minimizer of $E_\alpha$. 

Therefore, we define
\begin{equation}
    \alpha_*:= \inf\{\alpha\ge 0: \text{$\rho_{\textnormal{1D}}$ is the unique minimizer of $E_\alpha$}\},
\end{equation}
that is a positive number, with the property that $\rho_{\textnormal{1D}}$ is the unique minimizer of $E_\alpha$ for any $\alpha>\alpha_*$. For any $\alpha<\alpha_*$, $\rho_{\textnormal{1D}}$ cannot be a minimizer of $E_\alpha$, because otherwise we would get that $\rho_{\textnormal{1D}}$ is the unique minimizer of $E_{(\alpha+\alpha_*)/2}$, contradicting the definition of $\alpha_*$. Therefore the desired properties of $\alpha_*$ are proved.

\end{proof}

\begin{remark}
It is clear that $\rho_{\textnormal{1D}}$ is a minimizer of $E_{\alpha_*}$. In fact, for any $\alpha>\alpha_*$, we have $E_\alpha[\rho]\ge E_\alpha[\rho_{\textnormal{1D}}]$ for any probability measure $\rho$. Sending $\alpha\rightarrow\alpha_*$ gives the conclusion. However, it is not clear whether $\rho_{\textnormal{1D}}$ is the unique minimizer of $E_{\alpha_*}$.
\end{remark}

\subsection{Potentials with degeneracy near $\frac{\pi}{2}$}

\begin{theorem}\label{thm_degen}
Let $0<s<1$, $W$ be given by \eqref{W} with $\Omega$ satisfying {\bf (H)}, $\Omega(\frac{\pi}{2})=1$ and \begin{equation}
    \Omega(\theta)\le 1+C\big|\theta-\frac{\pi}{2}\big|^{\kappa}
\end{equation}
for some $\kappa>2$ and any $\theta$ near $\frac{\pi}{2}$. Then $\rho_{\textnormal{1D}}$ is not a $d_\infty$-local minimizer of the associated energy.

In particular, if $W_\alpha$ is given by \eqref{Walpha} with $\omega$ satisfying {\bf (h)} and \begin{equation}
    \omega(\theta)\le C|\theta-\frac{\pi}{2}|^{\kappa}
\end{equation} 
near $\frac{\pi}{2}$, then $\rho_{\textnormal{1D}}$ is not a $d_\infty$-local minimizer of the associated energy $E_\alpha$ for any $\alpha>0$.
\end{theorem}

\begin{remark}
The last statement complements Theorem \ref{thm_coer}, showing that the assumption \eqref{thm_coer_1} in the latter is sharp in the power.
\end{remark}

\begin{figure}
    \centering
    \includegraphics[width=0.2\textwidth]{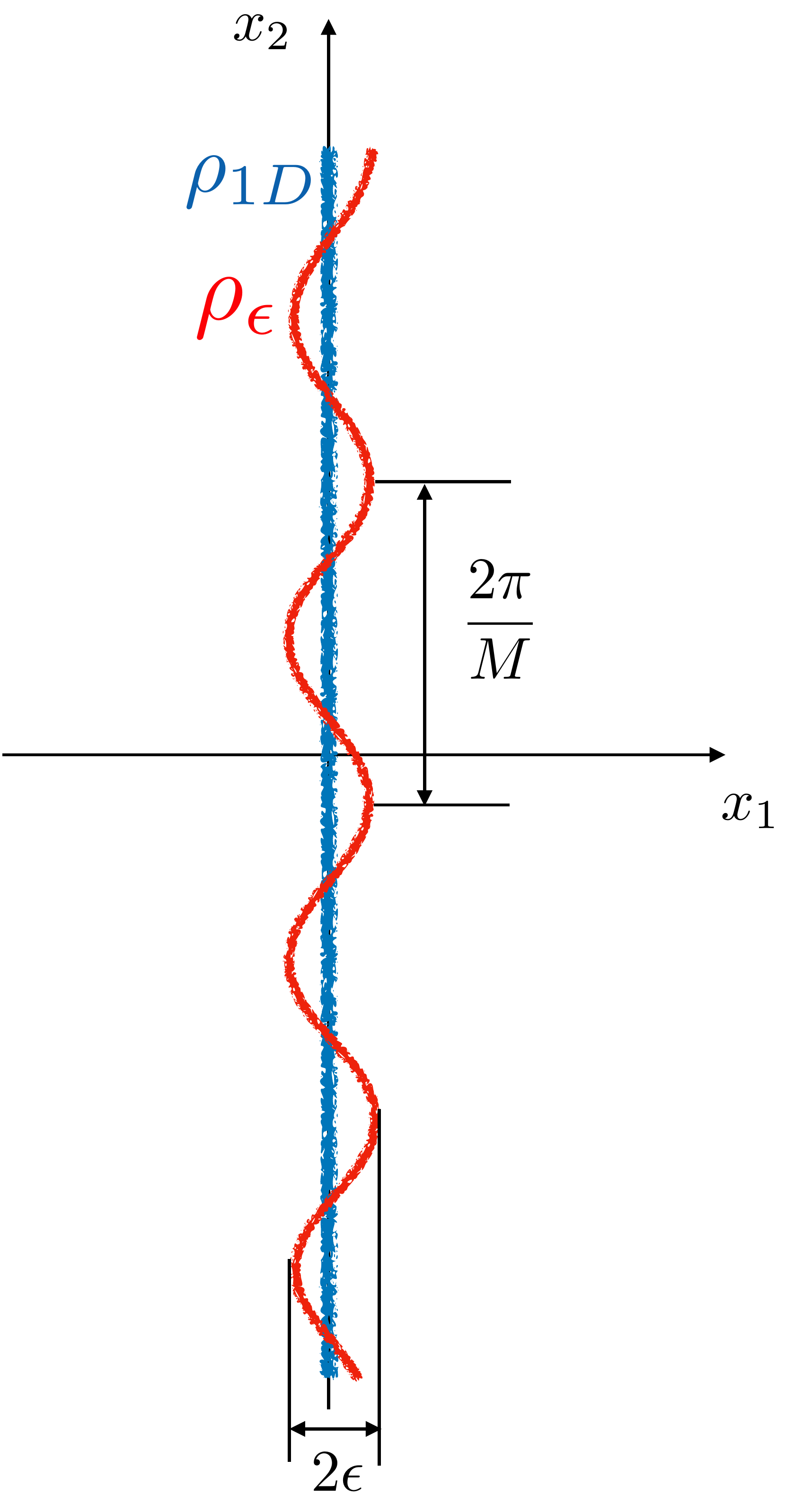}
    \caption{The perturbed distribution $\rho_\epsilon$.}
    \label{fig:perturb}
\end{figure}

\begin{proof}
Let $M$ be a positive integer to be chosen, and consider $\epsilon>0$ small. Define
\begin{equation}
    \rho_\epsilon(\bx) = \rho_1(x_2)\delta(x_1-\epsilon \sin(M x_2))
\end{equation}
as a perturbation of $\rho_{\textnormal{1D}}$. See Figure \ref{fig:perturb} as an illustration. We claim that there holds the asymptotic expansion
\begin{equation}\label{Esin}
    E_{\textnormal{iso}}[\rho_\epsilon] = E_{\textnormal{iso}}[\rho_{\textnormal{1D}}] + c_M\epsilon^2 + O_M(\epsilon^4)
\end{equation}
where $E_{\textnormal{iso}}$ denotes the energy with $\Omega=1$ (the isotropic part), and
\begin{equation}\label{cM}\begin{split}
    c_M= &-\frac{s}{4}\int_{\mathbb{R}}\int_{\mathbb{R}}\Big|\frac{\sin(M x)-\sin(M y)}{x-y}\Big|^2\, |x-y|^{-s}\rho_1(y)\rd{y}\,\rho_1(x)\rd{x} \\ & + \frac{1}{2}\int_{\mathbb{R}}\int_{\mathbb{R}}|\sin (M x)-\sin(M y)|^2\rho_1(y)\rd{y}\,\rho_1(x)\rd{x}.
\end{split}\end{equation}
To prove \eqref{Esin}, we compute the difference of the isotropic energies
\begin{equation}\begin{split}
& E_{\textnormal{iso}}[\rho_\epsilon]-E_{\textnormal{iso}}[\rho_{\text{1D}}] 
\\ = & \frac{1}{2}\int_{\mathbb{R}}\int_{\mathbb{R}} \Big[(|x-y|^2+\epsilon^2|\sin(M x)-\sin(M y)|^2)^{-s/2} \\ 
& -|x-y|^{-s} + \epsilon^2|\sin (M x)-\sin(M y)|^2\Big]\rho_1(y)\rd{y}\rho_1(x)\rd{x} \\
= & \frac{1}{2}\int_{\mathbb{R}}\int_{\mathbb{R}} \Big[\Big(1+\epsilon^2\Big|\frac{\sin(M x)-\sin(M y)}{x-y}\Big|^2\Big)^{-s/2}-1 \Big]|x-y|^{-s}\rho_1(y)\rd{y}\rho_1(x)\rd{x} \\
& +  \frac{\epsilon^2}{2}\int_{\mathbb{R}}\int_{\mathbb{R}} |\sin (M x)-\sin(M y)|^2\rho_1(y)\rd{y}\rho_1(x)\rd{x}\\
\end{split}
\end{equation}
Notice that $|\frac{\sin(M x)-\sin(M y)}{x-y}| \le M$ for any $x\ne y$ by the mean value theorem. Therefore we have the Taylor expansion
\begin{equation}
    \Big(1+\epsilon^2\Big|\frac{\sin(M x)-\sin(M y)}{x-y}\Big|^2\Big)^{-s/2}-1 = -\frac{s}{2}\epsilon^2\Big|\frac{\sin(M x)-\sin(M y)}{x-y}\Big|^2 + O(\epsilon^4)
\end{equation}
where the constant in the last $O(\epsilon^4)$ is independent of $x,y$. Therefore we obtain \eqref{Esin} with the stated coefficient $c_M$.

It is clear that the second integral in \eqref{cM} is no more than 4, and the first integral goes to infinity as $M\rightarrow\infty$ (by observing that $|\frac{\sin(M x)-\sin(M y)}{x-y}| \ge cM$ whenever $|x-y|<c/M$ and $Mx-\pi/2$ is not close to an integer multiple of $\pi$). Therefore, for sufficiently large $M$ (depending only on $s$), we have $c_M<0$.

Next we analyze the energy from the anisotropic part. Denote $\Omega_1 = \Omega-1$, and $E_{\textnormal{ani}}$ as the energy with $\Omega$ replaced by $\Omega-1$. Then $E_{\textnormal{ani}}[\rho_{\textnormal{1D}}]=0$, and
\begin{equation}\begin{split}
E_{\textnormal{ani}}[\rho_\epsilon] 
= & \frac{1}{2}\int_{\mathbb{R}}\int_{\mathbb{R}} (|x-y|^2+\epsilon^2|\sin(M x)-\sin(M y)|^2)^{-s/2} \\
& \qquad\cdot\Omega_1\Big(\tan^{-1}\frac{x-y}{\epsilon(\sin(M x)-\sin(M y))}\Big)\rho_1(y)\rd{y}\,\rho_1(x)\rd{x} \\
\le & \frac{1}{2}\int_{\mathbb{R}}\int_{\mathbb{R}}|x-y|^{-s}\Omega_{1,+}\Big(\tan^{-1}\frac{x-y}{\epsilon(\sin(M x)-\sin(M y))}\Big)\rho_1(y)\rd{y}\,\rho_1(x)\rd{x}
\end{split}
\end{equation}
where $\Omega_{1,+}(\theta):=\max\{\Omega_1(\theta),0\}$. 

\sloppy Since $|\frac{x-y}{\epsilon(\sin(M x)-\sin(M y))}|\ge \frac{1}{M\epsilon}$, we see that either $|\tan^{-1}\frac{x-y}{\epsilon(\sin(M x)-\sin(M y))}-\frac{\pi}{2}|<C\epsilon$ or $|\tan^{-1}\frac{x-y}{\epsilon(\sin(M x)-\sin(M y))}+\frac{\pi}{2}|<C\epsilon$ for small $\epsilon$. Using the assumption  $\Omega_1(\theta)\le C|\theta-\frac{\pi}{2}|^{\kappa}$ and the property of $\Omega_1(\theta)=\Omega_1(\theta+\pi)$, we see that
\begin{equation}\begin{split}
E_{\textnormal{ani}}[\rho_\epsilon] 
\le & C\epsilon^{\kappa}\int_{\mathbb{R}}\int_{\mathbb{R}}|x-y|^{-s}\rho_1(y)\rd{y}\,\rho_1(x)\rd{x} = C\epsilon^{\kappa}.
\end{split}
\end{equation}
Therefore, we conclude
\begin{equation}
    E[\rho_\epsilon]-E[\rho_{\textnormal{1D}}] \le c_M\epsilon^2 + C\epsilon^{\kappa} + O(\epsilon^4)
\end{equation}
with $c_M<0$, which implies that $E[\rho_\epsilon]<E[\rho_{\textnormal{1D}}]$ for all $\epsilon$ sufficiently small because $\kappa>2$. Since $d_\infty(\rho_\epsilon,\rho_{\textnormal{1D}})\le \epsilon$, we see that $\rho_{\textnormal{1D}}$ is not a $d_\infty$-local minimizer of $E$.

\end{proof}

\subsection{Estimate of the width of support}

As a complementary result to Theorem \ref{thm_degen}, we will show that any global minimizer has to have narrow support in $x_1$ for large $\alpha$, even if $\omega$ behaves like $|\theta-\frac{\pi}{2}|^\kappa$ with a large $\kappa$.

\begin{theorem}\label{thm_width}
Assume $0<s<1$, $W_\alpha$ given by \eqref{Walpha} with $\omega$ satisfying {\bf (h)} and 
\begin{equation}\label{thm_width_1}
    \omega(\theta) \ge c|\theta-\frac{\pi}{2}|^\kappa,\quad \forall\theta\in [0,\pi]
\end{equation} 
for some $\kappa\ge 2$. Then, for any minimizer $\rho$ of $E_\alpha$ with zero center of mass,
\begin{equation}
\sup_{\bx\in\supp\rho} |x_1| \le C \alpha^{-1/(\kappa+1)}
\end{equation}
with $C$ depending on $s$ and $\omega$.
\end{theorem}

\begin{proof}
By Lemma \ref{lem_R}, the conclusion is clearly true for $0\le \alpha \le 1$. In the rest of the proof, we will assume $\alpha>1$. Denote $X=\sup_{\bx\in\supp\rho} |x_1|$. Assume $X>C_1\alpha^{-1/(\kappa+1)}$ with $C_1>0$ to be determined. Let $\bx\in\supp\rho$ with $x_1=X$. Then for any $\epsilon>0$, 
\begin{equation}
m_\epsilon:= \int_{x_1>X-\epsilon}\rho\rd{\bx} >0
\end{equation}
We will always consider $\epsilon<X-C_1\alpha^{-1/(\kappa+1)}$, and $\epsilon$ sufficiently small so that $m_\epsilon < 2/3$ (the latter is possible because $\rho$ has zero center of mass).

Denote 
\begin{equation}
\rho_{(1)} = \rho \chi_{x_1\le X-\epsilon},\quad \rho_{(2)} = \frac{\rho_{(1)}}{\int \rho_{(1)}\rd{\bx}}.
\end{equation}
Then $\rho_{(2)}$ is a probability measure. We aim to show that $E_\alpha[\rho_{(2)}] < E_\alpha[\rho]$ for any $\epsilon$ with the above smallness conditions, which would lead to a contradiction.

First notice that the mean-zero condition for $\rho$ and Lemma \ref{lem_R} imply that
\begin{equation}\label{Rst}
\int_{x_1<C_1\alpha^{-1/(\kappa+1)}/2}\rho\rd{\bx} \ge R_*^{-1} C_1\alpha^{-1/(\kappa+1)},\quad R_* = 4\max\{R,C_1\},
\end{equation}
where $R>0$ is such that $\supp\rho\subset \cB(0;R)$ given by Lemma \ref{lem_R} and depending only on $s$. In fact, if this was not true, then combining with $\alpha>1$, one would have \begin{equation}\begin{split}
    \int_{\mathbb{R}^2} x_1\rho(\bx)\rd{\bx} = & \int_{x_1<C_1\alpha^{-1/(\kappa+1)}/2} x_1\rho(\bx)\rd{\bx} + \int_{x_1\ge C_1\alpha^{-1/(\kappa+1)}/2} x_1\rho(\bx)\rd{\bx} \\
    \ge & -R \cdot R_*^{-1} C_1\alpha^{-1/(\kappa+1)} + C_1\alpha^{-1/(\kappa+1)}/2 \cdot (1-R_*^{-1} C_1\alpha^{-1/(\kappa+1)}) \\
    \ge & C_1\alpha^{-1/(\kappa+1)}\Big(-\frac{1}{4}+\frac{1}{2}\cdot\frac{3}{4}\Big)>0\,,
\end{split}\end{equation}
where we used that $x_1\geq -R$ on the first integral, 
contradicting the mean-zero condition for $\rho$.

Then, since $W_\alpha>0$,
\begin{equation}\begin{split}
E_\alpha[\rho]-E_\alpha[\rho_{(1)}] \ge & \int_{x_1>X-\epsilon}\int_{y_1<C_1\alpha^{-1/(\kappa+1)}/2}W_\alpha(\bx-\by)\rho(\by)\rd{\by}\rho(\bx)\rd{\bx} \\
\ge & c \alpha (C_1\alpha^{-1/(\kappa+1)})^\kappa \int_{x_1>X-\epsilon}\int_{y_1<C_1\alpha^{-1/(\kappa+1)}/2}\rho(\by)\rd{\by}\rho(\bx)\rd{\bx} \\
\ge & c R_*^{-1} C_1^{\kappa+1}m_\epsilon  \\
\end{split}\end{equation}
since in the integrand we always have $x_1-y_1 > C_1\alpha^{-1/(\kappa+1)}/2$ and $|\bx-\by| \le C$ (by Lemma \ref{lem_R}, $C$ independent of $\alpha$) and thus, the angles of $\bx-\by$ corresponding to the integration set are bounded away from $\pi/2$. As a consequence, $W_\alpha(\bx-\by) \ge c \alpha (C_1\alpha^{-1/(\kappa+1)})^\kappa$ by the lower bound assumption \eqref{thm_width_1} on $\omega$.

Then notice that $\int \rho_{(1)}\rd{\bx} = 1-m_\epsilon$. Therefore, using $m_\epsilon<2/3$ and Taylor expansion of $(1-m)^{-2}$ at $m=0$, we get
\begin{equation}\begin{split}
E_\alpha[\rho_{(2)}] = & \frac{1}{(1-m_\epsilon)^2}E_\alpha[\rho_{(1)}] \le (1+Cm_\epsilon) E_\alpha[\rho_{(1)}] \le (1+Cm_\epsilon)(E_\alpha[\rho] - c R_*^{-1} C_1^{\kappa+1} m_\epsilon ) \\
\le & E_\alpha[\rho] + \Big(C E_\alpha[\rho]-c R_*^{-1} C_1^{\kappa+1}\Big) m_\epsilon 
\le  E_\alpha[\rho] + \Big(C E_0[\rho_{\textnormal{1D}}]-c R_*^{-1} C_1^{\kappa+1}\Big) m_\epsilon \\
\end{split}\end{equation}
where the last inequality uses the minimizing property of $\rho$ and the fact that $E_\alpha[\rho_{\textnormal{1D}}]=E_0[\rho_{\textnormal{1D}}]$ is finite for $0<s<1$. Recall that $R_*\le 4C_1$ by \eqref{Rst}. By taking $C_1$ sufficiently large (independent of $\alpha,\epsilon$), we obtain $E[\rho_{(2)}]<E[\rho]$, contradicting the minimizing property of $\rho$.

\end{proof}

%%%%%%%%%%%%%%%%%%%%%%%%%%%%%%%%

\section{Intermediate $\alpha$: possible complex behavior}\label{sec_complex}

In this section we analyze the 2D potential generated by 1D vertical distributions. This enables us to determine whether the vertical one dimensional distribution $\rho_{\textnormal{1D}}$ satisfies the first Euler-Lagrange condition. 

\begin{proposition}\label{prop_expan}
Let $0<s<1$ and $W$ be given by \eqref{W} with $\Omega$ satisfying {\bf (H)}. Let
\begin{equation}\label{prop_expan_1}
\rho(\bx) = \psi(x_2) \delta(x_1)
\end{equation}
where $\psi$ is a nonnegative continuous function on $\mathbb{R}$ with compact support and $\int_{\mathbb{R}}\psi(x)\rd{x}=1$. Assume $\psi$ is $C^1$ near a neighborhood of 0. Then
\begin{equation}\label{prop_expan_2}
(W*\rho)(\epsilon,0)-(W*\rho)(0,0) = \frac{1}{2\tau_{2-s}}\tilde{\Omega}(0)\psi(0)\epsilon^{1-s} + O(\epsilon)
\end{equation}
where $\tau_{2-s}$ is a negative number, given by \eqref{calc4}. When $\Omega$ achieves a local minimum at $\pi/2$, then the error term can be improved to $O(\epsilon^{2-s})$.

In particular, if $\tilde{\Omega}(0)>0$ and $\psi(0)>0$, then $\rho$ is not a $d_\infty$-local minimizer of $E$.
\end{proposition}

\begin{proof}We first express the difference of the potentials as
\begin{equation}\begin{split}
& (W*\rho)(\epsilon,0)-(W*\rho)(0,0) \\
= & \int_{\mathbb{R}} \Big((\epsilon^2+y^2)^{-s/2}\Omega\Big(\tan^{-1}\frac{y}{\epsilon}\Big) - \Omega\Big(\frac{\pi}{2}\Big) |y|^{-s} + \epsilon^2 \Big)  \psi(y) \rd{y} \\
= & \epsilon^{1-s}\int_{\mathbb{R}} \Big((1+y^2)^{-s/2}\Omega(\tan^{-1}y) -  |y|^{-s}\Omega\Big(\frac{\pi}{2}\Big) \Big)  \psi(\epsilon y) \rd{y} + \epsilon^2\int_{\mathbb{R}} \psi \rd{y} \\
= & \psi(0)\epsilon^{1-s}\int_{\mathbb{R}} \Big((1+y^2)^{-s/2}\Omega(\tan^{-1}y) - |y|^{-s} \Omega\Big(\frac{\pi}{2}\Big) \Big) \rd{y} + \epsilon^{1-s}\cR + \epsilon^2\int_{\mathbb{R}} \psi \rd{y} \\
\end{split}\end{equation}
where the remainder term $\cR$ is given by
\begin{equation}
    \cR = \int_{\mathbb{R}} \Big((1+y^2)^{-s/2}\Omega(\tan^{-1}y) - |y|^{-s}\Omega\Big(\frac{\pi}{2}\Big)  \Big)  (\psi(\epsilon y)-\psi(0)) \rd{y} \,.
\end{equation}
To calculate the integral in the main term, we use a change of variable $y = \tan\theta$ to get
\begin{equation}\begin{split}
& \int_{\mathbb{R}} \Big((1+y^2)^{-s/2}\Omega(\tan^{-1}y) - |y|^{-s}\Omega\Big(\frac{\pi}{2}\Big)  \Big) \rd{y} \\
= & \int_{\mathbb{R}} (1+y^2)^{-s/2}\Big(\Omega(\tan^{-1}y)-\Omega\Big(\frac{\pi}{2}\Big)\Big)\rd{y} - \Omega\Big(\frac{\pi}{2}\Big)\int_{\mathbb{R}} \Big(|y|^{-s}-(1+y^2)^{-s/2}   \Big) \rd{y}  \\
= & \int_{-\pi/2}^{\pi/2} |\cos\theta|^{s-2}\Big(\Omega(\theta) -  \Omega\Big(\frac{\pi}{2}\Big)\Big)  \rd{\theta} + \Omega\Big(\frac{\pi}{2}\Big) \frac{\gamma_{s-2}}{2} 
=  \frac{1}{2\tau_{2-s}}\tilde{\Omega}(0)
\end{split}\end{equation}
using \eqref{calc5} and \eqref{lem_FT_3}.

To estimate the error term $\cR$, notice that the smoothness of $\Omega$ implies that $|\Omega(\tan^{-1}y)-\Omega(\frac{\pi}{2})| \le C\min\{|y|^{-1},1\}$. Therefore, we estimate as
\begin{equation}\begin{split}
& \Big|(1+y^2)^{-s/2}\Omega(\tan^{-1}y) - |y|^{-s}\Omega\Big(\frac{\pi}{2}\Big)\Big| \\
\le & (1+y^2)^{-s/2}\Big|\Omega(\tan^{-1}y) - \Omega\Big(\frac{\pi}{2}\Big)\Big| + \Omega\Big(\frac{\pi}{2}\Big)\Big(|y|^{-s} - (1+y^2)^{-s/2}\Big) \\
\le & C\Big((1+y^2)^{-s/2}\min\{|y|^{-1},1\} + |y|^{-s}\min\{1,|y|^{-2}\}\Big).
\end{split}\end{equation}
Next notice that $|\psi(\epsilon y)-\psi(0)| \le C\min\{\epsilon |y|, 1\}$ since $\psi$ is $C^1$ at 0 and bounded on $\mathbb{R}$. Therefore
\begin{equation}\begin{split}
|\cR| \le & C \int_{|y|\le 1}\Big((1+y^2)^{-s/2} + |y|^{-s}\Big)  \epsilon |y| \rd{y} 
 + C \int_{1\le |y| \le 1/\epsilon} \Big((1+y^2)^{-s/2}|y|^{-1} + |y|^{-s}|y|^{-2}\}\Big)  \epsilon |y| \rd{y} \\
& + C \int_{|y|>1/\epsilon} \Big((1+y^2)^{-s/2}|y|^{-1} + |y|^{-s}|y|^{-2}\Big) \rd{y} \\
\le & C\epsilon + C\epsilon^s + C\epsilon^s = O(\epsilon^s)
\end{split}\end{equation}
and the conclusion in the general case follows. When $\pi/2$ is a local minimum point of $\Omega$, we have the improved estimate $|\Omega(\tan^{-1}y)-\Omega(\frac{\pi}{2})| \le C\min\{|y|^{-2},1\}$ since $\Omega'(\pi/2)=0$. Therefore we get $|\cR| \le C\epsilon$ and the conclusion follows.
\end{proof}

Let us interpret the result of the Proposition \ref{prop_expan}. It is clear that this result can be used with translation or rotation. Also, if $\rho$ is the sum of \eqref{prop_expan_1} and a measure whose support does not contain 0, then \eqref{prop_expan_2} is still true for $\epsilon$ small enough since the measure away from 0 can only contribute $O(\epsilon^2)$ to $(W*\rho)(\epsilon,0)-(W*\rho)(0,0)$. In particular, if $\tilde{\Omega}(\varphi)>0$, then the support of any $d_\infty$-local minimizer cannot contain any isolate segment along the $\vec{e}_\varphi^\perp$ direction with a $C^1$ profile on it.

This is particularly interesting if $\tilde{\Omega}(0)>0$ but $\tilde{\Omega}(\varphi)=\tilde{\Omega}(-\varphi)<0$ for some other angle $\varphi$. In this case, Proposition \ref{prop_expan} implies that $\rho_{\textnormal{1D}}$ cannot be a $d_\infty$-local minimizer. One expects to see the support of an energy minimizer to present a zigzag behavior, consisting of segments with angles $\varphi+\frac\pi{2}$ and $-\varphi+\frac\pi{2}$ with $\tilde{\Omega}(\varphi)=\tilde{\Omega}(-\varphi)<0$. For the parametrized potential $W_\alpha$ in \eqref{Walpha}, if $c_s\tilde{\omega}(\varphi)$, the angle function for the Fourier transform of $|\bx|^{-s}\omega(\varphi)$, achieves minimum at some $\varphi\ne 0$, then such zigzag happens for $\alpha_L<\alpha<\alpha_{L,0}$, where $\alpha_L$ is defined in \eqref{alphaL}, and $\alpha_{L,0}$ (depending on $s$ and $\omega$) is defined as
\begin{equation}\label{alphaL0}
    \alpha_{L,0} := -\frac{1}{\tilde{\omega}(0)}.
\end{equation}
See Sections \ref{sec_example2} and \ref{sec_example0} as examples.

\begin{remark}
If $\rho$ as in \eqref{prop_expan_1} is a steady state with $\tilde{\Omega}(0)<0$, then \eqref{prop_expan_2} basically tells that $W*\rho$ achieves local minimum on $\supp\rho$, which is one of the Euler-Lagrange conditions. However, this is not sufficient to guarantee that $\rho$ is a $d_\infty$-local minimizer.

In fact, if $\omega$ satisfies the assumption of Theorem \ref{thm_degen}, it may happen that the minimum of $\omega$ is only achieved at $\omega(\frac{\pi}{2})=0$. In this case, the formula \eqref{lem_FT_3} guarantees $\tilde{\omega}(0)<0$, which implies $\tilde{\Omega}_\alpha(0)<0$ for sufficiently large $\alpha$. In this case, Proposition \ref{prop_expan} applies to $\rho_{\textnormal{1D}}$ for $\Omega_\alpha$ (at least at the points $(0,x_2),\,-R_1<x_2<R_1$), but $\rho_{\textnormal{1D}}$ is actually not a $d_\infty$-local minimizer by Theorem \ref{thm_degen}.
\end{remark}

%%%%%%%%%%%%%%%%%%%%%%%%%%%%%%%%%%%%%%%%%%%%%%%%

\section{The $s=0$ (logarithmic potential) case}\label{sec_s0}

In this section we consider the logarithmic anisotropic potential 
\begin{equation}\label{Wlog}
    W_{\log}(\bx) = -\ln |\bx| + \Omega(\theta) + |\bx|^2
\end{equation}
with $\Omega$ satisfying {\bf (H)}. The following lemma gives a viewpoint of $W_{\log}$ as a limit of potentials like \eqref{W} as $s\rightarrow 0^+$.
\begin{lemma}\label{lem_FTlog}
Let $\Omega$ satisfy {\bf (H)} and $W_{\log}$ be given by \eqref{Wlog} . Define
\begin{equation}\label{Ws}
    W_{(s)}(\bx) = |\bx|^{-s}\Big(\frac{1}{s} + \Omega(\theta)\Big) - \frac{1}{s} + |\bx|^2,\quad 0<s<1.
\end{equation}
Then $W_{(s)}(\bx)\rightarrow W_{\log}(\bx)$ uniformly on any compact subset of $\mathbb{R}^2$ not containing 0. Furthermore,
\begin{equation}\label{lem_FTlog_1}
    \hat{W}_{\log}(\xi) = |\xi|^{-2}\tilde{\Omega}_{\log}(\varphi),\quad \forall \xi\ne 0
\end{equation}
where $\tilde{\Omega}_{\log}$ is given by a uniform-in-$\varphi$ limit
\begin{equation}\label{lem_FTlog_2}
    \tilde{\Omega}_{\log}(\varphi):=\lim_{s\rightarrow 0^+}\tilde{\Omega}_{(s)}(\varphi)=-(2\pi)^{-2}\pv\int_{-\pi}^{\pi} |\cos(\varphi-\theta)|^{-2}\Big(\Omega(\theta)-\Omega\Big(\varphi+\frac{\pi}{2}\Big)\Big)\rd{\theta} + (2\pi)^{-1}.
\end{equation}
Here $\tilde{\Omega}_{(s)}(\varphi)$ denotes the angle function of the Fourier transform of $W_{(s)}$ given by \eqref{lem_FT_1}. There holds the inversion formula
\begin{equation}\label{lem_FTlog_3}
    \Omega(\theta) = -\int_{-\pi}^\pi \ln|\cos(\varphi-\theta)|\tilde{\Omega}_{\log}(\varphi)\rd{\varphi} + \textnormal{constant}.
\end{equation}
\end{lemma}

\begin{remark}\label{rem_log}
Notice that the operator $\Omega\mapsto\tilde{\Omega}_{\log}$ as in \eqref{lem_FTlog_2} commutes with translation, and maps constant functions to $(2\pi)^{-1}$. Therefore $\tilde{\Omega}_{\log}$ always satisfies $\int_{-\pi}^\pi\tilde{\Omega}_{\log}(\varphi)\rd{\varphi}=1$, the last quantity representing the amount of logarithmic potential $-\ln|\bx|$ contained in $W_{\log}$. The formula \eqref{lem_FTlog_3} allows one to construct $\Omega$ from a given smooth $\tilde{\Omega}_{\log}$ with $\int_{-\pi}^\pi\tilde{\Omega}_{\log}(\varphi)\rd{\varphi}=1$ modulo a constant. Notice that the constants is irrelevant for minimizing the interaction energy.
\end{remark}

\begin{proof}
To see the convergence $W_{(s)}(\bx)\rightarrow W_{\log}(\bx)$ for $\bx\ne 0$, first notice that $\lim_{s\rightarrow 0^+}|\bx|^{-s}\Omega(\theta) = \Omega(\theta)$. Also, a Taylor expansion gives
\begin{equation}
    \frac{|\bx|^{-s}-1}{s} = \frac{e^{-s\ln|\bx|}-1}{s} = \frac{1-s\ln|\bx|+\frac{s^2y^2}{2}-1}{s} = -\ln|\bx| + \frac{s y^2}{2}
\end{equation}
where $y$ is between $\ln|\bx|$ and 0, and thus $\lim_{s\rightarrow 0^+}\frac{|\bx|^{-s}-1}{s} = -\ln|\bx|$. Both convergences are uniform on any compact subset of $\mathbb{R}^2$ not containing 0.

For any $0<s<1$ and $0<\epsilon<1$, we first estimate
\begin{equation}\begin{split}
    \|W_{(s)}\|_{L^2(\cB(0;\epsilon))}^2 \le & C\int_{\cB(0;\epsilon)}|\bx|^{-s}\rd{\bx} + \frac{C}{s^2}\int_{\cB(0;\epsilon)}(|\bx|^{-s}-1)^2\rd{\bx} = C\epsilon^{2-s} + \frac{C}{s^2}\int_0^\epsilon (r^{-s}-1)^2r\rd{r} \\
    = & C\epsilon^{2-s} + \frac{C}{s^2}\Big(\frac{\epsilon^{2-2s}}{2-2s} - \frac{2\epsilon^{2-s}}{2-s} + \frac{\epsilon^{2}}{2}\Big) = C\epsilon^{2-s} + \frac{C\epsilon^2}{s^2}\Big(\frac{e^{-2s\ln\epsilon}}{1-s} - \frac{2e^{-s\ln\epsilon}}{1-s/2} + 1\Big) \\
\end{split}\end{equation}
and we have the limit
\begin{equation}
    \lim_{s\rightarrow 0^+}\frac{1}{s^2}\Big|\frac{e^{-2s\ln\epsilon}}{1-s} - \frac{2e^{-s\ln\epsilon}}{1-s/2} + 1\Big| \le C\ln^2\epsilon
\end{equation}
by Taylor expansion at $s=0$. Therefore we obtain
\begin{equation}\label{WsL2}
    \limsup_{s\rightarrow 0^+}\|W_{(s)}\|_{L^2(\cB(0;\epsilon))}^2 \le C\epsilon^2\ln^2\epsilon
\end{equation}
Combined with the previously shown convergence, we see that $W_{(s)}\rightarrow W_{\log}$ in the space of tempered distributions, which implies $\hat{W}_{(s)}\rightarrow \hat{W}_{\log}$ in the same sense. By \eqref{lem_FT_1} and \eqref{lem_FT_3}, we have
\begin{equation}
    \hat{W}_{(s)}(\xi) = |\xi|^{-2+s}\tilde{\Omega}_{(s)}(\varphi),\quad \xi\ne 0
\end{equation}
with
\begin{equation}\label{tomegas}
    \tilde{\Omega}_{(s)}(\varphi) = \tau_{2-s}\int_{-\pi}^{\pi} |\cos(\varphi-\theta)|^{-2+s}\Big(\Omega(\theta)-\Omega\Big(\varphi+\frac{\pi}{2}\Big)\Big)\rd{\theta} + c_s\Big(\frac{1}{s}+\Omega\Big(\varphi+\frac{\pi}{2}\Big)\Big).
\end{equation}
As $s\rightarrow 0^+$, we have $\tau_{2-s}\rightarrow \tau_2 = -(2\pi)^{-2}$, $c_s\sim (2\pi)^{-1}s$ from the explicit formulas \eqref{calc4} and \eqref{calc2} as $\Gamma(s)\sim \tfrac1s$ as $s\rightarrow 0^+$. To take the limit of the integral in \eqref{tomegas}, one can cutoff the domain at $[\varphi+\frac{\pi}{2}-\epsilon,\varphi+\frac{\pi}{2}+\epsilon]$ and use a Taylor expansion of $\Omega$ inside this interval. The linear term from the Taylor expansion makes no contribution because $|\cos(\varphi-\theta)|^{-2+s}$ is symmetric around $\varphi+\frac{\pi}{2}$. In this way, we get the convergence $\lim_{s\rightarrow 0^+}\tilde{\Omega}_{(s)}(\varphi) =  \tilde{\Omega}_{\log}(\varphi)$ given as in \eqref{lem_FTlog_2} as a principal value integral, uniform in $\varphi$. It follows that $\lim_{s\rightarrow 0^+}\hat{W}_{(s)}(\xi)= |\xi|^{-2}\tilde{\Omega}_{\log}(\varphi)$, uniformly on any compact subset of $\mathbb{R}^2$ not containing 0. Combined with the distributional convergence $\hat{W}_{(s)}\rightarrow \hat{W}_{\log}$, we obtain \eqref{lem_FTlog_1}.

To see \eqref{lem_FTlog_3}, we apply \eqref{lem_FT_2r} to $\tilde{\Omega}_{(s)}$ and obtain
\begin{equation}\begin{split}
    \frac{1}{s} + \Omega(\theta) = & \tau_s\int_{-\pi}^\pi |\cos(\varphi-\theta)|^{-s}\tilde{\Omega}_{(s)}(\varphi)\rd{\varphi} 
    =  s\tau_s\int_{-\pi}^\pi \frac{|\cos(\varphi-\theta)|^{-s}-1}{s}\tilde{\Omega}_{(s)}(\varphi)\rd{\varphi} + \tau_s \int_{-\pi}^\pi\tilde{\Omega}_{(s)}(\varphi)\rd{\varphi}
\end{split}\end{equation}
i.e.,
\begin{equation}
    \Omega(\theta) = s\tau_s\int_{-\pi}^\pi \frac{|\cos(\varphi-\theta)|^{-s}-1}{s}\tilde{\Omega}_{(s)}(\varphi)\rd{\varphi} + C(s)
\end{equation}
for some constant $C(s)$ depending on $s$. As $s\rightarrow 0^+$, we have the limit $s\tau_s\rightarrow 1$ and the uniform-in-$\varphi$ limit $\tilde{\Omega}_{(s)}\rightarrow \tilde{\Omega}_{\log} $. We also have the limit $\frac{|\cos(\varphi-\theta)|^{-s}-1}{s}\rightarrow -\ln|\cos(\varphi-\theta)|$ in $L^2_\varphi$, by an estimate similar to \eqref{WsL2}. Then it is straightforward to pass to the limit in the above integral and obtain \eqref{lem_FTlog_3}.
\end{proof}

\begin{remark}
Notice that the previous result only gives us the pointwise values of the Fourier transform of the potential $\Omega(\theta)$ away from the origin. This information is enough for our purposes.
\end{remark}

For the interaction energy $E_{\log}$ associated to $W_{\log}$, one can also use the approximation argument in \cite[Lemma 2.5]{carrilloshu21} to justify the analogue of \eqref{EFT}:
\begin{equation}\label{EFTlog}
    2E_{\log}[\mu] = \int_{\mathbb{R}^2} |\xi|^{-2}\tilde{\Omega}_{\log}(\varphi)|\hat{\mu}(\xi)|^2 \rd{\xi}
\end{equation}
for any compactly supported signed measure $\mu$ with $\int_{\mathbb{R}^2}\mu(\bx)\rd{\bx}=\int_{\mathbb{R}^2}\bx\mu(\bx)\rd{\bx}=0$ and $E_{\log}[|\mu|]<\infty$. Then one can show that Theorem \ref{thm_LICequiv} also holds for $W_{\log}$.

\begin{corollary}
Let $W_{\log}$ is given by \eqref{Wlog} with $\Omega$ satisfying {\bf (H)}. Then $W_{\log}$ has the LIC property if and only if $\tilde{\Omega}_{\log}$ given as in \eqref{lem_FTlog_2} is nonnegative.
\end{corollary}

We now first extend Lemma \ref{lem_ab} to the logarithmic case by taking the limit $s\rightarrow 0^+$. For fixed $a,b\in(0,\infty)$, we notice that $\supp\rho_{a/R_2,b/R_2}$ is an ellipse with axes $a$ and $b$:
\begin{equation}
    \cB(0;a,b):=\Big\{(x_1,x_2):\frac{x_1^2}{a^2}+\frac{x_2^2}{b^2}\le 1\Big\}
\end{equation}
independent of $s$. We will pass to the limit as $s\rightarrow 0^+$ for the potential generated by $\rho_{a/R_2,b/R_2}$, which converges to a constant multiple of $\chi_{\cB(0;a,b)}$.
\begin{lemma}\label{lem_ablog}
Assume $\Omega$ satisfies {\bf (H)} and $a,b\in(0,\infty)$. Then
\begin{equation}\label{lem_ablog_0}
(-\ln|\bx|+\Omega(\theta)+A x_1^2 + B x_2^2 + 2D x_1x_2)*\frac{1}{|\cB(0;a,b)|}\chi_{\cB(0;a,b)} =  \textnormal{constant},\quad \bx\in \cB(0;a,b)
\end{equation}
where $A,B,D$ are given by 
\begin{equation}\label{lem_ablog_1}\begin{split}
\begin{pmatrix}A(a,b) \\
B(a,b) \\
D(a,b)
\end{pmatrix} 
= & \int_{-\pi}^\pi (a^2\cos^2\varphi + b^2\sin^2\varphi)^{-1}\begin{pmatrix}\cos^2\varphi \\
\sin^2\varphi \\
\cos\varphi\sin\varphi
\end{pmatrix}\, \tilde{\Omega}_{\log}(\varphi)\rd{\varphi}.\\
\end{split}\end{equation}
Furthermore, if $\tilde{\Omega}_{\log}\ge 0$, then $(-\ln|\bx|+\Omega(\theta)+A x_1^2 + B x_2^2 + 2D x_1x_2)*\frac{1}{|\cB(0;a,b)|}\chi_{\cB(0;a,b)}$ achieves its minimal value
on $\cB(0;a,b)$.

If $\tilde{\Omega}_{\log}\ge  0$ and $a=0,b>0$, then the same is true provided that the integral in the expression of $A$ is finite. If $\tilde{\Omega}_{\log}\ge  0$ and $a>0,b=0$, then the same is true provided that the integral in the expression of $B$ is finite. 
\end{lemma}

\begin{remark}\label{rem_Bab}
Here $\frac{1}{|\cB(0;0,b)|}\chi_{\cB(0;0,b)}$ is understood as the weak limit of $\frac{1}{|\cB(0;a,b)|}\chi_{\cB(0;a,b)}$ as $a\rightarrow 0^+$. It is given by the formula $C\delta(x_1)(b^2-x_2^2)^{1/2}\chi_{|x_2|\le b}$ with a normalization factor $C$. Similar for $\frac{1}{|\cB(0;a,0)|}\chi_{\cB(0;a,0)}$.
\end{remark}

\begin{proof}
We first treat the case $a,b>0$. We apply Lemma \ref{lem_ab} with $|\bx|^{-s}\Omega(\theta)$ replaced by  $|\bx|^{-s}(\frac{1}{s}+\Omega(\theta))-\frac{1}{s}$ and $\rho_{a,b}$ replaced by $\rho_{a/R_2,b/R_2}$, and obtain
\begin{equation}\label{abR2}\begin{split}
\Big(|\bx|^{-s}\Big(\frac{1}{s}+\Omega(\theta)\Big)  -\frac{1}{s}+A_{(s)} x_1^2 + B_{(s)} x_2^2 + 2D_{(s)} x_1x_2\Big) & *\rho_{a/R_2,b/R_2} =  \textnormal{constant},\\
& \bx\in \supp\rho_{a/R_2,b/R_2}=\cB(0;a,b)
\end{split}\end{equation}
with
\begin{equation}\label{abR2_1}\begin{split}
\begin{pmatrix}A_{(s)}(a,b) \\
B_{(s)}(a,b) \\
D_{(s)}(a,b)
\end{pmatrix} = & \tau_s (R_1/R_2)^{2+s}\int_{-\pi}^\pi ((a/R_2)^2\cos^2\varphi + (b/R_2)^2\sin^2\varphi)^{-(2+s)/2}\begin{pmatrix}\cos^2\varphi \\
\sin^2\varphi \\
\cos\varphi\sin\varphi
\end{pmatrix}\, \tilde{\Omega}_{(s)}(\varphi)\rd{\varphi}\\ 
= & \tau_s R_1^{2+s}\int_{-\pi}^\pi (a^2\cos^2\varphi + b^2\sin^2\varphi)^{-(2+s)/2}\begin{pmatrix}\cos^2\varphi \\
\sin^2\varphi \\
\cos\varphi\sin\varphi
\end{pmatrix}\, \tilde{\Omega}_{(s)}(\varphi)\rd{\varphi}\\
\end{split}\end{equation}
where $\tilde{\Omega}_{(s)}$ is given by \eqref{tomegas}. As $s\rightarrow 0^+$, we notice that $\tau_s R_1^{2+s}\rightarrow 1$ using the the formula in Remark \ref{rem_ab}. Therefore the limit of $A_{(s)},B_{(s)},D_{(s)}$ are given by \eqref{lem_ablog_1}, by the limit \eqref{lem_FTlog_2}.

Similar to the previous proof, we may show that the potential $(|\bx|^{-s}(\frac{1}{s}+\Omega(\theta))-\frac{1}{s}+A_{(s)} x_1^2 + B_{(s)} x_2^2 + 2D_{(s)} x_1x_2)$ in \eqref{abR2} converges to $(-\ln|\bx|+\Omega(\theta)+A x_1^2 + B x_2^2 + 2D x_1x_2)$ in $L^2$ on any compact set, as $s\rightarrow 0^+$. Also, $\rho_{a/R_2,b/R_2}$ converges to $\frac{1}{|\cB(0;a,b)|}\chi_{\cB(0;a,b)}$ in $L^2(\cB(0;a,b))$ since \eqref{rho2} gives the explicit formula $\rho_{a/R_2,b/R_2}(\bx) = C (1-\frac{x_1^2}{a^2}-\frac{x_2^2}{b^2})^{s/2}$ and its total mass is 1. Therefore we see that the LHS of \eqref{abR2} converges pointwisely to
\begin{equation}
    (-\ln|\bx|+\Omega(\theta)+A x_1^2 + B x_2^2 + 2D x_1x_2)*\frac{1}{|\cB(0;a,b)|}\chi_{\cB(0;a,b)}.
\end{equation}
Therefore we obtain \eqref{lem_ablog_0} from \eqref{abR2}. 

If we further assume $\tilde{\Omega}_{\log}\ge c > 0$, then the uniform-in-$\varphi$ limit in \eqref{lem_FTlog_2} shows that $\tilde{\Omega}_{(s)}\ge 0$ for sufficiently small $s>0$. Then, the same application of Lemma \ref{lem_ab} gives that the LHS of \eqref{abR2} achieves minimum on $\cB(0;a,b)$. Then taking the limit $s\rightarrow 0^+$ gives the conclusion that $(-\ln|\bx|+\Omega(\theta)+A x_1^2 + B x_2^2 + 2D x_1x_2)*\frac{1}{|\cB(0;a,b)|}\chi_{\cB(0;a,b)}$ achieves its minimal value
on $\cB(0;a,b)$. If we only assume $\tilde{\Omega}_{\log}\ge 0$, then the same is true because one can apply the previous result to $t\Omega,\,0<t<1$ (whose strict positivity of Fourier transform is guaranteed by \eqref{lem_FTlog_2}) and pass to the limit $t\rightarrow 1^-$.

For the case $\tilde{\Omega}_{\log}\ge 0$ and $a=0,b>0$, the conclusion can be obtained in the same way as before if $0\notin \supp \tilde{\Omega}_{\log}$. In fact, in this case $A_{(s)},B_{(s)},D_{(s)}$ are finite, and thus Lemma \ref{lem_ab} still gives \eqref{abR2}. For every fixed $x_1$, the potential $\Big(|\bx|^{-s}\Big(\frac{1}{s}+\Omega(\theta)\Big)  -\frac{1}{s}+A_{(s)} x_1^2 + B_{(s)} x_2^2 + 2D_{(s)} x_1x_2\Big)$, as a function of $x_2$, also admit the uniform-on-compact-set convergence away from the origin  and a uniform-in-$s$ $L^2$ estimate near the origin. Therefore, by viewing $\rho_{0,b/R_2}(\bx)=C\delta(x_1)(b^2-x_2^2)^{(1+s)/2}$ as a function of $x_2$, one can pass to the limit $s\rightarrow 0^+$ and obtain the same conclusion. 

For general $\tilde{\Omega}_{\log}\ge  0$ and $a=0,b>0$, one can proceed similar to the proof of Lemma \ref{lem_ab}. In fact, we approximate $\tilde{\Omega}_{\log}$ by an increasing sequence of nonnegative smooth functions $\tilde{\Omega}_{\log,n}$ with $0\notin \supp \tilde{\Omega}_{\log,n}$. If we denote $m_n:=\int\tilde{\Omega}_{\log,n}\rd{\varphi}$, then $\{m_n\}$ is an increasing sequence inside $(0,1)$ with $\lim_{n\rightarrow\infty}m_n=\int\tilde{\Omega}_{\log}\rd{\varphi}=1$. Due to Remark \ref{rem_log}, we cannot apply \eqref{lem_FTlog_3} to $\tilde{\Omega}_{\log,n}$, but we can instead apply it to $\frac{1}{m_n}\tilde{\Omega}_{\log,n}$ and construct the corresponding $\frac{1}{m_n}\Omega_n$ as
\begin{equation}
    \frac{1}{m_n}\Omega_n(\theta) = -\int_{-\pi}^\pi \ln|\cos(\varphi-\theta)|\frac{1}{m_n}\tilde{\Omega}_{\log,n}(\varphi)\rd{\varphi}.
\end{equation}
In other words, the potential $-m_n\ln|\bx|+\Omega_n(\theta)$ has Fourier transform $\tilde{\Omega}_{\log,n}(\xi)$ away from $\xi=0$. 

Since each potential $-m_n\ln|\bx|+\Omega_n(\theta)$ verifies $0\notin \supp \tilde{\Omega}_{\log,n}$, then it  satisfies the desired conclusion that 
\begin{equation}
    (-m_n\ln|\bx|+\Omega_n(\theta)+ A_n x_1^2 +  B_n x_2^2 + 2  D_n x_1x_2)*\frac{1}{|\cB(0;0,b)|}\chi_{\cB(0;0,b)}
\end{equation}
achieves its minimal value on $\cB(0;0,b)$, where $A_n,B_n,D_n$ are obtained from \eqref{lem_ablog_1} using $\tilde{\Omega}_{\log,n}$. The same holds if one adds a constant to the potential. Since \eqref{lem_FTlog_3} has a nonnegative convolution kernel $-\ln|\cos(\varphi-\theta)|$, we see that $\Omega_n(\theta)$ is increasing for each fixed $\theta$ and converges to $\Omega(\theta)$. Also, $\{m_n\}$ is increasing and converges to 1.  Therefore, on any ball $\cB(0;R)$, the potential $-m_n\ln|\bx|+\Omega_n(\theta)+m_n\ln R$ is pointwise increasing in $n$. This allows us to pass to the $n\rightarrow \infty$ limit by the monotone convergence theorem in the relation \eqref{lem_ablog_0} for $-m_n\ln|\bx|+\Omega_n(\theta)+m_n\ln R$ and obtain the conclusion for $-\ln|\bx|+\Omega(\theta)$, combining with $(A_n,B_n,D_n)\rightarrow(A,B,D)$ by the dominated convergence theorem since $A,B<\infty$ by assumption.

The case $\tilde{\Omega}_{\log}\ge  0$ and $a>0,b=0$ can be treated similarly.

\end{proof}

Denote $\cB(0;a,b,\eta)$ as the counterclockwise rotation of $\cB(0;a,b)$ around the origin by the angle $\eta$. In other words, $\chi_{\cB(0;a,b,\eta)}(\bx) = \chi_{\cB(0;a,b)}(\cR_{-\eta}\bx)$.
We notice that an analogue of Lemma \ref{lem_abf} for the formula \eqref{lem_ablog_1} can be proved in a similar way. Therefore, using a similar proof as Theorem \ref{thm_ell}, we obtain the following result.

\begin{theorem}
Let $W_{\log}$ be given by \eqref{Wlog} with $\Omega$ satisfying {\bf (H)}. Let $\tilde{\Omega}_{\log}$ be given by \eqref{lem_FTlog_1} with $\tilde{\Omega}_{\log}\ge 0$. Then exactly one of the following holds (up to translation):
\begin{itemize}
\item There exists a unique tuple $(a,b,\eta)\in (0,\infty)^2\times[0,\pi/2)$ such that $\frac{1}{|\cB(0;a,b)|}\chi_{\cB(0;a,b,\eta)}$
is the unique minimizer of $E_{\log}$.
\item There exists a unique pair $(b,\eta)\in(0,\infty)\times[0,\pi)$ such that $\frac{1}{|\cB(0;0,b)|}\chi_{\cB(0;0,b,\eta)}$ is the unique minimizer of $E_{\log}$.
\end{itemize}
If $\tilde{\Omega}_{\log}\ge c >0$, then item 1 must happen.
\end{theorem}

As mentioned in the introduction, the previous result generalizes \cite{MMSRV21-2}. Notice that no smallness assumption is needed on $\Omega$ and the condition $\tilde{\Omega}_{\log}$ is sharp due to Theorem \ref{thm_LICequiv}. Finally we give an analogue of Theorem \ref{thm_coer} for logarithmic potentials.

\begin{theorem}\label{thm_coerlog}
There exists a constant $C_*$ such that the following holds. Let $W_{\log}$ be given by \eqref{Wlog} with $\Omega$ satisfying {\bf (H)} and 
\begin{equation}
    \Omega(\theta)\ge \Omega(\frac{\pi}{2}) + C_*\big|\theta-\frac{\pi}{2}\big|^2,\quad \forall\theta\in [0,\pi]. 
\end{equation}
Then  $\rho_{\textnormal{1D}}$ is the unique minimizer of $E_{\log}$ (up to translation).

Assume $W_{\log,\alpha}$ is given by \eqref{Wlog} with $\Omega(\theta)=\alpha\omega(\theta)$, $\omega$ satisfying {\bf (h)} and 
\begin{equation}\label{thm_coerlog_1}
    \omega(\theta)\ge c_\omega\big|\theta-\frac{\pi}{2}\big|^2,\quad \forall\theta\in [0,\pi].
\end{equation}
Then there exists a unique $0<\alpha_*\le C_*/c_\omega$ (depending on $\omega$), such that for any $\alpha > \alpha_*$, $\rho_{\textnormal{1D}}$ is the unique minimizer of $E_{\log,\alpha}$ (up to translation), and for any $\alpha<\alpha_*$, $\rho_{\textnormal{1D}}$ is not a minimizer of $E_{\log,\alpha}$.
\end{theorem}
Here $\rho_{\textnormal{1D}}$ is still given by \eqref{rho1D} (with $s=0$), where $\rho_1(x)=C_1(R_1^2-x^2)^{1/2}$ is the unique minimizer for the 1D interaction potential $-\ln|x|+|x|^2$. Notice that $-\ln|x|+|x|^2=\lim_{s\rightarrow 0^+} \frac{|x|^{-s}-1}{s}+|x|^2$. Therefore the constants in $\rho_1$ are $R_1=1$, $C_1=2/\pi$, which can be easily derived as the limit as $s\rightarrow 0^+$ of $R_1(s)s^{1/(-s-2)}$ and $C_1(s)s$ given by \eqref{R1C1}. 

Since the proof is similar to Theorem \ref{thm_coer}, we only give a sketch of the proof. Using \eqref{lem_FTlog_3}, one can first construct $\Omega_*$, such that $W_{\log,*}:=-\ln|\bx|+\Omega_*(\theta)+|\bx|^2$ satisfies $\hat{W}_*(\xi)\ge 0$ for any $\xi\ne 0$ and $W_**\rho_{\textnormal{1D}}$ achieves minimum on $\supp\rho_{\textnormal{1D}}$, similar to Lemma \ref{lem_coer}. This can be done by taking a smooth $\tilde{\Omega}_{\log,*}$, sufficiently concentrated near $\varphi=\pi/2$, with the properties $0\notin\supp\tilde{\Omega}_{\log,*}$, $\int_{-\pi}^\pi\tilde{\Omega}_{\log,*}\rd{\varphi}=1$. Then the same comparison argument as in the proof of Theorem \ref{thm_coer} gives the proof of Theorem \ref{thm_coerlog}.

%%%%%%%%%%%%%%%%%%%%%%%%%%%%%%%%%%%%%%%%%%%%%%%

\section{The range $1\leq s<2$}\label{sec_s12}

Finally, we discuss further the range $1\leq s<2$ for the potential $W$ in \eqref{W}. We remind the reader that we were able to compute its Fourier transform in Lemma \eqref{lem_FT}. We observe that for $1<s<2$ the convolution kernel in \eqref{lem_FT_2} is strictly positive, while for $s=1$ we have $\tilde{\Omega}(\varphi) = \Omega(\varphi+\frac{\pi}{2})$ by \eqref{lem_FT_3}. Combining with Theorem \ref{thm_LICequiv}, we obtain the following lemma.

\begin{lemma}
Let $1\le s < 2$ and $W$ be given by \eqref{W} with $\Omega$ satisfying {\bf (H)}. Then 
\begin{equation}
\hat{W}(\xi)>0,\quad \forall \xi\ne 0,
\end{equation}
i.e., $W$ has LIC.
\end{lemma}

Then we proceed to study the unique energy minimizer. For simplicity, we will add the extra symmetry condition $\Omega(\theta)=\Omega(-\theta)$. We will show that the minimizer is given by some $\rho_{a,b}$ defined in \eqref{rhoab} with $a,b\in (0,\infty)$. 

\begin{lemma}\label{lem_ab2}
Let $1\le s < 2$ and $W$ be given by \eqref{W} with $\Omega$ satisfying {\bf (H)}, with $\Omega(\theta)=\Omega(-\theta)$. Assume $a,b\in(0,\infty)$. Then
\begin{equation}\label{lem_ab2_1}
(|\bx|^{-s}\Omega(\theta)+A x_1^2 + B x_2^2)*\rho_{a,b} =  C_{\Omega,a,b},\quad \bx\in \supp\rho_{a,b}
\end{equation}
for some constant $C_{\Omega,a,b}$, with
\begin{equation}\label{lem_ab2_2}\begin{split}
A(a,b) = & a^{-2}\frac{C_2}{\bar{C}_1(s-1)}\int_{-\pi/2}^{\pi/2} \big((s-1)\cos^2\theta + \sin^2\theta\big)  \\
& \cdot(a^2\cos^2\theta + b^2\sin^2\theta)^{-s/2}\Omega\Big(\tan^{-1}\Big(\frac{b}{a}\tan\theta\Big)\Big)\rd{\theta},\\ 
B(a,b) = & b^{-2}\frac{C_2}{\bar{C}_1(s-1)}\int_{-\pi/2}^{\pi/2} \big(\cos^2\theta + (s-1)\sin^2\theta\big)   \\
& \cdot(a^2\cos^2\theta + b^2\sin^2\theta)^{-s/2}\Omega\Big(\tan^{-1}\Big(\frac{b}{a}\tan\theta\Big)\Big)\rd{\theta},\\ 
\end{split}\end{equation}
where $\bar{C}_1$ denotes the constant $C_1$ defined in \eqref{R1C1} with $s$ replaced by $s-1\in (0,1)$ in the case $1<s<2$. In the case $s=1$, $\bar{C}_1(s-1)$ is understood as the limit $\lim_{s\rightarrow 1^+}\bar{C}_1(s-1)=2/\pi$.
\end{lemma}

It can be shown that \eqref{lem_ab2_2} is equivalent to \eqref{lem_ab_2} (with $D=0$ due to the symmetry $\Omega(\theta)=\Omega(-\theta)$) for any $1\le s < 2$, with the understanding in Remark \ref{rem_ab}. The details of this equivalence lead to a cumbersome exercise with change of variables and special functions left to the interested reader. In fact, one can start from a change of variable $\tan\theta_1=\frac{b}{a}\tan\theta$ in \eqref{lem_ab2_2} to convert it into an integral against $\Omega(\theta_1)$. Then use $\tilde{\Omega} = \tau_{2-s}|\cos(\cdot)|^{-2+s}*\Omega$ to write \eqref{lem_ab_2} as a double integral and change the order of integrals. Then we have an outer integral with weight $\Omega(\theta)$, and the result can be obtained by calculating the inner integral in $\varphi$ explicitly. 

Since $\tilde{\Omega}\ge c > 0$ holds by \eqref{lem_FT_2} (for $1<s<2$) and \eqref{lem_FT_3} (for $s=1$), STEP 1 of the proof of Lemma \ref{lem_abf} works for $1\le s < 2$. In fact, most of the proof works in the same way, except for the justification of the limit of $f(b)$ as $b\to 0^+$. In the case $1\leq s<2$, both numerator and denominator of \eqref{fb} diverge to infinity, however the numerator is much smaller than the denominator, leading to the same result. We obtain the following consequence.
\begin{theorem}
Let $W$ be given by \eqref{W} with $1\le s < 2$ and $\Omega$ satisfying {\bf (H)} and $\Omega(\theta)=\Omega(-\theta)$. Then the unique energy minimizer is given by $\rho_{a,b}$ for some $a,b>0$.
\end{theorem}

\begin{proof}[Proof of Lemma \ref{lem_ab2}]
We first treat the case $a=b=1$ and $1<s<2$. We write the repulsive part of the potential as $W_{\textnormal{rep}}(\bx) = |\bx|^{-s}\Omega(\theta)$. We compute its contribution for the potential generated by the density $\rho_{1,1}=\rho_2$ as 
\begin{equation}\begin{split}
    (W_{\textnormal{rep}}*\rho_{1,1})(\bx) = & \int_{\mathbb{R}^2} W_{\textnormal{rep}}(\by) \rho_2(\bx-\by)\rd{\by} = \int_{-\pi}^\pi \int_0^\infty \rho_2(\bx-r \vec{e}_\theta) r^{1-s}\rd{r}\, \Omega(\theta)\rd{\theta}\\
    = & \int_{-\pi/2}^{\pi/2} \int_{\mathbb{R}} \rho_2(\bx-r \vec{e}_\theta) |r|^{1-s}\rd{r}\, \Omega(\theta)\rd{\theta}\,,
\end{split}\end{equation}
where the last equality uses $\Omega(\theta)=\Omega(\theta+\pi)$. 

Now we study the inner integral $\int_{\mathbb{R}} \rho_2(\bx-r \vec{e}_\theta) |r|^{1-s}\rd{r}$ for a fixed $\theta$. We write 
$
    \bx = u_1\vec{e}_\theta + u_2\vec{e}_\theta^\perp 
$, to obtain 
\begin{equation}
    \bx-r \vec{e}_\theta = (u_1-r)\vec{e}_\theta + u_2\vec{e}_\theta^\perp,\quad |\bx-r \vec{e}_\theta|^2 = (u_1-r)^2+u_2^2\,.
\end{equation}
Due to the definition of $\rho_2$ in \eqref{rho2}, if $|u_2|\ge R_2$, then $\int_{-\infty}^\infty \rho_2(\bx-r \vec{e}_\theta) |r|^{1-s}\rd{r}$ is clearly zero. If $|u_2|<R_2$, then
\begin{equation}\label{rho2rint}\begin{split}
\int_{\mathbb{R}} \rho_2(\bx-r \vec{e}_\theta) |r|^{1-s}\rd{r}  
& = C_2\int_{\mathbb{R}} (R_2^2-u_2^2-(u_1-r)^2)_+^{s/2} |r|^{1-s}\rd{r}  \\
& = \frac{C_2}{\bar{C}_1}\lambda^{-s-1}\int_{\mathbb{R}} \lambda\bar{\rho}_1(\lambda(u_1-r)) |r|^{1-s}\rd{r} \,. \\
\end{split}\end{equation}
where $\bar{\rho}_1$ denotes the 1D minimizer defined in \eqref{rho1D} with $s$ replaced by $s-1\in (0,1)$, similar for $\bar{C}_1, \bar{R}_1$, $\bar{V}_1$, and
\begin{equation}
    \lambda = \frac{\bar{R}_1}{\sqrt{R_2^2-u_2^2}}\,.
\end{equation}
The fact that $\bar{\rho}_1$ minimizes the energy associated to the potential $|x|^{1-s}+|x|^2$ implies that
\begin{equation}
\int_{\mathbb{R}} (|r|^{1-s}+|r|^{2})\bar{\rho}_1(u_1-r)\rd{r} \left\{\begin{split}
   & = \bar{V}_1,\quad u_1\in [-\bar{R}_1,\bar{R}_1] \\
   & > \bar{V}_1,\quad u_1\notin [-\bar{R}_1,\bar{R}_1] \\
\end{split}\right..
\end{equation}
Rescaling by $\lambda$, we get
\begin{equation}
\int_{\mathbb{R}} (|r|^{1-s}+\lambda^{1+s}|r|^{2})\lambda\bar{\rho}_1(\lambda(u_1-r))\rd{r} \left\{\begin{split}
   & = \tilde{V}_1\lambda^{s-1},\quad u_1\in [-\tilde{R}_1/\lambda,\bar{R}_1/\lambda] \\
   & > \bar{V}_1\lambda^{s-1},\quad u_1\notin [-\bar{R}_1/\lambda,\bar{R}_1/\lambda] \\
\end{split}\right..
\end{equation}
Notice that
\begin{equation}
    \int_{\mathbb{R}} |r|^2 \lambda\bar{\rho}_1(\lambda(u_1-r))\rd{r} = \int_{\mathbb{R}} |u_1-r|^2 \lambda\bar{\rho}_1(\lambda r)\rd{r} = \lambda^{-2}\int_{\mathbb{R}} |\lambda u_1-r|^2 \bar{\rho}_1( r)\rd{r} = u_1^2 + \lambda^{-2} \bar{C}_{1,*}
\end{equation}
where
\begin{equation}
    \bar{C}_{1,*}:=\int_{\mathbb{R}} |r|^2\bar{\rho}_1(r)\rd{r}
\end{equation}
is the second moment of $\bar{\rho}_1$. Therefore we see that 
\begin{equation}\label{V1}
\lambda^{-s-1}\int_{\mathbb{R}} |r|^{1-s}\lambda\bar{\rho}_1(\lambda(u_1-r))\rd{r} \left\{\begin{split}
   & = (\bar{V}_1-\bar{C}_{1,*})\lambda^{-2} -  u_1^2  ,\quad u_1\in [-\bar{R}_1/\lambda,\bar{R}_1/\lambda] \\
   & > (\bar{V}_1-\bar{C}_{1,*})\lambda^{-2} -  u_1^2,\quad u_1\notin [-\bar{R}_1/\lambda,\bar{R}_1/\lambda] \\
\end{split}\right..
\end{equation}
Notice also that 
\begin{equation}
    (\bar{V}_1-\bar{C}_{1,*})\lambda^{-2} -  u_1^2 = \frac{\bar{V}_1-\bar{C}_{1,*}}{\bar{R}_1^2}(R_2^2-u_2^2) -  u_1^2
\end{equation}
and a calculation using special functions shows that $\frac{\bar{V}_1-\bar{C}_{1,*}}{\bar{R}_1^2} = \frac{1}{s-1}$. Therefore, we get (for any $|u_2|<R_2$)
\begin{equation}
    \int_{\mathbb{R}} \rho_2(\bx-r \vec{e}_\theta) |r|^{1-s}\rd{r} + \frac{C_2}{\bar{C}_1}\Big(u_1^2 + \frac{1}{s-1}u_2^2\Big) 
    \left\{\begin{split}
   & = \frac{C_2R_2^2}{\bar{C}_1(s-1)}  ,\quad u_1\in \Big[-\sqrt{R_2^2-u_2^2},\sqrt{R_2^2-u_2^2}\Big] \\
   & > \frac{C_2R_2^2}{\bar{C}_1(s-1)},\quad u_1\notin \Big[-\sqrt{R_2^2-u_2^2},\sqrt{R_2^2-u_2^2}\Big] \\
\end{split}\right.,
\end{equation}
where the $u_1^2$ and $u_2^2$ terms have positive coefficients. Since the LHS is continuous in $(u_1,u_2)$ and increasing in $|u_2|$ for $|u_2|\ge R_2$, we see that (for any $\bx$)
\begin{equation}\label{rho2rint2}
    \int_{\mathbb{R}} \rho_2(\bx-r \vec{e}_\theta) |r|^{1-s}\rd{r} + \frac{C_2}{\bar{C}_1}\Big(u_1^2 + \frac{1}{s-1}u_2^2\Big) 
    \left\{\begin{split}
   & = \frac{C_2R_2^2}{\bar{C}_1(s-1)}  ,\quad \bx\in \supp\rho_2\\
   & > \frac{C_2R_2^2}{\bar{C}_1(s-1)},\quad \bx\notin \supp\rho_2 \\
\end{split}\right.\,.
\end{equation}
Then integrating in $\Omega(\theta)\rd{\theta}$ we get the conclusion, since $u_1^2 = (x_1\cos\theta+x_2\sin\theta)^2,\,u_2^2 = (-x_1\sin\theta+x_2\cos\theta)^2$ are quadratic functions in $x_1,x_2$, and the $x_1x_2$ terms are cancelled during integration due to the symmetry property $\Omega(\theta)=\Omega(-\theta)$.

For the case $a=b=1$ and $s=1$, one can conclude directly from \eqref{rho2rint} that
\begin{equation}\begin{split}
\int_{\mathbb{R}} \rho_2(\bx-r \vec{e}_\theta) |r|^{1-s}\rd{r}  
=  C_2\int_{\mathbb{R}} (R_2^2-u_2^2-(u_1-r)^2)_+^{1/2} \rd{r}  
=  \frac{\pi}{2}C_2(R_2^2-u_2^2)_+
\end{split}\end{equation}
which gives \eqref{rho2rint2} (with non-strict inequality) if one views $\bar{C}_1(s-1)$ as the limit $\lim_{s\rightarrow 1^+}\bar{C}_1(s-1)=2/\pi$ as in the statement of the lemma. Then the conclusion follows similarly as the previous case.

For the general case, notice that
\begin{equation}\begin{split}
    (W_{\textnormal{rep}}*\rho_{a,b})(\bx) = & \int_{\mathbb{R}^2} W_{\textnormal{rep}}(\bx-\by) \frac{1}{ab}\rho_2\Big(\frac{y_1}{a},\frac{y_2}{b}\Big)\rd{\by} 
    =  \int_{\mathbb{R}^2} W_{\textnormal{rep}}\big(x_1-a y_1,x_2-a y_2\big) \rho_2(\by)\rd{\by} \\
    = & \int_{\mathbb{R}^2} W_{\textnormal{rep}}\Big(a \Big(\frac{x_1}{a}-y_1\Big),b \Big(\frac{x_2}{b}-y_2\Big)\Big) \rho_2(\by)\rd{\by} 
    =  \big(W_{\textnormal{rep}}(a\cdot,b\cdot)*\rho_2\big)\Big(\frac{x_1}{a},\frac{x_2}{b}\Big) .
\end{split}\end{equation}
Notice that $W_{\textnormal{rep},a,b} := W_{\textnormal{rep}}(a\cdot,b\cdot)$ can be written as
\begin{equation}
    W_{\textnormal{rep},a,b}(\bx) = W_{\textnormal{rep}}(a |\bx| \cos\theta, b|\bx| \sin \theta) = |\bx|^{-s}\Omega_{a,b}(\theta) 
\end{equation}
with the angle function
\begin{equation}
    \Omega_{a,b}(\theta) = (a^2\cos^2\theta + b^2\sin^2\theta)^{-s/2}\Omega\Big(\tan^{-1}\Big(\frac{b}{a}\tan\theta\Big)\Big)
\end{equation}
satisfying {\bf (H)}. Therefore, applying the previous result to $W_{\textnormal{rep},a,b}*\rho_2$, we get the conclusion for general $a,b$.

\end{proof}

%%%%%%%%%%%%%%%%%%%%%%%%%%%%%%%%%%%%%%%%%%%%%%%

\section{Numerical examples}

In this section we give some numerical examples for the energy minimizers of $E_\alpha$, the interaction energy associated to the potential $W_\alpha$ given by \eqref{Walpha}. For this purpose, we consider the associated particle gradient flow
\begin{equation}\label{parflow}\begin{split}
    \dot{\bx}_j = -\frac{1}{N}\sum_{k=1,\,k\ne j}^N \nabla W(\bx_j-\bx_k),\quad j=1,\dots,N,
\end{split}\end{equation}
whose formal mean-field limit is the Wasserstein-2 gradient flow \eqref{flow}. The particle-level total energy
\begin{equation}
    E(\bx_1,\dots,\bx_N) = \frac{1}{2N^2}\sum_{j\ne k} W(\bx_j-\bx_k)
\end{equation}
corresponding to the energy functional \eqref{E}, is decreasing along the solution of \eqref{parflow}, and it is expected that a long time simulation of \eqref{parflow} will minimize the energy, at least locally. 

We take a few examples of $W_\alpha$ as the interaction potential and solve \eqref{parflow} numerically. We take the number of particles $N=1600$. We start from a random initial data (i.i.d. uniform distribution on $[-1/2,1/2]^2$) and solve \eqref{parflow} by the forward Euler method\footnote{For a more accurate simulation of \eqref{parflow}, higher order time integrators are preferred. However, this is not crucial in the current paper because our main interest is the energy minimizers instead of the gradient flow dynamics.}. The time steps are chosen adaptively, which guarantees the stability of the time integrator. The numerical simulation is terminated when the total energy stabilizes (up to a tolerance level $10^{-5}$ between adjacent time steps) or the total number of time steps is greater than $n_{\max}=20000$. It is expected that the final state of the particles is approximately an energy minimizer, at least locally.

In the rest of this section, we present the examples and the numerical results, as well as the explanation for them based on our theory from previous sections. We always take $s=0.4$ in the simulations unless mentioned otherwise. In the figures, the blue ellipses are the predicted shape of the unique minimizer in the LIC cases, based on Theorem \ref{thm_ell} and Lemma \ref{lem_ab}. The blue dashed lines indicate the height of $\rho_{\textnormal{1D}}$ as in \eqref{rho1D}. Different regimes for the parameter $\alpha$ are marked by different colors on the $\alpha$-axis.

\subsection{A classical example $\omega(\theta) = \cos^2\theta$}\label{sec_example1}
This example is the direct generalization of the anisotropic logarithmic potential in \cite{CMMRSV} to $0<s<1$. The Fourier transform of $W_\alpha$, in its angle variable, is given by
\begin{equation}
    \tilde{\Omega}_\alpha(\varphi) = c_s \Big(1+\alpha\frac{1-(2-s)\cos^2\varphi}{s}\Big).
\end{equation}
Therefore, $W_\alpha$ is LIC if and only if 
\begin{equation}
    \alpha \le \alpha_L = \frac{s}{1-s}
\end{equation}
where $\alpha_L$ is as in \eqref{alphaL}.
At $\alpha = \alpha_L$, we have $\tilde{\Omega}_{\alpha_L}(\varphi) = C\sin^2\varphi$. In this case, the function $f(b)$ defined in \eqref{fb}, taking $b\rightarrow\infty$, can be computed by
\begin{equation}\begin{split}
    f(\infty) = & \frac{\int_0^\pi |\sin\varphi|^{-s}\cos^2\varphi\rd{\varphi}}{\int_0^\pi |\sin\varphi|^{-s}\sin^2\varphi\rd{\varphi}} 
    =  \frac{\gamma_{-s}}{\gamma_{2-s}} - 1 
    =  \frac{-s+2}{-s+1} - 1 = \frac{1}{1-s}>1.
\end{split}\end{equation}
Therefore, item 1 of Theorem \ref{thm_ell} holds, i.e., the support of the unique minimizer is a proper ellipse. This implies that $\rho_{\textnormal{1D}}$ is not a minimizer for $W_{\alpha_L}$, and thus not a minimizer for $W_{\alpha}$ if $\alpha$ is slightly larger than $\alpha_L$. Therefore the minimizers behave as the following (see Figure \ref{fig:example1}):
\begin{itemize}
    \item For $0\le \alpha \le \alpha_L$ (marked blue on the $\alpha$-axis), the minimizer is some $\rho_{a,b}$ supported on a proper ellipse.
    \item For $\alpha_L<\alpha<\alpha_*$ (with $\alpha_L<\alpha_*$, marked green on the $\alpha$-axis), any minimizer is not an ellipse or $\rho_{\textnormal{1D}}$, and any of its superlevel set has no interior point due to Proposition \ref{thm_concave}.
    \item For $\alpha\ge \alpha_*$ (marked red on the $\alpha$-axis), $\rho_{\textnormal{1D}}$ is the unique minimizer due to Theorem \ref{thm_coer}.
\end{itemize}

\begin{remark}
It is worth noticing that the anisotropic logarithmic potential considered in \cite{CMMRSV} (i.e., the $s=0$ limiting case of the current example) does not admit the `green' phase as in Figure \ref{fig:example1}. In other words, as $\alpha$ increases, the minimizer starts as ellipses and transits into $\rho_{\textnormal{1D}}$ at $\alpha_L$ directly. It satisfies the conditions of Lemma \ref{lem_coer} for $\Omega_*$ (under a suitable generalization to logarithmic potentials). 

In the above example, we see that the natural generalization of the anisotropic logarithmic potential in \cite{CMMRSV} to $0<s<1$ admits the `green' phase, and thus its corresponding $1+\alpha_L \omega$ does not satisfy the requirements for $\Omega_*$ in Lemma \ref{lem_coer}. Therefore, even if one is only interested in the case $\omega(\theta)=\cos^2\theta$ with $0<s<1$, it does not seem possible to prove Theorem \ref{thm_coer} for it without considering a more general class of potentials. A bigger pool of potentials enables one to find the correct $\Omega_*$ to compare with, as in the proof of Theorem \ref{thm_coer}.
\end{remark}

\begin{figure}
    \centering
    \includegraphics[width=0.8\textwidth]{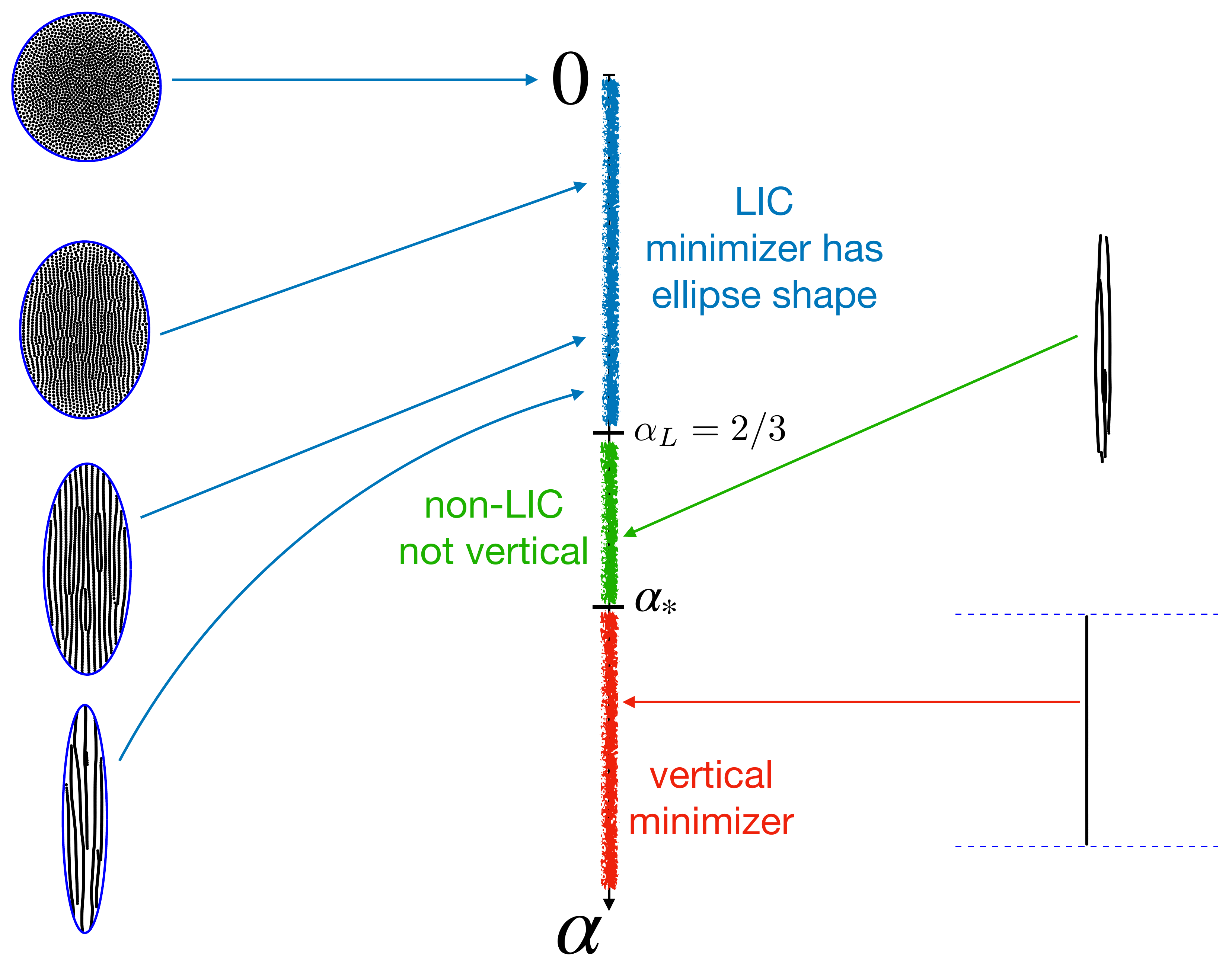}
    \caption{Numerical results for the energy minimizers, for $W_\alpha$ with $s=0.4$ and $\omega(\theta) = \cos^2\theta$. The $\alpha$ values are $0,0.2,0.5,0.65,0.675,0.7$.}
    \label{fig:example1}
\end{figure}

\subsection{An example with degeneracy near $\pi/2$: $\omega(\theta)=\cos^4\theta$}\label{sec_example2}
This example is designed to have a high order degeneracy near $\theta=\pi/2$. In fact, the condition \eqref{thm_coer_1} is not satisfied by this $\omega$, and Theorems \ref{thm_degen} and \ref{thm_width} apply with $\kappa=4$.

The Fourier transform of $W_\alpha$, in its angle variable, is given by
\begin{equation}
    \tilde{\Omega}_\alpha(\varphi) = c_s \Big(1+\alpha\frac{3-6(2-s)\cos^2\varphi + (4-s)(2-s)\cos^4\varphi}{s(s+2)}\Big)
\end{equation}
The last fraction $\tilde{\omega}(\varphi)$ has minimum $-\frac{6(1-s)}{s(s+2)(4-s)}$, achieved at $\cos^2\varphi=\frac{3}{4-s}\in (0,1)$. Therefore $W_\alpha$ is LIC if and only if 
\begin{equation}
    \alpha \le \alpha_L = \frac{s(s+2)(4-s)}{6(1-s)}.
\end{equation} 
Notice that $\tilde{\omega}(\varphi)$ is not minimized at $\varphi=0$. This means that for $\alpha$ in the range
\begin{equation}
    \alpha_L < \alpha < \alpha_{L,0} = -\frac{1}{\tilde{\omega}(0)} = \frac{s(s+2)}{(1-s)(1+s)},
\end{equation}
(where $\alpha_{L,0}$ is as in \eqref{alphaL0}) we have (1) LIC fails and any superlevel set of any $d_\infty$-local minimizer has no interior point as in Proposition \ref{thm_concave}; (2) no `vertical segment' as in Proposition \ref{prop_expan} is allowed in any $d_\infty$-local minimizer. In particular, for $\alpha$ slightly larger than $\alpha_L$, this forces the zigzag behavior: the support of a $d_\infty$-local minimizer is expected to be a union of segments along the directions $\vec{e}_\varphi^\perp$ with $|\varphi|$ near $\cos^{-1}\sqrt{\frac{3}{4-s}}$. 

Therefore the minimizers behave as the following (see Figure \ref{fig:example2}):
\begin{itemize}
    \item For $0\le \alpha \le \alpha_L$ (marked blue on the $\alpha$-axis), the minimizer is some $\rho_{a,b}$ supported on a proper ellipse.
    \item For $\alpha_L<\alpha<\alpha_{L,0}$ (marked green dashed line on the $\alpha$-axis), any minimizer is not an ellipse or $\rho_{\textnormal{1D}}$, and any of its superlevel set has no interior point. No `vertical segments' are allowed. Also, for $\alpha$ close to $\alpha_L$, the slopes of the tilted segments agree well with the predicted angle $\cos^{-1}\sqrt{\frac{3}{4-s}}$.
    \item For $\alpha\ge\alpha_{L,0}$ (marked green on the $\alpha$-axis), similar as the previous case, but we cannot exclude the possibility of `vertical segments'. Minimizers get closer to being vertical as $\alpha$ increases, but never becomes $\rho_{\textnormal{1D}}$ due to Theorem \ref{thm_degen}.
\end{itemize}

\begin{remark}\label{rem_par}
Although our theory guarantees that the zigzag behavior cannot appear for $\alpha\le \alpha_L$ in the continuum model, one can indeed observe such phenomenon in the particle simulation when $\alpha$ is slightly smaller than $\alpha_L$ (for example, the $\alpha=0.95$ case in Figure \ref{fig:example2}). From this example, one can see that the particle model may behave differently from the continuum model when the LIC condition barely fails. It would be interesting to apply numerical methods for the continuum model and compare with the numerical results for the particle model.
\end{remark}

\begin{figure}
    \centering
    \includegraphics[width=0.7\textwidth]{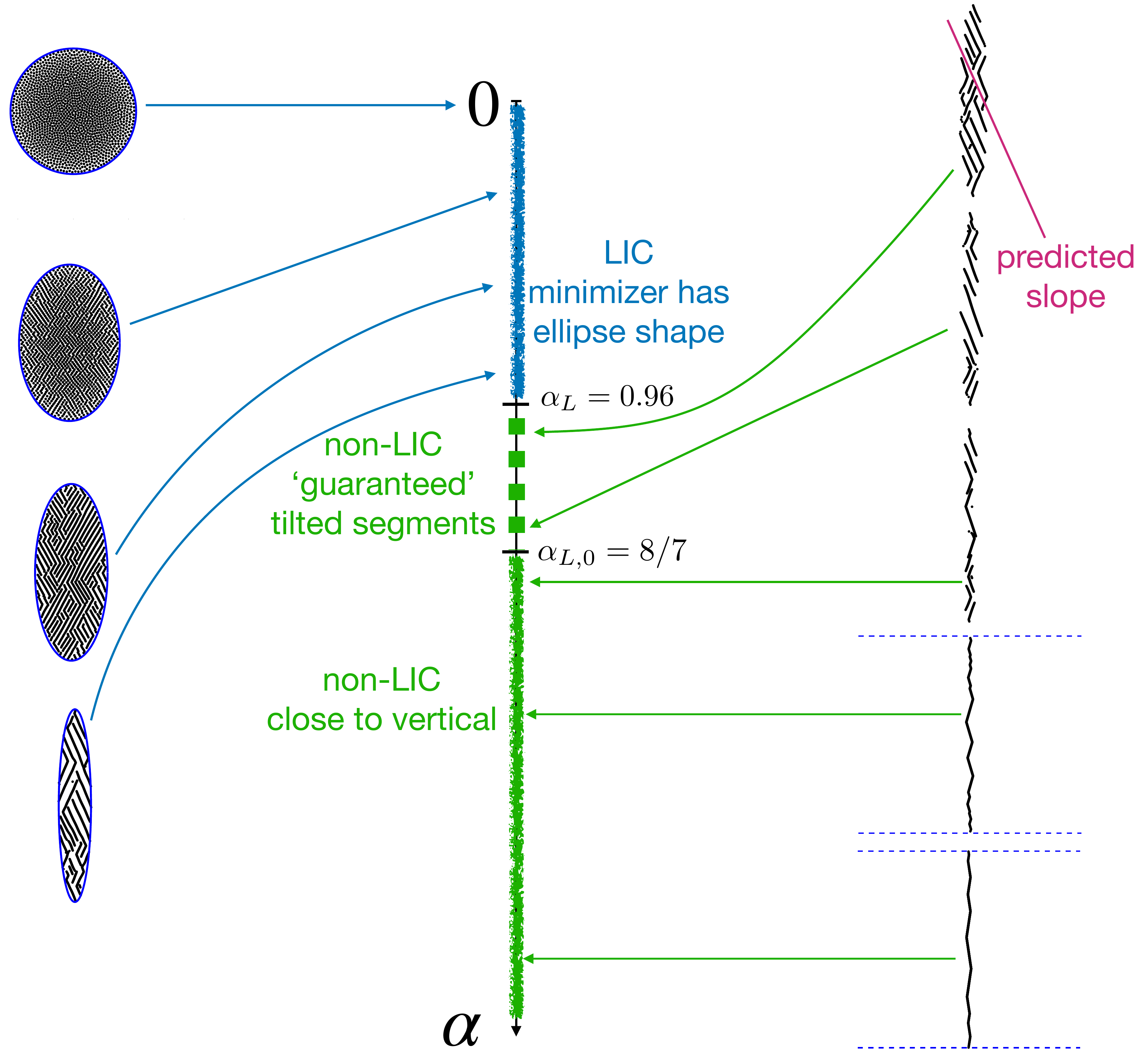}
    \caption{Numerical results for the energy minimizers, for $W_\alpha$ with $s=0.4$ and $\omega(\theta) = \cos^4\theta$. The $\alpha$ values are $0,0.3,0.6,0.95,1,1.1,1.15,1.5,5$.}
    \label{fig:example2}
\end{figure}

\subsection{A generic example $\omega(\theta) = \cos^4 \theta + 0.1 \cos^2\theta$}\label{sec_example0}

The Fourier transform of $W_\alpha$, in its angle variable, is given by (with $s=0.4$ substituted)
\begin{equation}
    \tilde{\Omega}_\alpha(\varphi) = c_s \Big(1+\alpha\frac{3.24-9.984\cos^2\varphi + 5.76\cos^4\varphi}{0.96}\Big)
\end{equation}

Denote the above fraction as $\tilde{\omega}(\varphi)$. The minimum of $\tilde{\omega}(\varphi)$ is achieved at $\cos^2\varphi = \frac{9.984}{2\cdot 5.76}$, and the minimal value is (approximately) $-1.1317$. Therefore $\alpha_L \approx 1/1.1317 \approx 0.8837$. Also, $\tilde{\omega}(0) = -1.025$. Therefore $\alpha_{L,0} = 1/1.025 \approx 0.9756$.

In this case, we have $\alpha_L<\alpha_{L,0}$, and there exists $\alpha_*$ in the sense of Theorem \ref{thm_coer}. The behavior of minimizers can be analyzed similar to the previous two examples using Theorems \ref{thm_ell} and \ref{thm_coer}, and Propositions \ref{thm_concave} and \ref{prop_expan}. Therefore the minimizers behave as the following (see Figure \ref{fig:example0}):
\begin{itemize}
    \item For $0\le \alpha \le \alpha_L$ (marked blue on the $\alpha$-axis), the minimizer is some $\rho_{a,b}$ supported on a proper ellipse.
    \item For $\alpha_L<\alpha<\alpha_{L,0}$ (marked green dashed line on the $\alpha$-axis), any minimizer is not an ellipse or $\rho_{\textnormal{1D}}$, and any of its superlevel set has no interior point. No `vertical segments' are allowed. Also, for $\alpha$ close to $\alpha_L$, the slopes of the tilted segments agree well with the predicted angle $\cos^{-1}\sqrt{\frac{9.984}{2\cdot 5.76}}$.
    \item For $\alpha_{L,0}\le\alpha<\alpha_*$ (marked green on the $\alpha$-axis), similar as the previous case, but we cannot exclude the possibility of `vertical segments'. Minimizers is not $\rho_{\textnormal{1D}}$ and exhibit zigzag behavior.
    \item For $\alpha\ge \alpha_*$ (marked red on the $\alpha$-axis), $\rho_{\textnormal{1D}}$ is the unique minimizer.
\end{itemize}

\subsection{An example with two preferred directions: $\omega(\theta)=\cos^2\theta\sin^2\theta$}

In this example, $\omega(\theta)$ achieves minimal value at two different angles $\theta=0$ and $\theta=\pi/2$, i.e., both vertical and horizontal directions are preferred when $\alpha$ is large. 

The Fourier transform of $W_\alpha$, in its angle variable, is given by 
\begin{equation}\begin{split}
    \tilde{\Omega}_\alpha(\varphi) 
     = & c_s \Big(1+\alpha\frac{-(1-s)+(4-s)(2-s)\cos^2\varphi - (4-s)(2-s)\cos^4\varphi}{s(s+2)}\Big) .
\end{split}\end{equation}
The last fraction $\tilde{\omega}(\varphi)$ has minimum $-\frac{1-s}{s(s+2)}$, achieved at $\cos^2\varphi=0,1$. Therefore $W_\alpha$ is LIC if and only if
\begin{equation}
    \alpha \le \alpha_L = \frac{s(s+2)}{1-s}.
\end{equation}

For $\alpha \le \alpha_L$, Theorem \ref{thm_ell} shows that the unique energy minimizer is some $\rho_{a,b}$. Since $\omega$ is symmetric with respect to exchanging $x_1$ and $x_2$, the same is true for the unique energy minimizer. This implies $a=b$ (i.e., the support of the minimizer is a ball), and their value can be determined by setting $A=B=1$ in \eqref{lem_ab_2}. 

Notice that Theorem \ref{thm_coer} does not apply because $\pi/2$ is not the unique minimum point of $\omega$. In fact, since $\omega$ is symmetric with respect to exchanging $x_1$ and $x_2$, suppose $\rho_{\textnormal{1D}}=\rho_1(x_2)\delta(x_1)$ is a minimizer of $E_\alpha$, then $\rho_1(x_1)\delta(x_2)$ is also a minimizer. For large $\alpha$, we do not know whether $\rho_{\textnormal{1D}}$ is a minimizer. Theorem \ref{thm_width} does not apply either because $\omega(0)=0$ makes \eqref{thm_width_1} false. Therefore the minimizers behave as the following (see Figure \ref{fig:example3}):
\begin{itemize}
    \item For $0\le \alpha \le \alpha_L$ (marked blue on the $\alpha$-axis), the minimizer is some $\rho_{a,a}$ supported in a ball.
    \item For $\alpha>\alpha_L$ (marked red on the $\alpha$-axis), any superlevel set of any minimizer has no interior point, but the precise shapes of minimizers are unknown. Numerical results show the formation of some fractal structure, in which particles tend to align either vertically or horizontally at different spatial scales. However, it is likely that such configuration is merely a local minimizer because $\rho_{\textnormal{1D}}$ clearly has smaller energy than it, at least for large $\alpha$.
\end{itemize}
The behavior of the minimizers for large $\alpha$ remains open in this case. Also, similar comments regarding the difference between particle and continuum models as in Remark \ref{rem_par} apply to this example.

\begin{figure}
    \centering
    \includegraphics[width=0.8\textwidth]{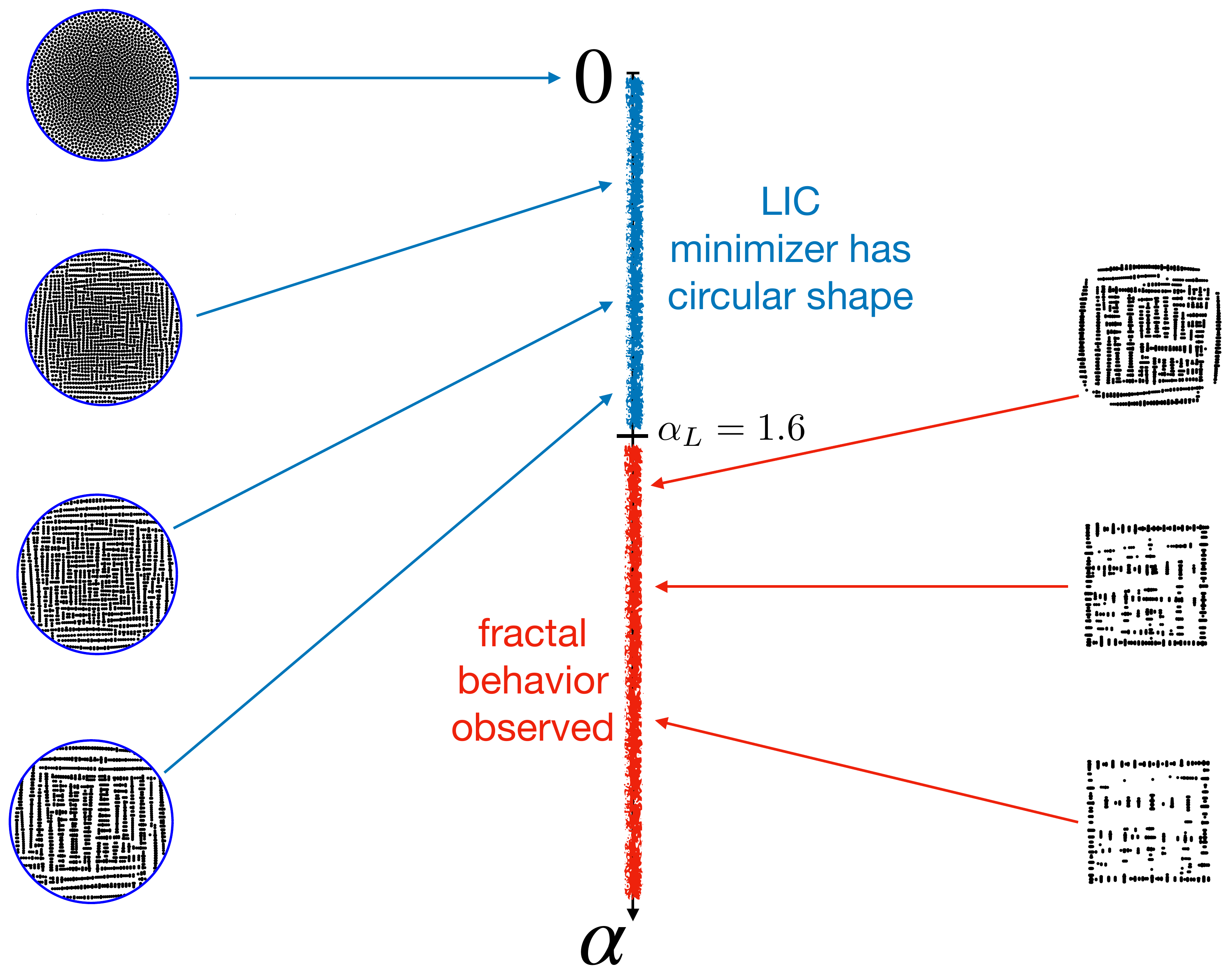}
    \caption{Numerical results for the energy minimizers, for $W_\alpha$ with $s=0.4$ and $\omega(\theta) = \cos^2\theta\sin^2\theta$. The $\alpha$ values are $0,0.5,1.1,1.5,1.7,3,5$.}
    \label{fig:example3}
\end{figure}

\subsection{An example with $1<s<2$}

In this example we take $s=1.4$ and $\omega(\theta)=\cos^2\theta$. The results in Section \ref{sec_s12} show that $W_\alpha$ is LIC for any $\alpha\ge 0$, and the unique energy minimizer is always some $\rho_{a,b}$. The numerical results shown in Figure \ref{fig:example4} verifies this phenomenon.

It is worth noticing that the gap between the cluster of particles and the predicted ellipse shape is larger than the previous examples with $s=0.4$. This is a consequence of the fact that $\rho_2(\bx) = C_2(R_2^2-|\bx|^2)^{s/2}$ has smaller values near the boundary of the support if $s$ is larger.

\begin{figure}
    \centering
    \includegraphics[width=0.7\textwidth]{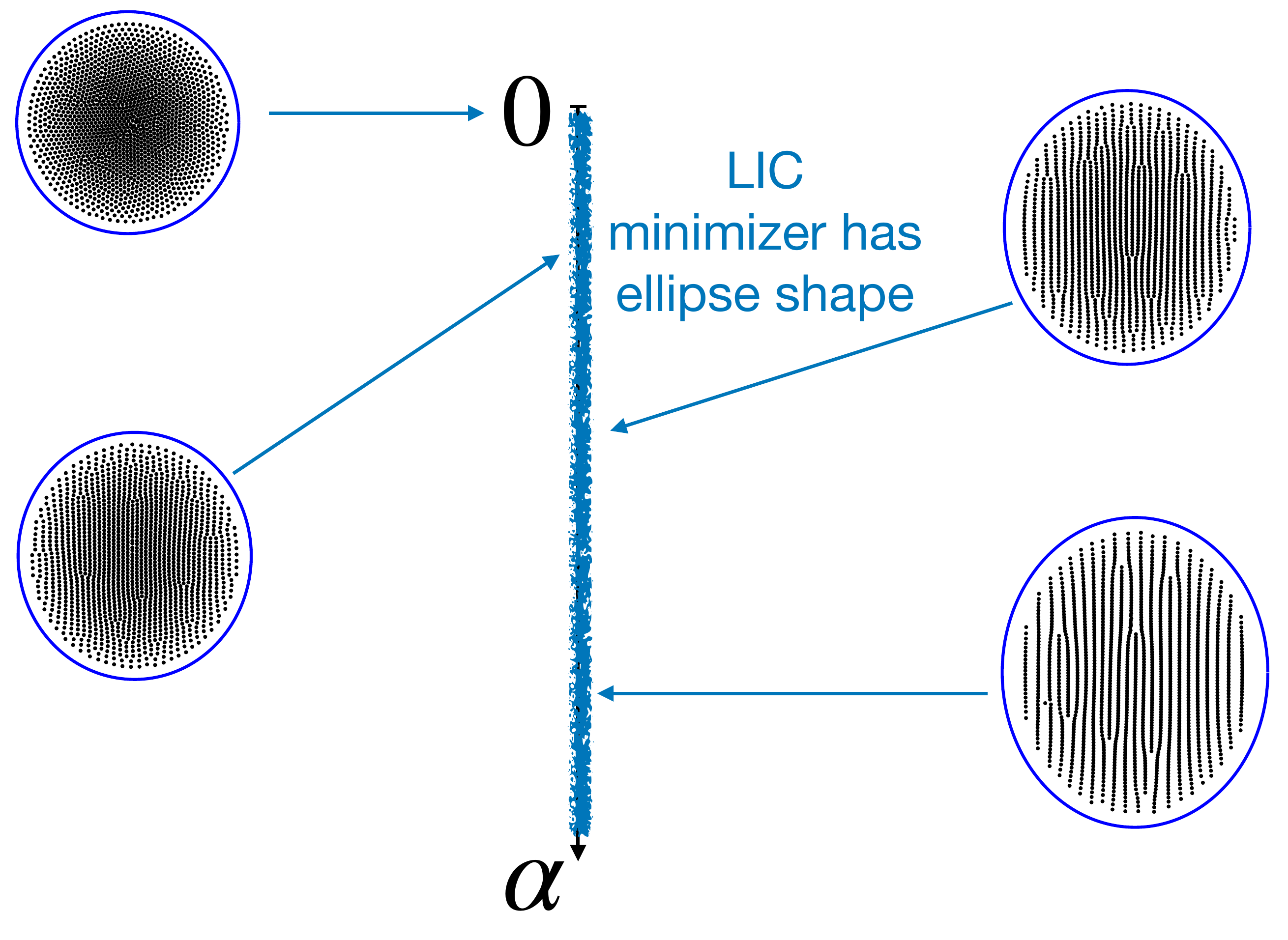}
    \caption{Numerical results for the energy minimizers, for $W_\alpha$ with $s=1.4$ and $\omega(\theta) = \cos^2\theta$. The $\alpha$ values are $0,0.6,1.5,3$.}
    \label{fig:example4}
\end{figure}

\subsection{Illustration of the asymptotic limit $s\rightarrow 0^+$}

Finally we give an example illustrating the asymptotic limit $s\rightarrow 0^+$ towards the logarithmic potential in Lemma \ref{lem_FTlog}. We fix the choice $\Omega(\theta)=1+0.5\cos^2\theta$ and take $W_{(s)}$ and $W_{\log}$ as in \eqref{Ws} and \eqref{Wlog} respectively. It is easily verified that $W_{(s)},\,0<s<1$ and its asymptotic limit $W_{\log}$ all have the LIC property, and the unique energy minimizer is some $\rho_{a,b}$ (for $W_{(s)}$) or $\frac{1}{|\cB(0;a,b)|}\chi_{\cB(0;a,b,\eta)}$ (for $W_{\log}$). See Figure \ref{fig:example5} for the numerical results. Also, as $s\rightarrow 0^+$, we observe that the minimizer for $W_{(s)}$ converges to that for $W_{\log}$, which can be viewed as a consequence of the convergence of \eqref{abR2_1} to \eqref{lem_ablog_1}.

\begin{figure}
    \centering
    \includegraphics[width=\textwidth]{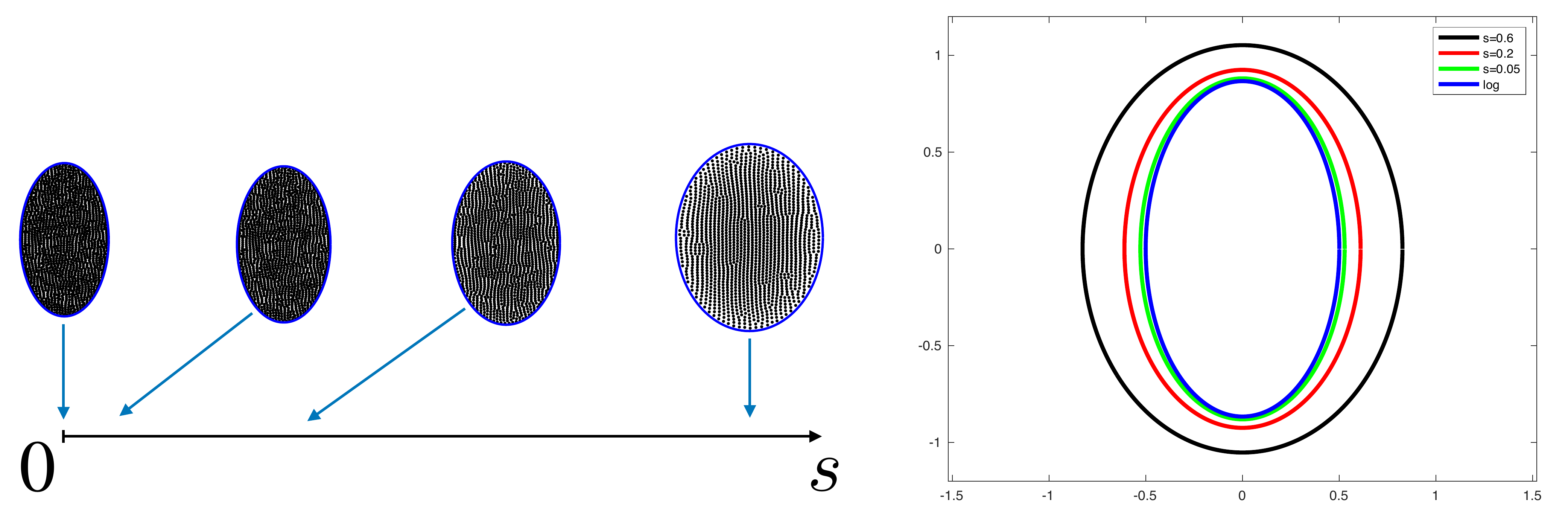}
    \caption{Numerical results for the energy minimizers, for $W_{(s)}$ with $\Omega(\theta)=1+0.5\cos^2\theta$ and various $s$. The $s$ values are $0$ (logarithmic), $0.05,0.2,0.6$. On the right, the supports of the minimizers for different $s$ are compared.}
    \label{fig:example5}
\end{figure}

%%%%%%%%%%%%%%%%%%%%%%%%
\newpage
\appendix

\section{List of notations and integral formulas}
\label{app:constants}

Unit vectors in $\mathbb{R}^2$ are denoted by
\begin{equation}\label{vece}
    \vec{e}_\varphi = (\cos\varphi,\sin\varphi)^\top,\quad \vec{e}_\varphi^\perp = (-\sin\varphi,\cos\varphi)^\top.
\end{equation}
For $0<s<2$, the constants $C_2,R_2$ in \eqref{rho2} are given by
\begin{equation}\label{R2C2}
    R_2 = \Big(\frac{8\sin\frac{s\pi}{2}}{s^2(2+s)\pi}\Big)^{\frac{1}{-s-2}},\quad C_2 = \frac{4\sin\frac{s\pi}{2}}{s^2\pi^2}.
\end{equation}
For $0<s<1$, the constants $C_1,R_1$ in \eqref{rho1D} are given by
\begin{equation}\label{R1C1}
    R_1 = \Big(\frac{2\cos\frac{s\pi}{2}}{s(s+1)\pi} \beta\Big(\frac{1}{2},\frac{3+s}{2}\Big)\Big)^{\frac{1}{-s-2}},\quad C_1=\frac{2\cos\frac{s\pi}{2}}{s(s+1)\pi}
\end{equation}
where $\beta$ denotes the Beta function. $R_2,C_2,R_1,C_1$ are obtained from \cite{CH} up to rescaling. 

For $0<s<1$, $\rho_1$ in \eqref{rho1D} is the unique energy minimizer for the 1D interaction potential $|x|^{-s}+|x|^2$, and thus $(|\cdot|^{-s}+|\cdot|^2)*\rho_1$ is constant on $\supp\rho_1 = [-R_1,R_1]$. This constant is given by
\begin{equation}\label{calcV1}
    V_1 = 2\int_0^{R_1}(x^{-s}+x^2)\rho_1(x)\rd{x} = R_1^2 \Big(\frac{1}{s}+\frac{1}{4+s}\Big).
\end{equation}

We need several constants in the Fourier expansion of the potentials in Sections \ref{sec_LIC} and \ref{sec_complex}, some details are omitted below for simplicity.
We need first the following integral
\begin{equation}\label{calc1}
    \gamma_s := \int_{-\pi}^{\pi} |\cos\theta|^s\rd{\theta} = \frac{2\sqrt{\pi}\Gamma((s+1)/2)}{\Gamma((s+2)/2)},\quad s>-1,\quad \frac{\gamma_{s+2}}{\gamma_s} = \frac{s+1}{s+2},
\end{equation}
which can be proved using Beta functions,  $\gamma_s$ can be extended to $s\in (-2,-1)$  naturally and takes negative values there.

The Fourier transform of power functions on $\mathbb{R}^2$ is given by
\begin{equation}\label{calc2}
    \cF[|\bx|^{-s}] = c_s |\xi|^{-2+s},\quad 0<s<2,\quad c_s = \pi^{s-1}\frac{\Gamma((2-s)/2)}{\Gamma(s/2)}.
\end{equation}
In the sense of improper integral, we can obtain
\begin{equation}\label{calc3}
    \int_0^\infty r^{s}\cos r   \rd{r} = -\Gamma(1+s)\sin\frac{s\pi}{2},\quad -1<s<0
\end{equation}
and
\begin{equation}\label{calc3s}
    \int_0^\infty r^{s}\sin r   \rd{r} = \Gamma(1+s)\cos\frac{s\pi}{2},\quad -1<s<0.
\end{equation}
These two formulas can be proved by contour integrals whose details are omitted. Finally, we obtain 
\begin{equation}\label{calc4}
    \tau_s = (2\pi)^{-s}\Gamma(s)\cos\frac{s\pi}{2} = \frac{c_{2-s}}{\gamma_{-s}},\quad 0<s<2 .
\end{equation}
The last equality can be proved using the functional relations for the Gamma function (with the understanding that $\gamma_{-1}=\infty$). Notice that $\tau_1=0$, and $\tau_s$ is negative for $1<s<2$ and positive for $0<s<1$.
The final integral relation we need is
\begin{equation}\label{calc5}
\int_{\mathbb{R}} \Big(|x|^{-s}-(1+x^2)^{-s/2}   \Big) \rd{x} = -\frac{\gamma_{s-2}}{2} ,\quad 0<s<1.
\end{equation}
To prove this formula for fixed $0<s<1$, we consider
\begin{equation}
    I(t) = \int_0^\infty \Big(x^{-s}-(1+x^2)^{-s/2}   \Big)(1+x^2)^{-t/2} \rd{x}
\end{equation}
which is well-defined and analytic on $\{\text{Re}\,t > -1\}$. If $t\in (1-s,\infty)$, then one can separate the integrand as
\begin{equation}\begin{split}
    I(t) = & \int_0^\infty x^{-s}(1+x^2)^{-t/2} \rd{x} - \int_0^\infty (1+x^2)^{-(s+t)/2} \rd{x} \\ = & \frac{1}{2}\left(\beta\Big(\frac{-s+1}{2},\frac{s+t-1}{2}\Big)-\beta\Big(\frac{1}{2},\frac{s+t-1}{2}\Big)\right). \\
\end{split}\end{equation}
Since both sides are meromorphic functions of $t$ on $\{\text{Re}\,t > -1\}$ and agrees on $ (1-s,\infty)$, they agree on the whole domain $\{\text{Re}\,t > -1\}$. In particular, evaluating at $t=0$ gives \eqref{calc5}.

\section{Estimate on the support of minimizer}\label{app_exist}

\begin{lemma}\label{lem_R}
Assume $0<s<1$ and $W_\alpha$ given by \eqref{Walpha} with $\omega$ satisfying {\bf (h)}. Then there exists a minimizer of $E_\alpha$ in $\mathcal{P}(\mathbb{R}^2)$. Let $\rho$ be any minimizer of $E_\alpha$ with zero center of mass. Then $\supp\rho\subset \cB(0;R)$ with $R$ depending on $s$ but independent of $\alpha$ and $\omega$.
\end{lemma}

\begin{proof}
The proof mainly follows \cite{CCP15}. The existence of minimizer would follow from the support estimate. In fact, one may consider a minimizer $\rho_R$ on $\mathcal{P}(\bar{B}(0;R))$ (which clearly exists by the weak lower semicontinuity of $E_\alpha$), use this estimate for $\rho_R$, and pass to the limit $R\rightarrow\infty$. 

Let $\rho$ be a minimizer of $E_\alpha$ with zero center of mass. We first show the analogy of \cite[Lemma 2.6]{CCP15}: there exist $r>0$, independent of $\alpha$ and $\omega$, such that for any $\bx_0\in\supp\rho$, 
\begin{equation}\label{Bxr}
\int_{\cB(\bx_0;r)} \rho(\bx)\rd{\bx} > \frac{1}{2}.
\end{equation}
We first show \eqref{Bxr} for $\bx_0$ $\rho$-almost everywhere. Lemma \ref{lem_EL} gives that 
\begin{equation}\label{Bxr1}
(W_\alpha * \rho)(\bx_0) = \frac{1}{2}E_\alpha[\rho] 
\end{equation}
for $\bx_0$ $\rho$-almost everywhere. For the RHS, we have the estimate
\begin{equation}\label{E0rho}
    E_\alpha[\rho] \le E_\alpha[\rho_{\textnormal{1D}}] = E_0[\rho_{\textnormal{1D}}]
\end{equation}
by the minimizing property of $\rho$. On the other hand, if \eqref{Bxr} does not hold, then unit mass implies $\int_{\bx\notin \cB(\bx_0;r)}\rho(\bx)\rd{\bx} \ge 1/2$, and thus
\begin{equation}\begin{split}
(W_\alpha * \rho)(\bx_0) = & \int_{\bx\in \cB(\bx_0;r)} W_\alpha(\bx_0-\bx)\rho(\bx)\rd{\bx} + \int_{\bx\notin \cB(\bx_0;r)} W_\alpha(\bx_0-\bx)\rho(\bx)\rd{\bx} 
\ge  \frac{r^2}{2}
\end{split}\end{equation}
since $W_\alpha > 0$ and $W_\alpha(\bx) > r^2/2$ if $|\bx| \ge r$. This contradicts \eqref{Bxr1} if we choose $r = \sqrt{E_0[\rho_{\textnormal{1D}}]}$. Notice that the last quantity is independent of $\alpha$ and $\omega$. Extension to all $\bx_0\in\supp\rho$ can be done similarly as in \cite{CCP15}, by finding a point near any given $\bx_0\in \supp\rho$ to apply the previous estimate.

Then we conclude that $\sup_{\bx,\by\in\supp\rho}|\bx-\by| \le 3r$ because otherwise we can find $\bx_0$ and $\bx_1$ in $\supp\rho$ with $\cB(\bx_0;r)\cap \cB(\bx_1;r)=\emptyset$, contradicting \eqref{Bxr}. This implies $\supp\rho\subset \cB(0;3r)$ because $\rho$ has zero center of mass.

\end{proof}

\begin{remark}
For the case $1\le s < 2$, one can show the existence of compactly supported minimizers in the same way, but we do not expect to have a uniform-in-$\alpha$ estimate on the size of support because the $E_\alpha$ energy of any measure supported on a vertical line is infinity (and thus we do not have the uniform-in-$\alpha$ estimate \eqref{E0rho}).
\end{remark}

\section*{Acknowledgements}
JAC and RS were supported by the Advanced Grant Nonlocal-CPD (Nonlocal PDEs for Complex Particle Dynamics: Phase Transitions, Patterns and Synchronization) of the European Research Council Executive Agency (ERC) under the European Union's Horizon 2020 research and innovation programme (grant agreement No. 883363). JAC was also partially supported by the EPSRC grant number EP/T022132/1 and EP/V051121/1.

% Bibliography
\bibliographystyle{abbrv}
\bibliography{biblio}

\end{document}